\documentclass[12pt]{article}
\usepackage{philstyle}

\def\cF{{\cal F}}
\def\cFs{{\cal F}^s}
\def\cFu{{\cal F}^u}
\def\cP{{\cal P}}
\def\cI{{\cal I}}
\def\cB{{\cal B}}
\def\cG{{\cal TAF}}

\def\QED{\ $\blacksquare$\smallskip}
\def\T{{\mathbb T}}
\def\bbone{\mathbbmss{1}}

\def\de#1{\textit{#1}}

\def\complexes{{\mathbb C}}

\def\reals{{\mathbb R}}
\def\integers{{\mathbb Z}}
\def\naturals{{\mathbb N}}
\def\raw{\rightarrow}

\def\us{{\underline s}}
\def\ut{{\underline t}}
\def\ux{{\underline x}}

\def\I{^{-1}}
\def\bv{{\mathbf v}}

\def\vn{{\vec{n}}}
\def\vm{\vec{m}}
\def\vv{\vec{v}}

\def\vr{\vec{r}}
\def\vell{\vec{\ell}}
\def\tf{\tilde{f}}

\def\tM{\tilde{M}}
\def\tphi{\tilde{\phi}}
\def\tpsi{\tilde{\psi}}

\def\homeo{homeomorphism}
\def\homeos{homeomorphisms}
\def\diffeo{diffeomorphism}

\def\Z{\integers}
\def\C{\complexes}

\def\N{\naturals}
\def\R{\reals}
\def\cP{\cal{P}}
\DeclareMathOperator{\id}{id}
\DeclareMathOperator{\im}{im}

\def\F{{\mathbb F}}

\def\tcF{\tilde{\cal F}}
\def\tF{\tilde{\cal F}}
\def\pA{pseudo-Anosov}

\def\tX{\tilde{X}}
\def\tW{\tilde{W}}

\def\bv{{\mathbf v}}

\def\tphi{{\tilde{\phi}}}
\def\tcF{\tilde{\cal F}}
\def\tgamma{{\tilde{\gamma}}}
\def\tchi{{\tilde{\chi}}}

\def\tH{{\tilde{H}}}
\def\tx{{\tilde{x}}}
\def\tp{{\tilde{p}}}
\def\ty{\tilde{y}}
\def\tz{{\tilde{z}}}

\def\osum{\oplus}

\def\ie{i.e.}
\def\cf{\textit{cf.}}
\def\tomega{{\tilde{\omega}}}

\def\hJ{{\hat{J}}}
\def\cP{{\cal P}}
\def\cA{{\cal A}}
\def\cE{{\cal E}}
\def\cK{{\cal STAF}}

\def\cC{{\cal C}}
\def\cR{{\cal{TC}}(\cF^s)}
\def\cL{{\cal L }}
\def\tL{\tilde{L}}
\def\allowsum{\sum_{j\raw k }}

\def\tGamma{\tilde{\Gamma}}
\def\talpha{\tilde{\alpha}}

\def\hbeta{\hat{\beta}}

\def\td{\tilde{d}}
\def\nun#1{\vert #1 \vert_\nu}

\def\cns{$(C,\nu)$-steep}
\def\stx{s_{\tx}}
\def\cmap{$c$-map}
\def\cmaps{$c$-maps}

\def\dvn{\delta_{\vn}}

\def\hM{\hat{M}}
\def\hx{\hat{x}}
\def\hy{\hat{y}}
\def\hz{\hat{z}}
\def\arc{arc}
\def\arcs{arcs}
\def\vK{\mathbf{K}}
\def\vKn{\vK^{(n)}}
\def\vKnj{(\vK^{(n)})_j}
\def\bM{\overline{M}}
\def\bcF{\overline{\cF}}
\def\PF{Peron-Fr\"obenius}
\def\ilimit{\underleftarrow{\lim}}
\def\hK{\hat{K}}
\DeclareMathOperator{\var}{var}
\DeclareMathOperator{\trace}{trace}
\DeclareMathOperator{\Fix}{Fix}
\DeclareMathOperator{\diam}{diam}

\usepackage{paralist}
\setdefaultleftmargin{2\parindent}{}{}{}{}{}
\setdefaultenum{(a)}{}{}{}
\usepackage{graphicx}
\usepackage{floatflt}
\usepackage{bbm}
\begin{document}
\begin{center}
 \textbf{On eigen-structures for \pA\ maps}\\
\smallskip
Philip Boyland, Department of Mathematics\\
University of Florida, \texttt{boyland@ufl.edu}
\end{center}

\begin{narrower}
\noindent\textbf{Abstract:}
We investigate various structures associated with the
hyperbolic Markov and homological spectra of a \pA\ map
$\phi$ on a surface. Each unstable eigenvalue of 
the action of  $\phi$ on first cohomolgy
yields an eigen-cocycle that is transverse and holonomy
invariant   to the stable foliation $\cF^s$ of $\phi$.
Each unstable eigenvalue $\mu$ of a Markov transition matrix for
$\phi$ yields a holonomy invariant additive function $G$ on 
transverse arcs to $\cF^s$ with $\phi^* G = \mu G$. Except
when $\mu$ is the dilation of $\phi$, these transverse arc
functions do not yield measures, but rather holonomy invariant
eigen-distributions which are dual to H\"older functions.
Stable homological and Markov eigenvalues yield analogous transverse
structures to the unstable foliation of $\phi$.
The main tool for working with the homological spectrum is 
the Franks-Shub Theorem which holds for a general manifold
and map. For the Markov spectrum we use the correspondence
of the leaf space of stable foliation with a one-sided subshift of finite
type. This identification allows the symbolic analog
of a transverse arc function to be defined,  analyzed, 
and applied.

\end{narrower}
\bigskip
\tableofcontents

\vskip 2.0in


\section*{Preface to arxiv posted version} The contents of this paper 
will eventually be included in a monograph.
It therefore contains more expository material and redundancy
than  is usual for a journal paper. 

\newpage
\section{Introduction}
One of the striking features of the theory
of surface automorphisms, developed by Nielsen, Thurston, 
and many others, is the occurrence of linear and piecewise
linear (PL) structures in situations which  at first glance seem
highly nonlinear. Examples include the PL-parameterization
of closed curves and measured laminations on surfaces, the PL-action
of the mapping class group in this parameterization, the
affine  structure of \pA\ \homeos, and the Markov transition
matrix which codes \pA\ dynamics.  There is also
a surprising amount of information which is sometimes 
obtainable from the action of an automorphism on first homology. 

When there is a linear action the main objects 
of interest are often eigenvalues and eigenvectors.
Within surface theory, \pA\ \homeos\ play
a  central role, and 
the eigenvalue of greatest import for \pA\ maps
is the dilation. It is denoted $\lambda$
and occurs in many circumstances:
it is the \PF\ eigenvalue of the Markov transition matrix,
the spectral radius of the action on first homology when
the invariant foliations are oriented, the exponential
growth rate of the  action of the
\pA\ map on the fundamental group, and the spectral radius
of the induced action on closed curves. Its eigenvector
is used to construct the transverse measures to the invariant
foliations of the \pA\ map
and is reflected in the fundamental property of
these measures usually written as $\phi_* \;\cF^u = \lambda\; \cF^u$
or $\phi_*\; m^u = \lambda\; m^u$.

Given the importance of this \PF\ eigenvalue and eigenvector, 
it is natural to study 
the  meaning and uses of the rest of various spectra and their  
associated eigenvectors. Described roughly,
the main results here show that the  other eigenvalues
which are off the unit circle 
 give rise to semi-conjugacies from a covering
space to a linear map  as well as to additive functions defined
 on transverse arcs to the invariant foliations. 
Since it is well known that \pA\ foliations are
uniquely ergodic, these set functions cannot extend to measures,
but they are regular enough to define 
holonomy invariant eigen-distributions in the sense
of  continuous linear functionals on a space of H\"older
functions. 

The paper begins with topological constructions
based on results of  Franks (\cite{franks}) and 
Shub (\cite{shub}). These constructions work on 
any manifold $M$ for any continuous map $f$.
To describe them, assume that the action $f^*$
on first cohomology $H^1(M;\C)$ has an expanding
eigenvector $|\mu| > 1$ with eigen-class $c\in H^1(M;\C)$,
and so $f^* c = \mu c$. This says that for any cocycle
$\zeta\in c$, we have $f^* \zeta = \mu \zeta + d\chi$
for some function $\chi$. A natural question is
when is there an actual eigen-cocycle, \ie\ a cocycle
$\zeta$ with $f^* \zeta = \mu \zeta$? 

Perhaps the simplest situation in which  
such an eigen-cocycle exists
is when $f$ is smooth and there is a closed 
one-form $\omega$ with $f^* \omega = \mu\, \omega$. 
This is the case for \pA\ maps with orientable foliations
where $\mu$ is the dilation $\lambda$ and the kernel of
$\omega$ is tangent to the stable foliation. However,
having an eigen-one-form is a very strong property 
and cannot be expected 
to hold in any generality.  To get a general result
we need to extend the space of closed one-forms to include
enough cocycles so that each eigen-class actually
contains an eigen-cocycle. Perhaps the simplest formulation
of this extension  uses what is called a path cocycle below. 
A closed one-form can be used to    assign
numbers to paths in a homologically  invariant fashion,
and  we adopt this property as 
the definition of a path cocycle. In this language the Franks-Shub Theorem
says that for an unstable eigenvalue
$\mu$ there is always a path cocycle $F$ with $f^* F = \mu F$.

The first formulation and proof we give of the Franks-Shub Theorem in 
Theorem~\ref{mainthm} is not in terms of path cocycles,
but rather in terms of \cmaps. These are maps 
from the universal free Abelian cover
of $M$ into   $\R$ or $\C$ which transform under the deck
group as dictated by the cohomology class $c$.
 When such a map $\talpha$ represents an
eigen-cohomology class it satisfies 
\begin{equation}\label{firsteigen}
\talpha\tf = \mu \talpha
\end{equation}
for some lift $\tf$ and
is thus a dynamical semiconjugacy from $\tf$ to a linear map.
Since they are   functions on a manifold, \cmaps\ are often
technically a bit easier to work with than cocycles and, in addition,
they  are
useful for dynamical applications. The \cmap\ formulation
of the Franks-Shub Theorem is closer in spirit  to
Franks   \cite{franks} while the path cocycle version
is closer to Shub \cite{shub}.

In \S\ref{decomp}
 we note that when the level sets of an eigen-\cmap\ are projected
back to the base manifold $M$ they form an $f$-invariant
decomposition of $M$. The corresponding eigen-path cocycle
then describes the expansion by the factor $\mu$ ``transverse'' to the
decomposition  under the action
by $f$.  Since the sets of this decomposition can
be quite wild, we cannot make much progress at this
level of generality. From \S\ref{pAsec} to the end of the paper
we focus on \pA\ \homeo,  $\phi$, acting on a compact surface, $M$.
These \pA\ maps are characterized by a pair of transverse
foliations $\cF^s$ and $\cF^u$ each equipped with 
a transverse measure that is expanded or contracted by
a factor of $\lambda$ under the action by $\phi$.  
When these foliations are orientable, the stable
foliation $\cF^s$ is the decomposition determined by the eigen-\cmap\ 
with eigenvalue $\lambda$, and the corresponding
path cocycle is the transverse measure. By using $\phi\I$ one
obtains a decomposition and eigen-cocycle corresponding to
the unstable foliation and its transverse measure. 
The next step is to note that any other eigen-path cocycle 
for an eigenvalue $|\mu| > 1$ is also transverse and holomomy
invariant to $\cF^s$.
 In Proposition~\ref{tcfact}, we show that the
collection of transverse cocycles to $\cF^s$
is isomorphic to the unstable subspace of $\phi^*$ 
acting on $H^1(M, \C)$. In particular, each cohomology
class in that space contains exactly one transverse
cocycle. 

In \S\ref{tadsec} we introduce the class of transverse arc functions
(taf) and show that when the foliations are orientable,
they correspond to transverse cocycles. In the case of
non-orientable foliations there is no such connection 
and in \S\ref{sed} we begin the use of symbolic methods which
work in both the non-orientable and orientable
cases. The symbolic constructions are
based on the standard correspondence between
the leaf space of the stable foliation and  the 
one-sided subshift of finite type
 $\Lambda^+_A$ generated by the transition 
matrix of the \pA\ map, $A$. Our main objects 
are the symbolic transverse arc functions (staf)
on $\Lambda^+_A$. These are additive functions defined
on the collection of
cylinder sets which satisfy a coherence condition 
which ensures their correspondence to taf for $\phi$.  Fact~\ref{stdthread}
says  that 
the linear space of staf is naturally identified with the eventual 
image of $A$ and so is spanned by the eigen-staf with non-zero 
eigenvalues. Theorem~\ref{radonstd}
 says that each staf yields an element of the 
continuous dual of the appropriate class of H\"older
functions on $\Lambda^+_A$, but only the eigen-staf of 
the \PF\ eigenvalue of  $A$ extends to a signed or 
complex Borel measure on $\Lambda^+_A$.

The next step is to connect taf and  symbolic taf
in Theorem~\ref{tadstd}. While there is a symbolic taf
for each eigenvalue $\mu$ of the transition matrix,
only the unstable ones, $|\mu|>1$, yield tafs
to the stable foliation. The reason for this is roughly
that under the correspondence 
of  $\Lambda^+_A$ with the leaf space of 
$\cF^s$, the Cantor set $\Lambda^+_A$ is
collapsed into an arc $\Gamma$ transverse to $\cF^s$. While a
symbolic taf need only assign finite values to cylinder sets,
a taf must assign a finite value for all transverse arcs.
The collapses of cylinder sets from $\Lambda^+_A$ form a rather small
subset in the collection of all arcs in $\Gamma$, and when a symbolic
taf is not unstable, its push forward to $\Gamma$ would
assign infinite values to any arc outside this small subset. 

Using the correspondence  of  $\Lambda^+_A$ with $\cF^s$
in conjunction with Theorem~\ref{radonstd},
 we have in Theorem~\ref{tafdist}  that  
each taf yields an element of the 
continuous dual of the appropriate class of H\"older
functions on transversals to $\cF^s$, but only the eigen-staf of 
the \PF\ eigenvalue of $A$ yields a measure, namely,
the standard (and only) transverse measure to $\cF^s$.

In the last few sections we study eigen-\cmaps\ in
more detail and show in Theorem~\ref{cmapreg} that for eigenvalues
$ 1 < |\mu| < \lambda$, these \cmaps\ are nowhere locally
of bounded variation and nowhere differentiable as well as 
H\"older with exponent $\log(|\mu|)/\log(\lambda)$, 
but not H\"older for any larger exponents. The main idea
is that as a consequence of the eigen-property
\eqref{firsteigen}, if we restrict the \cmap\ to
a lifted unstable leaf we get a function $f$ that
everywhere satisfies $f(\lambda t) = \mu f(t)$. Since
$|\mu| < \lambda$ this means that $f$ has to ``fold up''
everywhere leading to the low regularity. This result
is illustrated in  \S\ref{eviltwin} with
an example of a \pA\ map $\psi$  on a genus two surface for
which the eigen-\cmaps\ patch together in pairs   yielding
semiconjugies from $\psi$  to two toral automorphisms, the first with
the same entropy as $\psi$ and the second with lesser entropy.
The first semiconjugacy is a branched cover while the
second semiconjugacy is nowhere differentiable and the
preimage of a typical point is a Cantor Set.

There is a fair amount of literature associated with
various aspects of this paper. While we have strived to 
keep the paper self-sufficient, we will
describe at least part of this literature because of
its importance for inspiration,  reference, and further
developments.

The first application of the Franks-Shub Theorem
to the study of \pA\ maps was by Fathi in \cite{fathi}.
That paper was the source of many   ideas developed
here. We also note that  Robertson
in  \cite{robertson} proves a  version of 
the Franks-Shub theorem concerning 
eigen-currents under homologically expanding 
smooth maps.

When a \pA\ foliation is orientable, one
can consider the flow along the leaves. 
The transverse holonomy invariant structures
developed here using the action of the \pA\ map
are  invariant under the flow. In \cite{forni1}
Forni developed a deep and general
theory of invariant distributions to a
flow on a surface. He
applied this to the study of the Teichm\"uller geodesic
flow in \cite{forni2} (cf. \cite{forni3}).
 The main ideas of relevance here
are  roughly as follows. A \pA\ 
foliation is a periodic orbit under the Teichm\"uller
flow. The second component of the Zorich-Kontsevich cocycle
of the return map is essentially the action
of the \pA\ mapping class on the first  
cohomology of the surface. Forni shows that the eigenvalues
of this linear action yield holonomy (or flow) invariant
distributions, or more precisely, basic currents of 
the foliation.

Next note that in \cite{bon1}
and \cite{bon2} Bonahon  develops the theory
of transverse H\"older distributions to geodesic
laminations on surfaces. There are many points
of contact and significant differences between
this theory and that of transverse structures to
\pA\ invariant foliations. First, the lamination
theory holds for any lamination while here we
just consider the foliations associated with \pA\ maps.
On the other hand, our main interest and methodology
is the induced action of the \pA\ map on various
linear structures while Bonahon's papers consider
geometric and analytic aspects of the laminations
without the action of a mapping class. Most fundamentally,
the map which collapses a \pA\ invariant lamination to 
a foliation is not H\"older since any smooth transverse
arc always intersects the lamination in a Cantor set of 
Hausdorff dimension zero. Thus the 
various manifestations of H\"older regularity in
the two theories are different. For example,
for a \pA\ stable foliation only the unstable eigenvalues
of the transition matrix yields transverse arc functions
(see Remark~\ref{badstaf}).
The analogous structure for a lamination defined in
\cite{bon2} need only assign a finite value to arcs with endpoints
in the complement of the lamination and thus stable
eigenvalues also generate a transverse structure.

A \pA\ \homeo\ on a surface is hyperbolic at
all but finitely many points and our use
of symbolic dynamics puts us squarely within the classical
theory of hyperbolic dynamics.
Since we are concerned with the spectrum
of Markov transition matrices, in some cases
we are dealing with the simplest special case 
of the vast and deep theory of  transfer operators.
We do not describe this is any detail here, but 
refer the reader to Baladi's excellent book \cite{baladi}.
Of direct relevance  however is the paper  \cite{ruelle}
in which Ruelle  defines Gibbs distributions
dual to H\"older functions on a subshift
of finite type.    Haydn in \cite{haydn} shows that
Gibbs distributions are always eigen with respect to 
the action of the dual of the transfer operator. In
Theorem~\ref{radonstd} we show that the distributions constructed
from eigen-stafs have  this property.

A sequel to this paper will consider these
connections in more detail as well as 
applications of this paper  to the dynamics,
statistics and geometry of of \pA\ maps.

\medskip
\noindent\textbf{Acknowledgement:} Many of the ideas
in this paper had their inception in extensive 
conversations with Gavin Band in Spring, 2007.

\section{Preliminaries}

\subsection{Linear algebra}\label{linalg}
To set terminology we begin with some standard notions.
For a square matrix $A$, recall that  $\mu$
is called an \textit{eigenvalue} if 
$\ker(A - \mu I) \not= 0$, and
$v$ is called a corresponding \textit{eigenvector}
if $v \in \ker(A - \mu I)$ and $v$ is called a corresponding
\textit{generalized eigenvector} 
if $v \in \ker(A - \mu I)^k$ for some $k>1$.
Note that under these definitions an eigenvector
is not a generalized eigenvector. The \textit{generalized
eigenspace} of $\mu$ is the subspace consisting of all $\mu$'s
eigenvectors and generalized eigenvectors in addition to the
zero vector. An \textit{eigenchain} for $\mu$ of length 
$k \geq 1$ is a set
of vectors $v_1, \dots, v_k$ with $v_1$ an eigenvector,
each $v_i$ with $i > 1$ a generalized eigenvector with
$(A - \mu I) v_{i+1} = v_i$ for all $i$. The generalized eigenspace
of $\mu$ always has a basis consisting of the union of eigenchains;
each eigenchain is the basis of one of the Jordan
blocks corresponding to $\mu$.

Let $V$ be
a finite dimensional vector space over $\R$ or $\C$ and
$T:V\raw V$ a linear transformation. An eigenvalue $\mu$
of $T$ is called unstable, central, stable, and 
nilpotent, respectively, if $|\mu|>1, |\mu|=1, 0 < |\mu| < 1$,
and $\mu=0$.
The \de{unstable subspace} of $T$, denoted $Un(T,V)$ is the direct
sum of all the generalized eigenspaces of $T$ associated
with unstable eigenvalues.  The 
\de{central}, \de{stable}, and \de{nilpotent} subspaces
are denoted by $Cen(T,V), Stab(T,V),$ and $Nil(T,V)$,
respectively, are associated with 
central, stable, and nilpotent eigenvalues. The direct
sum decomposition,
 $V = Un(T,V)\osum Cen(T,V)\osum Stab(T,V) \osum Nil(T,V)$,
is preserved by $T$, and $T$ restricted to a factor is
denoted $T^{Un}$, etc. The \de{non-nilpotent} subspace,
denoted $NonN(T,V)$, is the direct sum of all the 
generalized eigenspaces connected with nonzero eigenvalues,
and so $NonN(T,V) = Un(T,V)\osum Cen(T,V)\osum Stab(T,V)$.
The non-nilpotent subspace is the same as the eventual
range of $T$, $NonN(T,V) = \cap_{n\in\N} T^n(V)$, and
$T^{NonN}$ is always a  self-isomorphism of $NonN(T,V)$.
The \de{hyperbolic} subspace is  the direct sum
of  the stable and unstable ones.

There will be a variety of linear objects discussed in
this paper. The terminology ``unstable'' when applied
to such objects always indicates that the object is
contained in the unstable subspace of the linear
transformation under discussion. 

We will also need some of the results that go under the general
rubric of the \PF\ Theorem.
A matrix $A$ with $A^n>0$ for some $n>0$ 
always has a simple eigenvalue of largest modulus, which is
always real and it has a strictly positive eigenvector.
Let the unit length, strictly positive eigenvectors
from the left be $\vell$ and the right $\vr$. No other eigenvalues
have  strictly positive eigenvectors, and if 
$\vv$ is any non-negative vector,
then $A^n \vv / \|A^n \vv\| \raw \vr$ as $n\raw\infty$. 
In addition,
\begin{equation}\label{PFconv}
A^n/\lambda^n \raw P
\end{equation}
as $n\raw\infty$ where $P =  \vr  \; \vell$. Note that
here $\vr$ is treated as a column vector and $\vell$ as
a row vector.  

For future use we record an easy but somewhat technical fact whose
proof is a straightforward application of the \PF\ theorem
and the Jordan canonical form.
\begin{fact}\label{PFfact}
Assume that $A$ is a square matrix 
with $A^n>0$ for some $n>0$, and let $\lambda$ be
its \PF\ eigenvalue and $r>0$ be such that  
 $\lambda > r > |\mu| $ for any eigenvalue
$\mu \not = \lambda$. Given a vector
$v\in NonN(A)$ which is not an eigenvector for $\lambda$,
there exist $C>0$ and $N$ so that for any vector $w$ 
with $A^n w = v$ and $n>N$, we have $\|w \|_1 > C/ r^n$.
\end{fact}

\begin{definition}[The field $\F$]
In the sequel it will often be the case that the appropriate
field for coefficients of homology/cohomology or
for the range of a map or homomorphism will depend
on whether an eigenvalue $\mu$ under consideration 
is real or complex. To avoid the awkwardness of 
the constant repetition of the phrase ``where the
field $\F$ is $\R$ or $\C$ depending on whether $\mu$ is
real or complex'', we adopt the convention that the
field is denoted $\F$ and has a value  $\F = \R$ or
$\C$ as is appropriate in the given situation.
\end{definition}

\subsection{First homology, cohomology and the
universal Abelian covering space}\label{linrep}
Let $M$ be a smooth, connected, compact manifold of any
dimension.  Fix a base point $x_0\in M$ and
a set of generators of the fundamental group $\pi_1(M, x_0)$
whose Abelianizations give the basis of $H_1(M;\Z)$. The 
\de{universal Adelina covering space} (also called the
homology cover) is the
largest covering space of $M$ whose automorphism (or deck) group is Abelian.
Thus, this covering space, which is denoted
 $\pi:\tM\raw M$ (or just $\tM$) here, satisfies
$\pi_*(\pi_1(\tM))=[\pi_1(M,x_0), \pi_1(M, x_0)]$, 
the commutator subgroup. To obtain a metric on
$\tM$, we fix a metric on $M$,
and lift it to $\tM$ yielding an equivariant, topological metric we
denote $\td$.
The usual universal cover of
$M$ which is the cover with deck group 
equal to $\pi_1(M)$ will only rarely be used here and
is denoted $\hM$.

For simplicity of exposition we assume that 
$H_1(M;\Z)$ is torsion-free and has rank $d>0$. We leave to
the reader the minor changes needed for the case of torsion.  Thus 
the deck group of $\tM\raw M$ is $\Z^d$ where $d$ is the
first Betti number of $M$. For
$\vn\in\Z^d$, we let $\delta_{\vn}$ denote the corresponding
element of the deck group.

Recall that a pair of paths 
 $\gamma_1$ and $\gamma_2$ in $M$ are said to be 
\de{homologous}, if the loop $\gamma_1\# \gamma_2\I$ is
null-homologous. 
A important feature of the universal Abelian cover $\tM $
  is that a loop $\Gamma\subset M$ lifts to a loop in $\tM$ if and
only if $\Gamma$ is null homologous in $M$, or equivalently,
two paths in $M$, $\gamma_1$ and $\gamma_2$, with the same endpoints
lift to two paths in $\tM$, $\tgamma_1$ and $\tgamma_2$, 
with the same endpoints 
if and only if  $\gamma_1$ and $\gamma_2$ are homologous in $M$.

For a continuous self-map $f:M\raw M$, let $f_*$ and
$f^*$ be the induced actions on
$H_1(M)$ and $H^1(M)$. 
Any  $f:M\raw M$ lifts to the universal
Abelian cover $\tM$.  If $\tf$ is
a lift of $f$ to $\tM$, a fundamental relation is
\begin{equation}\label{equilift}
\tf\circ\delta_{g} = \delta_{f_*(g)}\circ \tf.
\end{equation}

For the purposes of comparison with other cocycles
discussed below
we also recall standard terminology surrounding
first de Rham cohomology,
$H^1_{DR}(M;\R)$. The vector
space $H^1_{DR}(M;\R)$ is composed of the cohomology classes
of closed one-forms on $M$ with two closed forms $\omega_1$ and
$\omega_2$ being called \de{cohomologous} if there is a smooth function
$\chi:M\raw \R$ so that the  
exact one-form $d\chi$ satisfies $\omega_1 = \omega_2 + d\chi$.
For a smooth loop $\Lambda$ in $M$ with homology
class $[\Lambda]\in H_1(M;\Z)$ and a closed one-form
$\omega\in c$, define $\Phi_c([\Lambda]) = \int_\Lambda \omega$.
This definition is independent of the choices and 
the map $c\mapsto \Phi_c$ is an isomorphism from
$H^1_{DR}(M;\R)$ to $Hom(H_1(M;\Z), \R)$. Here we will
usually extend $\Phi_c$ to a functional 
$H_1(M;\R) \raw \R$ and without comment treat
elements of $H^1(M;\R)$ as elements of the dual
space of $H_1(M;\R)$, or after a choice of generators,
with linear maps $\R^d\raw\R$. We will also consider de Rham
cohomology with complex coefficients where all the same
definitions and properties apply.

\subsection{H\"older spaces}
If $(X, \rho)$ is a metric space, the space of all continuous
$f:X\raw \F$ is denoted $C^0(X,\F)$. This space is 
given the sup-metric
\begin{equation}
d_0(f,g) = \sup \{ |f(x)- g(x)| : x\in X\},
\end{equation}
when the supremum is finite.
For  $f\in C^0(X,\F)$
 and $0 < \nu \leq 1$, define
\begin{equation}
\nun{f} = \sup_{x_1\not = x_2} 
\frac {|f(x_1) - f(x_2)|}{d(x_1, x_2)^\nu}.
\end{equation}
If $\nun{f} \leq C < \infty$, then $f$ 
is said to be \de{ $(\nu, C)$-H\"older}.
The space of all H\"older $f$ with given exponent 
$0 < \nu < 1$ is denoted
$C^\nu(X,\F)$ and is given the metric
\begin{equation}
d_\nu(f,g) = d_0(f,g) + \nun{f-g}.
\end{equation}
It is standard that with these metrics for all $0< \nu<1$  and
for $X$ compact, $C^\nu(X,\F)$ is complete and separable.
As is conventional, functions
which are  $(1, \lambda)$-H\"older will be called \de{$\lambda$-Lipschitz}. 
The space of all Lipschitz $f$ is denoted $C^{Lip}(X,\F)$. The notation
$C^1(X,\F)$ is reserved for the space of all functions
with continuous first derivatives.

\subsection{Note on  the terms cocycle and distribution}
The terms cocycles,  cohomologous, etc have
different meanings in dynamics and algebraic topology.
Here it will always be the latter usage unless there is 
an explicit comment to the contrary.

Also, the term distribution is often used in dynamics to
refer to an invariant foliation or lamination, as in
``the unstable distribution''. Here distribution will
always be in the sense of Schwartz as an element of a
dual space, i.e. a linear functional. 

\section{$C$-Maps and the Franks-Shub Theorem}\label{cmapsec}
As motivation and illustration we start with
the  construction of a \cmap\ in the simplest case of an
eigen-object as was described in the introduction. So assume
that $f:M\raw M$ is smooth and there 
we have a closed one-form $\omega$ with $f^* \omega =
\mu \omega$ with $\mu\in\R$.
 Lift (or pullback) the eigen-one-form $\omega$ to $\tomega$ on 
the universal  Abelian cover
$\tM$ and after fixing a base point
 $\tz_0\in\tM$, define
$\sigma_\omega:\tM\raw\R$ by 
\begin{equation}\label{formmap}
\sigma_\omega(\tz) = \int_\tgamma \tomega,
\end{equation}
where $\gamma$ is any smooth path in $\tM$ connecting $\tz_0$ and
$\tz$. Since $\omega$ and thus $\tomega$ is closed, 
the definition is independent of the choice of path. In fact, 
$\tomega$ is exact in $\tM$ with $\tomega = d \sigma_\omega$. Informally,
this happens because the universal Abelian cover $\tM$ is exactly the space 
that is constructed by unwrapping  
$M$ just enough 
to remove  all the  obstructions to a closed form being exact. 
Now lift the \diffeo\ $f:M\raw M$ to $\tf:\tM\raw\tM$ and for
simplicity assume that the base point $\tz_0$ is a fixed point of
$\tf$. Thus, as a consequence of $f^* \omega = \mu\omega$, we
have that $\sigma_\omega\circ \tf = \mu\; \sigma_\omega$. In other
words, $\sigma_\omega$ gives a semi-conjugacy from $\tf$ on $\tM$
to multiplication by $\mu$ on $\R$.

The map $\sigma_\omega$ is a special case of 
a \textit{\cmap}\ defined in the next subsection.
In Theorem~\ref{mainthm} we show that in the general situation
of a continuous self-map of a manifold $M$, 
 each eigenvalue $\mu$ with $|\mu|>1$ 
has a corresponding eigen-\cmap.

\subsection{Definition of \cmaps}
The definition of a \cmap\ is in terms of a given, specified
cohomology class. In \S\ref{cocyclesec} below, we consider structures
more suitable for general cohomology formulation.

Given a non-zero class $c\in H^1(M,\F)$,
a map $\sigma:\tM\raw \F$ is called a \de{\cmap} if 
 \begin{equation}\label{cmapdef}
\sigma \circ \delta_{\vn} = \sigma + \Phi_c(\vn),
\end{equation}
for all $\vn\in\Z^d$ where the linear functional $\Phi_c:\F^d\raw \F$
represents the class $c$. 
Two \cmaps\ representing the same 
class $c$   are said to be \de{cohomologous}.
It follows from  \eqref{cmapdef} 
that $\sigma_1$ and $\sigma_2$ are cohomologous if and only if 
\begin{equation}\label{cohom1}
\sigma_1\circ\delta_\vn - \sigma_2\circ\delta_\vn =
\sigma_1 - \sigma_2,
\end{equation}
for all $\vn\in\Z^d$.
In turn, \eqref{cohom1} happens if and only if 
\begin{equation} \label{cohom2}
\sigma_1 = \sigma_2 + \chi\circ \pi,
\end{equation}
for some continuous $\chi:M\raw \F$, where $\pi:\tM\raw M$ is the
covering map. 

\begin{remark}
For a torus the universal cover is the same
as the universal Abelian cover which can be 
identified with the vector space $H_1(M;\R)$.
 In most other cases the universal Abelian
cover is not a vector space, but it is often useful
to adapt the heuristic of $\tM$ being identified 
with  $H_1(M;\R)$. (They are in fact ``coarsely equivalent''
by  a standard equivariant embedding of $\tM$ into
$\R^d$, see, for example, \cite{bdtrans}).
In the case at hand, a real cohomology class is actually a linear
functional 
$H_1(M;\R)\raw \R$, but we ``identify''
 it with  a \cmap\  $\tM\raw\R$.
\end{remark}

For a cohomology class $c\in H^1(M,\F)$, let $S_c^\nu$ denote
the collection of \cmaps\ $\sigma\in C^\nu(\tM, \F)$.
The metric on  $S_c^\nu$ is that induced as a subspace of $ C^\nu(\tM, \F)$.
Now if  $\omega$ is a $C^\infty$, closed one-form
with $c = [\omega]$ in  $H^1_{DR}(M;\F)$, then
it is immediate that 
$\sigma_\omega$ defined by \eqref{formmap} is a \cmap.
Now for any other \cmap\   $\sigma\in S_c^\nu$,
by \eqref{cohom2}, there exists an $\chi\in C^\nu(M,\F)$ 
with $\sigma = \sigma_\omega + \chi\circ\pi$.
This yields an isometry of $S_c^\nu$ 
with $C^\nu(M, \F)$ and so, in particular, $S_c^\nu$ is 
a complete, separable metric space

\begin{remark}\label{largelip}
As we just noted, 
any \cmap, $\sigma$, can be written $\sigma = \sigma_\omega + \chi\circ\pi$
with $\sigma_\omega$ constructed from a closed one-form in
the class of $\sigma$. This implies that for
any equivariant metric $\td$ on $\tM$ there are 
$\lambda, K > 0$   so that for all $\tx, \ty\in\tM$, 
$|\sigma(\tx) - \sigma(\ty)| \leq \lambda\; \td(\tx, \ty) - K$.
This property is something expresses by saying that
$\sigma$ is large scale Lipschitz. 
\end{remark}

\subsection{The Franks-Shub Theorem}
This section contains the Franks-Shub Theorem
formulated in the language of \cmaps. This is essentially
the point of view in Franks \cite{franks}, but applied to
one eigenvalue at a time. Shub in \cite{shub} gave an equivalent
theorem in the language of Alexander-Spanier cocycles.

\begin{theorem}[Franks, Shub]\label{mainthm}
Assume that $f:M\raw$ is a continuous map of the smooth, connected manifold
$M$ with $H^1(M;\Z)$ torsion-free
and $\mu\in\F$ is an eigenvalue of \newline
$f^*:H^1(M;\Z) \raw H^1(M;\Z)$
with $|\mu| > 1$ and eigenchain
 $\{c_1, \dots, c_k\}\subset H^1(M;\F)$. For each lift
$\tf:\tM\raw \tM$ of $f$ to the universal Abelian cover  $\tM$
there exists unique $c_i$-maps $\talpha_i:\tM\raw\F$ for 
$i = 1, \dots, k$ with 
\begin{equation}\label{semieq}
\talpha_1\circ\tf = \mu \talpha_1 \ \text{and}  \ \text{for} 
\ i>1, \ 
 \talpha_{i}\circ\tf = \mu \talpha_{i} + \talpha_{i-1}.
\end{equation}
Further, if $f$ is Lipschitz 
with constant $\lambda$, then $\talpha_i\in C^\nu(M,\F)$  for all 
$0 \leq\nu < \log(|\mu|)/\log(\lambda)$.
\end{theorem}

\begin{proof}  The proof is by induction on the
index $i$ in the eigenchain. To prove the
case $i=1$, assume that $c = c_1$ is 
an eigen-class for $f^*$. For
$\sigma\in S_c^\nu$, let 
\begin{equation*}
F(\sigma) = \frac{\sigma\circ \tf}{\mu}.
\end{equation*}
Now using \eqref{equilift} and \eqref{cmapdef}
\begin{equation*}
\begin{split}
F(\sigma)\circ\delta_{\vn} &= 
 \frac{\sigma\circ \tf \circ\delta_{\vn}}{\mu}\\
&= \frac{\sigma\circ \delta_{f_*(\vn)}\circ \tf}{\mu}\\
&= \frac{\sigma\circ \tf +  \Phi_c(f_*(\vn)) }{\mu} \\
&= \frac{\sigma\circ \tf}{\mu}
 +  \Phi_c(\vn),\\
\end{split}
\end{equation*}
where in the last line we used $\Phi_c\circ f_* =
f^*(\Phi_c) = \mu \Phi_c$, because
$\Phi_c$ represents the eigen-class $c$. 
Thus $F:S_c^0\raw S_c^0$.
Now if $\sigma$ is $(C,\nu)$-H\"older, using the
fact that $\tf$ is $\lambda$-Lipschitz, 
we get  that 
$F(\sigma)$ is $(C\lambda^\nu/{\mu}, \nu)$-H\"older.
Thus $F:S_c^\nu\raw S_c^\nu$, for $0 < \nu < 1$ as well.

When $|\mu| > 1$, it is obvious that
$F:S_c^0\raw S_c^0$ is a contraction with constant $|\mu|^{-1}$.
Thus $F$ has  a fixed point $\talpha$ which is at least 
$C^0$ and satisfies \eqref{semieq}.
For $0 < \nu < 1$, 
\begin{equation*}
\begin{split}
d_\nu(F(\sigma_1), F(\sigma_2) 
&= \sup_{\tx\in\tM}\left|
\frac{\sigma_1\circ\tf(\tx) -\sigma_2\circ\tf(\tx)}{\mu}\right|\\
&+ \sup_{\tx_1\not=\tx_2}\left|
\frac{(\sigma_1\circ\tf(\tx_1) - \sigma_1\circ\tf(\tx_2)) -
(\sigma_2\circ\tf(\tx_1) - \sigma_2\circ\tf(\tx_2))}
{\mu \td(\tx_1, \tx_2)^\nu}\right|\\
&= \frac{d_0(\sigma_1, \sigma_2)}{\mu}\\
&+  \sup_{\tx_1\not=\tx_2}\left|
\left(\frac{(\sigma_1-\sigma_2)(\tf(\tx_1)) -
(\sigma_1-\sigma_2)(\tf(\tx_2))}{\mu \td(\tf(\tx_1), \tf(\tx_2)^\nu}\right)
\left(\frac{\td(\tf(\tx_1), \tf(\tx_2)^\nu}{\td(\tx_1, \tx_2)^\nu}\right)
\right|\\
&\leq \frac{d_0(\sigma_1, \sigma_2)}{|\mu|} + 
\frac{\nun{\sigma_1 - \sigma_2}\lambda^\nu}{|\mu|}\\
&\leq\frac{d_\nu(\sigma_1, \sigma_2)}{ |\mu|\lambda^{-\nu}},\\
\end{split}
\end{equation*}
Where in the last line we used the fact that
when $f$ is  $\lambda$-Lipschitz, then the existence of
a $C^0$-$\talpha$ satisfying \eqref{semieq} implies that
$\lambda \geq |\mu|$.
Thus if $|\mu|\lambda^{-\nu} >1$, \ie\ when 
$\nu<{\log|\mu|}/{\log\lambda}$,
$F$ is a contraction on the complete vector space $S_c^\nu$, 
and so it has a unique H\"older fixed point $\talpha$, finishing the proof
for $i=1$.

Assume now that we have proven the result for $c_i$.
 Let $S^\nu_{c_{i+1}}$ be all
the $\nu$-H\"older \cmaps\ that represent the class $c_{i+1}$, 
and for $\sigma\in S^\nu_{c_{i+1}}$, let 
\begin{equation*}
G(\sigma) = \frac{\sigma\circ\tf - \talpha_i}{\mu},
\end{equation*}
where $\talpha_i$ is the eigen-\cmap\ representing
$c_i$ given by the inductive hypothesis. It is easy to 
check that in fact
$G:S^\nu_{c_{i+1}} \raw S^\nu_{c_{i+1}}$ and is a contraction
for all $0 \leq \nu < 1$, and so $G$ has a unique fixed
point, yielding a generalized eigen-\cmap\ $\talpha_{i+1}$ with
\begin{equation}\label{gencmapeq}
\talpha_{i+1}\circ\tf = \mu \talpha_{i+1} + \talpha_i,
\end{equation}
finishing the induction.
\QED\
\end{proof}

\subsection{Remarks on the Franks-Shub Theorem}
If $f$ is a \homeo\ then by using $f\I$, there is an eigen-\cmap\
for all eigenvalues with $0 < |\mu|< 1$. Note that 
$\mu = 0$ does not occur when $f$ is a \homeo. 

While we used a fixed point argument,
the version of the theorem
in Shub~\cite{shub} (and implicitly in Franks~\cite{franks})
is proved using a  summation formula for 
$\talpha$. In the current context the argument
goes like this. Given the eigen-class $c$, pick
a close one-form $\omega$ in the class and
construct $\sigma_\omega$ as in \eqref{formmap}.
Now as in the proof above, 
$(f^* \sigma_\omega)\circ \delta_{\vn} =
\mu \sigma_\omega + \Phi_{\mu c} \vn$, and
so  $(f^* \sigma_\omega)$ is a $(\mu c)$-map.
Now $\mu \sigma_\omega$ is also a $(\mu c)$-map,
and so 
by \eqref{cohom2}, there exists an $\chi\in C^\nu(M,\F)$ 
with $f^* \sigma_\omega = \mu \sigma_\omega + \chi\circ\pi$.
Now define 
\begin{equation}\label{cmapsum}
\talpha:= \sigma_\omega + \sum_{j=1}^\infty 
\frac{\chi\circ\pi\circ\tf^{j-1}}{\mu^j},
\end{equation}
which converges since  $|\mu|>1$,  and   
$f^* \talpha = \mu \talpha$ by direct
verification.

A summation formula as in \ref{cmapsum} is familiar
classically  from Weierstrass nowhere differentiable
functions and more recently in  the theory of fractal
functions. It is therefore not surprising that even
for very smooth $f$, the eigen-cmaps are often
of very low regularity. For the case of eigenvalues
of a \pA\ map with $1 < |\mu| < \lambda$, 
 Theorem~\ref{cmapreg} below shows that 
$\talpha$ is nowhere differentiable.

The summation formula also reveals the connection
to  well-known facts about solutions to
hyperbolic \textit{dynamical cocycle} equation. Specifically,
continue to assume that we have an eigen-class
containing the closed one-form $\omega$ and the
corresponding \cmap\ $\sigma_\omega$ satisfies
 $f^* \sigma_\omega = \mu \sigma_\omega + \chi\circ\pi$
for some  $\chi\in C^r(M,\F)$ when $f$ is $C^r$. 
Since  the eigen-\cmap\
 we seek is in the same class, we may represent
$\talpha$ as $\talpha = \sigma_\omega + h\circ\pi$
for some $h\in C^0(M,\F)$. Plugging into 
$f^* \talpha = \mu \talpha$ yields an
equation $(h\circ\pi)\circ\tf - \mu (h\circ\pi) =
-\chi\circ\pi$. Thus projecting to the base the
semi-conjugacy questions reduces to solving
the hyperbolic dynamical cocycle equation
\begin{equation}\label{dyncocycle}
h \circ f - \mu h = -\chi,
\end{equation}
for $h$ given $\chi$. It is well
known that the solution is the summation in \eqref{cmapsum} 
projected to the base, and that even for very
regular $\chi$ the corresponding $h$ is often
no more regular than H\"older.

We also note that methods similar to Theorem~\ref{mainthm}
 yield a semiconjugacy from the full universal
cover of $M$ to linear
expansion by $\mu$ on $\F$. However, the smaller
the cover, the more information a semiconjugacy 
yields about the dynamics. For example, \pA\ maps lifted
to their universal cover are dynamically uninteresting having
at most one recurrent point, but they
can be transitive in the universal Abelian cover (\cite{bdtrans}). The
most useful information is semiconjugacy from
the manifold itself.

Thus a natural question is when the semiconjugacy given by 
a single eigen-\cmap\ or a collection of eigen-\cmaps\
descends to a semiconjugacy defined on the manifold $M$ itself.
This simplest  case is when the eigenvalue $\mu$ of $f^*$ on 
$H^1(M;\Z)$ is an integer $n$ in which case the eigen-\cmap\ descends
to a semiconjugacy from $(M,f)$ to $z\mapsto z^n$ on the circle.
The simplest case of this is when $M$ is the circle
and $f$ is a degree-$n$ map with $|n|>1$ (\cf\ \cite{doubling}).
Perhaps the most studied case is when the spectrum of 
$f^*$ is pure hyperbolic (no eigenvalues are on the unit circle)
in which case the eigen-\cmaps\ fit together and descend
to a semiconjugacy into a torus of dimension equal to
that of $H^1(M;\R)$. This result is contain
in Franks paper \cite{franks} and was the 
origin of our investigations. Fathi in \cite{fathi}
uses Franks' result to send a \pA\ map on a surface
into a compact invariant subset of an Anosov toral
automorphism. Whether this can be done injectively
in a fascinating open question also studied in
\cite{band} and \cite{bargekwap}.
Fathi in \cite{fathi} also 
gives conditions in terms of the splitting of the characteristic
polynomial of $f^*$ which imply a semiconjugacy from $M$ to
a torus of dimension equal to the degree of an irreducible
factor. 

\subsection{Decompositions and transverse structure}\label{decomp}
Recall that a \de{decomposition} of a manifold is 
way to write it as a disjoint union of sets.
Assume now that $\talpha$ is an eigen-\cmap\ for $\tf$
with factor $\mu$ with $|\mu| > 1$. We  get
a decomposition of $\tM$ by closed sets in the usual fashion:
for each $r\in\F$, let $\tX_r := \talpha^{-1}(r)$.
It is clear that the decomposition is $\tf$-invariant
with  $\tf(\tX_r) = \tX_{\mu r}$.

To describe the corresponding invariant decomposition
in the base manifold $M$, first note that 
it follows directly from definition of  a \cmap\
and \eqref{equilift} that
$\delta_{\vn}(\tX_r) = \tX_{r + \Phi_c(\vn)}$, and so 
$\delta_{\vn}(\tX_r) = \tX_r$ for all $\vn\in\ker(\Phi_c)$.
Next, recalling  that  $\pi:\tM\raw M$ is the projection, 
it is easy to check that two projected decomposition
elements sets $\pi(\tX_r)$ and $\pi(\tX_s)$ intersect
if and only if they coincide in which case $s = r + \Phi(n)$
for some $n\in\Z^d$. 

Thus the collection of sets $\pi(\tX_r)$ form a decomposition
of $M$, and this collection can be be indexed
as $\{X_\eta\}$ with $\eta\in \F/\Phi(\Z^d)$. 
By construction, the decomposition $\{ X_\eta\}$ is
$f$-invariant. Specifically, for each $\eta\in \F/\Phi(\Z^d)$,
$f(X_\eta) = X_{\eta'}$ where $\eta'$ is the class
of $\mu \eta$. Note that this  makes sense, for if
$r,r'\in F$ and $r-r' = \Phi(\vn)$, then $\mu r - \mu r'
= \mu(r-r') = \Phi(f_*(\vn))$ and so multiplication
by $\mu$ on $\F$ respects the equivalence relation and so
descends to a map on the quotient $\F/\Phi(\Z^d)$.

If the eigen-cohomology class $c$ is irrational in the
sense that $\Z^d \cap \ker(\Phi_c) = 0$,
$\pi$ restricted to every $\tX_r$ is  injective,
and so the sets $X_\eta$ generally wrap around
$M$ and are not closed and are not the level
sets of a continuous function on $M$. Instead,   
 the decomposition of $M$ is
made up of level sets of a family of
locally defined continuous functions. Shub calls this a
translational H-striation in \cite{shub} where he notes
that ``H'' stands for Haeflinger. He also 
notes that if $f$ has a collection $\mu_i$ of expanding
eigenvalues, and so in our language we have many eigen-\cmaps\
$\talpha_{\mu_i}$ and their corresponding $f$-invariant
decompositions, $\{X_\eta^{(i)}\}$, one may also form
the intersection decomposition $\{\cap X_\eta^{(i)} \}$
and it will also be $f$-invariant. If $f$ is a \homeo\
one may also include the decompositions coming from eigenvalues with
$|\mu| < 1$.

The generalized eigen-cocycles coming from generalized
eigenvectors of $f^*$ also yield invariant decompositions.
For example, assume that 
$\talpha$ and $\talpha'$ satisfy \eqref{gencmapeq}. For
$r, s\in \F$, let $\tW_{r,s} = \talpha\I(r) \cap (\talpha')\I(s)$.
Then this yields a decomposition of $\tM$ which satisfies
$\tf(\tW_{r,s}) = \tW_{\mu r,\mu s + r}$, and pushes down
to one on $M$  that is transversally stretched but ``skewed'' 
as well.

In general, the decomposition elements 
$\tX_r$ and $X_\nu$ can be topologically
extremely complicated, and we do not pursue this
level of generality. In the case of main interest here,
a \pA\ map $\phi$, if $\phi$ has oriented foliations and $\lambda$ is
the eigenvalue of $\phi^*$ with largest modulus, 
then the invariant decompositions
corresponding to $\lambda$ and $\lambda\I$ are the stable and
unstable foliations of $\phi$. For an eigenvalue $\mu$ with 
$1 < |\mu| < \lambda$, we shall see that the decomposition 
elements $X_\mu$ are unions of stable leaves and
are typically locally a Cantor set cross an interval
(see Theorem~\ref{cmapreg}). The 
intersection decomposition elements formed
using the decompositions coming from both $\mu$ and $\mu\I$ are 
thus typically Cantor sets.
In the next section we consider the cocycle which in the
general topological case 
plays the same role of the transverse measure of a \pA\ map.

\section{Cocycles}\label{cocyclesec}
While dynamically valuable and simple  
to work with analytically, \cmaps\ are awkward in
various ways. For example, we have seen that
the corresponding eigen-objects depend on the choice of 
lift of $f$ to the universal Abelian cover. In
this section we consider two types of
cocycle, each being  convenient in certain
circumstances. Our purpose
here is not to develop full cohomology theories, but
rather just describe the cocycles useful in the sequel.

As described in the introduction,
these various cocycles may be viewed  as ways
to extend or complete the collection of closed one-forms
to a larger theory in which each eigen-cohomology
class contains a eigen-cocycle. As such, we
proceed each definition with the version
of the cocycle associated to a closed
one-form.

\subsection{Path cocycles}
A path cocycle is the generalization
of the integral of a closed one-form over an path.
If $\omega$ is a closed one-form on $M$, and $\gamma$ is a smooth,
oriented path in $M$, define 
\begin{equation}\label{pathform}
F_\omega(\gamma) = \int_\gamma \omega.
\end{equation}
If $\gamma$ is not smooth, let $F_\omega(\gamma) = F_\omega(\gamma')$,
where $\gamma'$ is smooth, has the same endpoints as $\gamma$,
and is $C^0$-close to $\gamma$. This $F_\omega$ will be additive on
oriented paths with a common endpoint and further, its value
on a path depends just on the homology class of the path. 
Put another way, since $\omega$ is closed, if $\gamma_1$ and
$\gamma_2$ are homologous rel endpoints, then $F_\omega(\gamma_1) =  
F_\omega(\gamma_2)$. This implies that if $\Gamma\subset M$ is a
closed curve, then $F_\omega(\Gamma)$ depends just on the
homology class of $\Gamma$ and so $F$ induces a linear
functional $H_1(M,\Z)\raw \R$ and thus yields a cohomology
class on $M$. If $\omega$ is an eigen-one-form, $f^*\omega = 
\lambda \omega$, then $F_\omega$ will be an eigen-path cocycle
in the sense that $f^* F_\omega = \lambda F_\omega$. Thus $F_\omega$ represents
an expanding transverse structure when $|\mu| > 1$. 

To give the general definition, 
let $\cP = C^0([0,1], M)$ with the sup metric. Note that 
this is the space of \textit{oriented paths}.  The terminology
 \textit{arc} will be used here for corresponding non-oriented
set and will be considered in \S\ref{tadsec} below.
Given two paths $\gamma_1, \gamma_2\in\cP$,
with $\gamma_1(1) = \gamma_2(0)$, then
$\gamma_1 \# \gamma_2$ represents the usual path sum.
The essentials of the  following definition 
come from \cite{fathi}.

\begin{definition}[Path cocycle]\label{pathcocycledef}
 An \de{path cocycle} over $\F$ 
is a continuous map $F:\cP\raw \F$ which is
\begin{compactenum}
\item Additive: if $\gamma_1(1) = \gamma_2(0)$, then
$F(\gamma_1 \# \gamma_2) = F(\gamma_1) + F(\gamma_2)$.
\item Closed or Homology Invariant: if $\Gamma$ is a closed
loop with $[\Gamma] = 0$ in $H_1(M;\Z)$, then $F(\Gamma) = 0$.
Equivalently, if $\gamma_1$ and $ \gamma_2$
have the same endpoints and are homologous, then 
$F(\gamma_1) = F(\gamma_2)$. 
\end{compactenum}
\end{definition}

As a consequence of (b), the value of an path cocycle is 
independent of parameterization of the path
$\gamma$ in the sense that if
$\tau:[0,1]\raw[0,1]$ is an orientation-preserving \homeo, then
$F(\gamma\circ\tau) = F(\gamma)$. Also note that if 
$\tau$ is an orientation-reversing \homeo, then
$F(\gamma\circ\tau) = -F(\gamma)$. In addition, the
inclusion of the requirement that $F$ be continuous implies that 
if $\gamma_n$ is a sequence of paths converging
to a constant function then $F(\gamma_n)\raw 0$. \

To define cohomologous
path cocycles  we require the appropriate notion
of an exact path cocycle. For a continuous function $\chi:M\raw \F$, let 
$\delta\chi:\cP\raw\F$ be defined by
$\delta\chi(\gamma) = \chi(\gamma(1)) - \chi(\gamma(0))$. It is
easy to check that $\delta\chi$ satisfies Definition~\ref{pathcocycledef}.
We say that the two path cocycles $F_1$ and $F_2$ are 
cohomologous if $F_1 = F_2 + \delta\chi$ for some
continuous $\chi\in C^0(M, \F)$.

\subsection{Cover cocycles}
Again we start with the construction of this
cocycle from a
given closed one-form $\omega$.
Pick $\tx, \ty\in\tM$ and let 
\begin{equation}\label{coverform}
\beta_\omega(\tx, \ty) = \int_{\tgamma} \tomega,
\end{equation}
where $\tgamma$ is any smooth path in $\tM$ connecting
$\tx$ and $\ty$ and $\tomega$ is a pull back to
$\tM$ of $\omega$. Since $\omega$ is closed,
the definition is independent of the
choice of path. Further, $\beta_\omega:\tM\times\tM\raw \F$
 is invariant under the diagonal action of 
$\Z^d$,  
$\beta_\omega(\delta_{\vn} \tx, \delta_{\vn} \ty) = \beta_\omega(\tx, \ty)$
and is additive in the sense that 
$\beta_\omega(\tx, \ty) + \beta_\omega(\ty, \tz) = \beta_\omega(\tx, \tz)$
for all $\tx, \ty, \tz \in\tM$. These last two properties
are used the define a general such cocycle.

\begin{definition}[Cover cocycle]\label{covercocycledef}
A \de{cover cocycle} over $\F$ is a continuous
 map $\beta:\tM\times\tM\raw \F$
 which is
\begin{compactenum}
\item Equivariant under the diagonal action of the
deck group: so $\beta(\delta_{\vn} \tx, \delta_{\vn} \ty) = \beta(\tx, \ty)$,
for all $\vn \in  H_1(M;\Z)$ and all $\tx, \ty\in\tM$.
\item Additive: 
\begin{equation}\label{add}
\beta(\tx, \ty) + \beta(\ty, \tz) = \beta(\tx, \tz)
\end{equation}
for all $\tx, \ty, \tz \in\tM$.
\end{compactenum}
\end{definition}

Note that if as with Alexander-Spanier cocycles (see \cite{spanier})
we define the differential of the cover cocycle $\beta$ as
$\delta\beta:\tM\times\tM\times\tM\raw \F$ defined as
\begin{equation*}
\delta \beta (\tx_1, \tx_2, \tx_3) = \beta(\tx_2, \tx_3) 
-\beta(\tx_1, \tx_3) + \beta(\tx_1, \tx_2),
\end{equation*}
then condition \eqref{add} just says that $\delta\beta=0$,
\ie\ $\beta$ is closed or is a cocycle.

The appropriate notion of cohomologous cocycles requires
a definition of an exact cocycle. For a continuous function $\tchi:\tM\raw\F$,
define $d \tchi (\tx, \ty)  := \tchi(\ty) - \tchi(\tx)$. If
$\tchi(\delta_{\vn} \tx) = \tchi(\tx)$ for all $\vn\in\Z^d$,
or equivalently, if $\tchi = \chi\circ\pi$ for some $\chi\in C^0(M;\F)$,
then $d\tchi$ is a cover cocycle. We then 
say that two path cocycles $\beta_1$ and $\beta_2$ are
cohomologous if and only if 
\begin{equation*}
\beta_2 = \beta_1 + \chi\circ\pi,
\end{equation*}
for some $\chi\in C^0(M;\F)$. 


\subsection{Correspondences and isomorphism to de Rham theory}
Not surprisingly there is a simple correspondence 
between path and cover cocycles, and this correspondence
respects cohomology classes. The collection of these 
cohomology classes has a natural structure
as a vector space over $\F$ which is isomorphic to
usual first cohomology.

\begin{fact}\label{pathcover}
There is a natural bijection
between path cocycles and cover cocycles, and the
collection of their cohomology classes
is isomorphic to first de Rham cohomology $H^1_{DR}(M;\F)$.
\end{fact}

\begin{proof}
Given an path cocycle $F$, define $\beta:\tM\times\tM\raw \F$ by
$\beta(\tx, \ty) = F(\pi\circ \tgamma)$ where $\tgamma$ is any path
in $\tM$ connecting $\tx$ to $\ty$. If we choose a different
path $\tgamma'$ connecting $\tx$ to $\ty$, then 
$\pi\circ\tgamma$ and $\pi\circ\tgamma'$ are homologous
in $M$, and so $F(\pi\circ \tgamma') = F(\pi\circ \tgamma)$, and
so $\beta$ is well-defined. The map $\beta$ is clearly equivariant
under the diagonal action of the deck group of $\tM$, and is
additive since $F$ is. 
  
Conversely, given a cover cocycle $\beta$, for $\gamma$ a path in
$M$, define $F(\gamma) = \beta(\tgamma(0), \tgamma(1))$ where
$\tgamma$ is any lift of $\gamma$ to $\tM$. Now $F$ is
independent of the choose of lift $\tgamma$ because
$\beta$ is  equivariant under the diagonal action of the deck
group of $\tM$. If $\gamma$ and $\gamma'$ are homologous
in $M$, then they lift to $\tgamma$ and $\tgamma'$ in
$\tM$ with the same pair of endpoints, and so 
$F(\gamma) =  F(\gamma')$. Finally, the additivity of
$F$ follows from that of $\beta$. 

Under the correspondences just delineated, it is easy to
see that exact path cocycles correspond to exact cover cocycles
and so it yields an isomorphism between the vector spaces
of cohomology classes of the cocycles. 

To prove the second statement of the fact, 
if $\omega$ is a closed one-form on $M$, and $\gamma\in\cP$
is smooth, define the path cycle 
$F_\omega(\gamma)$ as in \eqref{pathform}.
 Now if
$\chi:M\raw\F$ is smooth, then its usual exterior derivative
is the closed one-form $d \chi$  which 
yields the path cocycle $F_{d \chi}$ which is the same
the exact path cocycle  $\delta\chi$ defined below
 Definition~\ref{pathcocycledef}. Thus the map 
$\omega\mapsto F_\omega$ induces  a map
from $H^1_{DR}(M;\F)$ to the vector space of
cohomology classes of path cocycles which is
clearly a homomorphism. 

Finally, given a cohomology class of
path cocycles over $\F$ we will associate it with 
a linear functional $\F^d\raw\F$ and thus
a class in $H^1_{DR}(M, \F)$ by the usual identifications.
Given an path cocycle $F$, define $\Phi_F:H^1(M;\Z)\raw \F$
as $\Phi_F([\Lambda]) = F(\Lambda)$
where $\Lambda$ is a closed loop in $M$ and $[\Lambda]$ its
class in $H_1(M, \Z)$.  Let $\Phi_F$ also
denote the extension to a linear functional
$\F^d\raw\F$. By Definition~\ref{pathcocycledef}(b),
the definition of $\Phi_F$ is independent of
choice of $\Lambda$ in its homology class. For an exact
path cocycle $\delta\chi(\Lambda) = 0$ for any loop
$\Lambda$, and so $\Phi_{\delta\chi}=0$. Thus
$\Phi_F$ depends only on the cohomology class of $F$.
For a closed one-form $\omega$, and a smooth
loop  $\Lambda$, since  $F_\omega(\Lambda) = \int_{\Lambda} \omega$,
$\Phi_{F_\omega}$ represents the cohomology class of
$\omega$, which implies that homomorphisms are isomorphisms,
finishing the proof.\QED\
\end{proof}

In view of the equivalence of path and cover cocycles,
we will often just call them ``cocycles'' with the type
clear from the context.

\subsection{Topological one-forms}
There is yet another structure generalizing
closed one-forms which will be of use here.  
 Shub calls it a translational H-structure 
(with ``H'' for Haeflinger) \cite{shub}  and
Farber, et. al in \cite{farber1, farber2} call it a topological one-form.
In the context of this paper these represent the analog for
path cocycles of the Poincar\'e lemma, \ie\
closed one-forms are locally exact.

\begin{definition}[Topological one-form]\label{toponedef}
A translational H-structure or topological one-form consists of
a collection of pairs $(U_i, f_i)$ where
\begin{compactenum}
\item $\{U_i\}$ is a finite open cover of $M$ by open
topological disks $U_i$,
\item  the continuous maps
$f_i:U_i\raw\F$ satisfy the overlap condition that whenever
$U_i\cap U_j \not=\emptyset$ there exists constants $r_{ij}\in\F$ with
$f_i - f_j = r_{ij}$  on each connected
component of  $U_i\cap U_j$.
\end{compactenum}
\end{definition}

The papers \cite{farber1, farber2} prove the existence of
Lyapunov topological one-forms for certain flows or 
subflows, a result that is
complementary or perhaps ``orthogonal''
to the  existence of eigen-cocycles. 
The following is noted  in \cite{farber1}.

\begin{fact}\label{onepath} A topological one-form yields a
unique path cocycle. A path cocycle and a finite open cover of $M$
by topological disks gives a topological one-form.
\end{fact}

\begin{proof}
Assume we are given a topological one-form $\{(U_i, f_i)\}$.
For each $i$ we define  a local path 
cocycle $\delta f_i$ defined on a path with $\gamma([a,b])\subset U_i$ by   
 $\delta f_i(\gamma) = f_i(\gamma(b)) - f_i(\gamma(a))$.
Now given a general path $\gamma:[0,1] \raw M$,
 we may find a
a subdivision $0 = t_0 < t_j < \dots < t_n = 1$, so that for each
$i$ there is a $i(j)$ with
$\gamma([t_j, t_{j+1}]) \subset U_{i(j)}$. Let $\gamma_j$ be
$\gamma$ restricted to $[t_j, t_{j+1}]$ and define
$F(\gamma) = \sum \delta f_{i(j)}(\gamma_j)$. The overlap 
conditions on the topological one-form 
imply that $F$ is additive in the sense of 
Definition~\ref{pathcocycledef}(a). In addition,
a standard argument   breaks a 
large homotopy  into a sequence of smaller ones, each
of which moves across  just one overlap $U_i\cap U_j$ at a time.
This yields that $F$ is homotopy invariant 
in the sense that when $\gamma$ and $ \gamma'$
have the same endpoints and are homotopic, then 
$F(\gamma) = F(\gamma')$.  

It remains to show  that
$F$ is homology invariant as in Definition~\ref{pathcocycledef}(b).
This requires a standard argument whose algebraic content
is that a homomorphism from a group $G$ with Abelian image
always factors through the Abelianization of $G$. 
In the current context,  the simplest approach
is to note that since $F$ is homotopy
invariant, it defines a continuous map $\hbeta:\hM\times\hM\raw\F$ where
$\hM$ is the \textit{universal cover} of $M$ by letting
$\hbeta(\hx, \hy) = F(\hat{\pi}\circ\tgamma)$ where 
$\tgamma$ is any path connecting $\hx$ to $\hy$ and
$\hat{\pi}:\hM\raw M$ is the cover. Further,
$\hbeta(g \cdot \hx, g \cdot \hy) = \hbeta(\hx, \hy)$,
for all $g \in  \pi_1(M)$ (treated as deck transformations of
the universal cover) and 
 $\hbeta(\hx, \hy) + \hbeta(\hy, \hz) = \hbeta(\hx, \hz)$
 for all $\tx, \ty, \tz \in\hM$. It is easy to check that
these properties imply that 
$\hbeta(\hx, (h\I g h)\cdot \hy) =  
\hbeta(\hx,  g \cdot \hy)$ and $\hbeta(\hx, (h\I g\I h g)\cdot \hy) =  
\hbeta(\hx,  \hy)$.

 It is now straightforward to use $\hbeta$ to produce a
cover cocycle on the product of the homology covers
$\beta:\tM\times\tM\raw\F$.
Define
$\beta(\tx, \ty) := \hbeta(\hx, \hy)$ where 
$\hx$ and $\hy$ satisfy $p(\hx) = \tx$ and $p(\hy) = \ty$
where $p:\hM\raw\tM$ is the projection. 
To show this definition is independent of these choices,
we  assume that 
$p(\hx') = \tx$ and $p(\hy') = \ty$. This implies there are 
$g_x, g_y \in [\pi_1(M), \pi_1(M)]$ with 
$g_x \cdot \hx = \hx'$ and $g_y \cdot \hy = \hy'$.
Thus by the properties given at the end of
the previous paragraph,  $\hbeta(\hx', \hy') 
= \hbeta(\hx, \hy)$, as required. It now follows easily
that $\beta$ is a cover cocycle in the sense of
Definition~\ref{covercocycledef} on the universal Abelian
cover $\tM$. Thus using the correspondence of
path and cover cocycles, we see that $F$ is an path cocycle.

Now given an path cocycle $F$ and a finite open cover $\{U_i\}$
by simply connected open sets,
for each $i$ pick a base point $x_i\in U_i$ and for 
$x\in U_i$ let $f_i(x) = F(\gamma_x)$ where $\gamma_x$
is any path connecting $x_i$ to $x$. It is
easy to check that the resulting $\{(U_i, f_i)\}$ is 
a topological one-form. \QED
\end{proof}

\subsection{Eigen-cocycles and eigen-\cmaps}\label{cocyclecmap}
In this subsection we make explicit
the connection between the cocycles and \cmaps,
define eigen-cocycles, 
and then restate Theorem~\ref{mainthm}
in terms of cocycles. It is easiest to connect \cmaps\ with
cover cocycles; one then obtains the
corresponding  cocycles using Fact~\ref{pathcover} .

Given a \cmap\ $\sigma$, then 
 $\beta(\tx, \ty) = \sigma(\ty) - \sigma(\tx)$. 
is a cover cocycle. 
Conversely, given a cover cocycle $\beta$, fix a base point
$\tx_0$ and define 
\begin{equation*}
\sigma(\tx) = \beta(\tx_0, \tx).
\end{equation*}
 Now for any $\vn\in\Z^d$,
$\sigma(\delta_{\vn}\tx) = 
\beta(\tx_0, \dvn\tx) = 
\beta(\tx_0, \tx) + \beta(\tx, \dvn\tx)
= \beta(\tx_0, \tx) + \Phi_\beta(\vn),$
with $\Phi_\beta$ the linear functional representing 
the cohomology class of $\beta$ as
in \S\ref{linrep}. Thus, $\sigma$ is a \cmap\ where $c$ is the
cohomology class of $\beta$.
Note that  changing the base point adds a constant to $\sigma$.

To define eigen-cocycles we need the action
of a map $f$. Given a continuous $f:M\raw M$ for
the action on cover cocycles, 
pick a lift $\tf$ of $f$ and 
let $f^*\beta(\tx, \ty) = \beta(\tf\tx, \tf\ty)$. Since
any other lift of $f$ can be written as $\delta_{\vn}\tf$
for some $\vn\in\Z^d$, the definition
is independent of choice of the lift of $f$. In particular,
it defines an action of $f$ itself on cover cocycles
 whereas it is a specific lift of $f$ that
acts on  \cmaps.  For an path cocycle $F$, let 
$f^*F(\gamma) = F(f\circ\gamma)$. 

The cover cocycle $\beta$ is  called \de{eigen with 
factor $\mu$} if $f^*\beta = \mu \beta$, and the path 
cocycle $F$ is eigen if  $f^*F = \mu F$. 
The correspondence of these eigen-cocycles to eigen-\cmaps\ 
requires one more notion. A \cmap\ $\talpha$ is called
 \de{almost eigen} for the lift $\tf$ if  
$\talpha\circ \tf = \mu \talpha + K$ for some $K\in\F$.

Given an almost eigen-\cmap\ $\talpha$, it follows easily that 
$\beta(\tx, \ty) = \talpha(\ty) - \talpha(\tx)$ is 
an eigen-cocycle. 
Conversely, given an eigen-cocycle $\beta$  with factor $\mu$,
fix a base point $\tx_0$ and define the corresponding 
\cmap\ as in Fact~\ref{pathcover}:
$\talpha(\tx) = \beta(\tx_0, \tx)$.  Choose a lift $\tf$, and then
\begin{equation*}
\begin{split}
\talpha(\tf\tx) &= \beta(\tx_0, \tf\tx)\\
&= \beta(\tf\tf\I\tx_0, \tf\tx)\\
&= \mu\beta(\tf\I\tx_0, \tx)\\
&= \mu (\beta(\tf\I\tx_0, \tx_0) + \beta(\tx_0, \tx))\\
&=  K + \mu \talpha(\tx), 
\end{split}
\end{equation*}
where $K =  \mu \beta(\tf\I\tx_0, \tx_0) = 
\mu \beta(\tf\I\tx_0, \tf\I\tf\tx_0) = \beta(\tx_0, \tf\tx_0)$. 
 Changing the choice
of lift or base point changes the constant $K$.

We see then that an eigen-cocycle yields an almost eigen-\cmap.
The fact that one gets just an 
\textit{almost} eigen-\cmap\ is awkward as is the fact that the
\cmap\ is  eigen
for a lift and not for $f$ itself. At least the first
awkwardness is easily remedied in the 
``generic'' case.
If $\talpha$ is an almost eigen-\cmap\ with constant $K$ and
factor $\mu\not=1$, 
let $\hat\alpha = \talpha + \frac{K}{\mu-1}$. We then have 
$\hat\alpha\circ \tf = \mu\hat\alpha$, and so an almost
eigen-\cmap\ yields an actual eigen-\cmap. A factor 
$\mu = 1$ can only occur under very
special circumstances and is considered in
Lemma~\ref{simplelem}. 

Thus using the obvious definition of an eigenchain
of cocycles we have the following corollary to 
Franks-Shub Theorem~\ref{mainthm} 
\begin{corollary}\label{eigencocycle}
If $f:M\raw M$ is a continuous map of the smooth, connected manifold
$M$  and $\mu\in\F$ is an eigenvalue of $f^*:H^1(M;\Z) \raw H^1(M;\Z)$
with $|\mu| > 1$ and eigenchain
 $\{c_1, \dots, c_k\}\subset H^1(M;\F)$, 
then $f$ has   eigenchains of 
path cocycles and cover
cocycles with factor $\mu$ which 
represents the classes $\{c_1, \dots, c_k\}$.
\end{corollary}

The next lemma collects a few simple results
about the connection of dynamics and 
eigen-cocycle factors which we need in the sequel.

\begin{lemma}\label{simplelem} Assume $f:M\raw M$ has an 
eigen-cocycle or generalized eigen-cocycle
$\beta$ with factor $\mu$, and let 
 $\td$ be an equivariant metric on $\tM$. 

\begin{compactenum}
\item   $f$ has an eigen-cocycle with factor $\mu=1$
if and only if $f$ is semiconjugate to
rigid rotation on the unit circle $S^1$.

\item If $\mu \geq 1$ and $\talpha$ is an eigen-
or generalized eigen-\cmap\
corresponding to $\beta$ and $\tx, \ty\in\tM$ are such that
$\td(\tf^n\tx, \tf^n\ty) \raw 0$ as $n\raw\infty$,
then $\talpha(\tx) = \talpha(\ty)$. If $\mu > 1$ and
$\td(\tf^n\tx, \tf^n\ty)$ is bounded as $n\raw\infty$,
then $\talpha(\tx) = \talpha(\ty)$.

\item If $f$ is a \homeo\ and
 $\mu \leq 1$ and $\talpha$ is an eigen-
or generalized eigen-\cmap\
corresponding to $\beta$ 
and $\tx, \ty\in\tM$ are such that
$\td(\tf^n\tx, \tf^n\ty) \raw 0$ as $n\raw-\infty$,
then $\talpha(\tx) = \talpha(\ty)$. If $\mu < 1$ and
$\td(\tf^n\tx, \tf^n\ty)$ is bounded as $n\raw -\infty$,
then $\talpha(\tx) = \talpha(\ty)$.

\end{compactenum}
\end{lemma}
\begin{proof}
Assume that $f$ has an eigen-cocycle $\beta$ with factor
$\mu=1$. This means that $f^*$ acting on
$H^1(M;\Z)$ has an eigenvalue of $1$ and so we
may find an integral eigen-class $c$ which is 
represented by a linear $\Phi_c:\Z^d\raw\Z$. 
Now if $\talpha:\tM\raw\R$ is an almost
eigen-\cmap\ constructed from $\beta$ as in Corollary~\ref{eigencocycle},
then since $\talpha(\delta_{\vn} \tx) = \talpha(\tx) + \Phi_c(\vn)$,
$\talpha$ descend to a map $\overline{\alpha}:M\raw S^1$. 
Since $\talpha\tf = \talpha + K$,  $\overline{\alpha}$ gives 
a semiconjugacy between $f$ and rigid rotation by $K$ on
the circle $S^1$. Conversely, a semiconjugacy 
$\overline{\beta}:M\raw S^1$ between $f$ and rigid rotation
will lift to an almost eigen-\cmap\ $\talpha:\tM\raw\R$ with
factor $1$  which in turn yields the desired eigen-cocycle,
finishing the proof of (a).

We prove (b) under the assumption that
$\talpha$ is an eigen-\cmap. It is an easy induction
on the eigenchain entries to then get the result for
generalized eigen-\cmap.  By the semiconjugacy we have 
$|\mu|^n|\talpha(\tx) - \talpha(\ty))| = 
|\talpha\tf^n\tx - \talpha\tf^n\ty|$.
Now if  $\td(\tf^n\tx, \tf^n\ty)\raw 0$ 
as $n\raw\infty$, then by continuity of
$\talpha$, $|\talpha\tf^n\tx - \talpha\tf^n\ty|
\raw 0$, and since 
$|\mu|\geq 1$, the only
possibility is $\talpha(\tx) = \talpha(\ty)$.
For the second sentence of (b), note that
if $\td(\tf^n\tx, \tf^n\ty)$ is bounded,
 then since $\talpha$ is large scale
Lipschitz (Remark~\ref{largelip}), 
 $|\talpha\tf^n\tx - \talpha\tf^n\ty|$
is also bounded and so if
 $|\mu|> 1$, we have $\talpha(\tx) = \talpha(\ty)$
again. For the proof of (c), use 
the same argument with $\tf^{-1}$.
\QED\
\end{proof}

\begin{remark}\label{factorrk}
 If $f^*$ has an eigenvalue
with  $|\mu| > 1$, then using the corresponding
 eigen-\cmap\ one gets restrictions on the dynamics
of any lift of $f$ to the homology cover $\tM$. For example,
no lift $\tf$ to $\tM$ has a dense orbit. This
was proved in \cite{bdtrans} using  different
methods. Also, if $\tf$ has a local product structure
in the sense of hyperbolic dynamics, then $\tf$ 
has no eigen-\cmaps\ with factor $|\mu| = 1$. Note that
this certainly does not exclude $f^*$ having eigenvalues
of modulus one.
\end{remark}

\begin{remark} As in \eqref{cmapsum} there
is a summation formula for eigen-cocycles.
Given an eigen-class $c\in H^1(M;\F)$ with eigenvalue
$\mu$, pick a cover cocycle $\beta$ which represents the
class. Since $f^* c = \mu c$, using the isomorphism between
cover cocycles and usual first homology,  
$f^* \beta = \mu \beta + d\tchi$,
for some $\tchi = \chi\circ \pi$ with $\chi\in C^0(M;\F)$.
If $|\mu|>1$, 
\begin{equation*}
\beta_\mu := \beta + \sum_{j=1}^\infty 
\frac{d \tchi\circ\tf^{j-1}}{\mu^j},
\end{equation*}
converges by Weierstrass $M$-test, and   
 $f^* \beta_\mu = \mu \beta_\mu$ by direct
verification. A similar sum can obviously be given
 for eigen-path cocycles.
\end{remark}

\section{PseudoAnosov maps}\label{pAsec}
For the balance of this paper we will focus on
\pA\ \homeos\ of compact surfaces. See 
\cite{Thurston, FLP, Casson, Bdamster} for more
information and details. 

\subsection{Foliations, measures and metrics}\label{leafinfo}
A \pA\ \homeo\ $\phi$ is characterized by a pair of transverse
measured foliations, one termed stable, $\cF^s$,
and one unstable, $\cF^u$. The foliations are allowed
to have a finite number of singularities or prongs
of a controlled type and the foliation near a boundary
component looks like a blown up prong. In this paper we will also
include the case of \pA\ maps relative to a finite
set in the class of \pA\ maps. These so-called
rel \pA\ maps are allowed to have a finite number
of one-prongs at specified points and the
isotopy class of such a $\phi$ is always considered
relative to these points.

Let $P$ be the collection of singular points of the foliations.
For $x\in M-P$, let $L^s(x)$ be the leaf of $\cF^s$ which
contains $x$. If $x\in P$, then by convention  $L^s(x) = \{x\}$.
If $x$ is in a leaf that is
associated with a singularity,  then by $L^s(x)$ we
mean the half-infinite leaf which``begins'' at the
singularity. In this case the \de{extended leaf} containing $x$ 
is the union of all the leaves associated with the 
singularity.

A basic fact is that every infinite and half-infinite
leaf of a \pA\ invariant foliation is dense in the surface.
The lifts of the foliations to the universal Abelian cover
$\tM$ are denoted $\tF^s$ and $\tF^u$. In the cover
the leaf or half-infinite leaf containing $\tx$ is denoted
$\tL^s(\tx)$ or $\tL^u(\tx)$. 

A \pA\ map always has a Markov partition with a
\PF\ transition matrix. The \PF\ eigenvalue
of this matrix is usually denoted 
$\lambda$ and called the 
\de{dilation, expansion factor or stretch factor}.
As described in the introduction, $\lambda$
arises in many different contexts in the study of
\pA\ maps.

Each of the foliations carries a holonomy invariant
transverse measure constructed from 
the \PF\ eigenvalue, $\lambda$, and eigenvector of the transition
matrix of a Markov partition
 (\cf\ Remark~\ref{PFmeas}).  An arc $\gamma$ that is transverse
to $\cF^s$ is assigned a measure $m^u(\gamma)$ which
satisfies $m^u(\phi(\gamma)) = \lambda m_u(\gamma)$, and
when $\gamma$ is transverse
to $\cF^u$ one has  $m^s(\phi(\gamma)) = \lambda\I m^s(\gamma)$.
The reader is cautioned that there is a fair
amount of diversity in the literature in
the assigning of the labels "stable" and
"unstable" to structures associated with
a \pA\ map.

Using the transverse measures one constructs a 
topological metric $d_\phi$
based on the arc metric which is defined on short arcs transverse
to both foliations by $\sqrt{(m^u)^2 + (m^s)^2}$.
It is common to express these metrics in terms of
a flat structure with conic singularities.

The lift of the  metric $d_\phi$ to the universal
Abelian cover $\tM$ is denoted $\td_\phi$.
A useful property is that 
$\tx$ and $\ty$ are in the same extended
leaf of $\tF^s$ if and only if 
 $\td_\phi(\tphi^n(\tx), \tphi^n(\ty))\raw 0$ as $n\raw\infty$. 

Again for simplicity of exposition we restrict to the
case of orientation preserving \pA\ maps on orientable
surfaces.

\subsection{Rectangles and Markov partitions}\label{rectdef}
At this point we fix once and
for all a given \pA\ map $\phi$ with its pair of
transverse measured foliations  $\cF^s$ and 
$\cF^u$. A \de{rectangle}, $R$, is a topological disk whose
boundary consists of $4$ segments which are alternately
arcs in leafs of  $\cF^s$ and $\cF^u$. These arcs are 
called the stable and unstable edges of $R$, and if we
are considering an indexed set of rectangles (eg. a Markov
partition)  $\{R_i\}$, then the edges are denoted 
 $E_{i, 1}^s, E_{i, 1}^u, E_{i, 2}^s$, and  $E_{i, 2}^u$ . We also
allow a singularity to be a common endpoint of 
two of these arcs, \ie\ a rectangle
can have a singularity at a ``corner''. The interior of rectangle is
always a chart for both foliations, \ie\ there is
a \homeo\  $Int(R)\raw (0,1)^2$ which takes $\cF^s\cap Int(R)$ to
the vertical foliation of the open unit square and 
$\cF^u\cap Int(R)$ to the horizontal.

An \de{oriented rectangle}, $R$, is a rectangle with the additional
data of a \homeo\ $h:R\raw [0,1]^2$ so that $h$ restricted to
$Int(R)$ gives a chart for both foliations as just described.
Pulling back allows us to speak of the top, bottom, left and
right edges of $R$, clockwise rotation about the boundary
of $R$, etc.
\begin{definition}[Cover by rectangles]\label{recdef}
A \de{cover of $M$ by 
rectangles} is a finite collection of rectangles 
$\{R_i\}$ with 
\begin{compactenum}
\item $\cup R_i = M$,
\item When $i\not = j$, $Int(R_i) \cap Int(R_j) = \emptyset$,  and 
  if for some $i\not= j$, $Fr(R_i) \cap Fr(R_j)$ is nonempty, 
then it is connected and is contained 
in exactly one edge of each of $R_i$ and $R_j$.
\end{compactenum}
\end{definition}
 Note that this
implies that all singularities are in the boundary of
some $R_i$. 

A Markov partition for a \pA\ is a special cover
by rectangles which has nice dynamical properties. 
A construction of covers by rectangles and a Markov
partition is given in \cite{FLP}, expos\'e 10.

We may assume that for our given \pA\ map $\phi$ we have chosen
a Markov partition $\{R_i \}$ for $i = 1, \dots, d$ that is
fine enough so that each $\phi(R_j) \cap R_i$ contains at most one
connected component. Thus its transition matrix 
$A$ is a $d\times d$ matrix with entries of
either $0$ or $1$  with 
$A_{ij} = 1$ if and only if
$\phi(R_j) \cap R_i \not=\emptyset$. 

As is usual, the matrix $A$ generates a two-sided
subshift of finite type $\Lambda_A$ and there is 
a semi-conjugacy $(\Lambda_A, \sigma) \raw (M^2, \phi)$
given by the ``address map''. We shall mainly
be concerned here with the one-sided shift space
built from $A$ as described in \S\ref{sed} below.
For a \pA\ map the matrix $A$ always satisfies 
$A^n > 0$ for all $n$ larger than some $N$ which
implies that the associated shifts are topologically
transitive and that the \PF\ theorem holds for $A$.

\subsection{The \pA\ spectrum}
In this subsection we survey what is known
about the structure and uses of the \pA\
spectrum. Most of this material is not used
in the sequel, but it  provides a valuable context.

We use the term \pA\ spectrum to encompass
both the spectrum of the action of the \pA\ map
on first homology and the spectrum of  Markov
transition matrices. When a distinction is needed, 
the first is called the \de{homological spectrum}
and the latter the \de{Markov spectrum}.
When $\phi$ has orientable
foliations a theorem of Rykken (\cite{rykken}) implies that the
hyperbolic portions of these spectra agree 
(\cf\ Remark~\ref{rykremark}).
In the general case they each may contain 
hyperbolic eigenvalues not contained in the other.
While one may always lift a \pA\ to a branched
cover with orientable foliations,
it is usually not very straightforward to track
the influence of the lift on the spectra, and so
in most cases it is necessary to consider the
two spectra separately.

We first discuss the homological spectrum.
A homeomorphism $f$ of a closed surface 
preserves the homological intersection form of closed
curves. In a standard basis this form is
the standard symplectic form and
so the matrix of $f_*$  on $H_1(M;\Z)$ is symplectic and so
has a palindromic characteristic
polynomial. Since boundary components of the surface
are permuted by a \homeo, their presence only contributes
roots of unity to the homological spectrum and so
the homological characteristic polynomial as a whole
is always palindromic.  This implies that if $\mu$ is an eigenvalue, then
so are $\mu\I, \overline{\mu}$ and $\overline{\mu}\I$.
In particular, since characteristic polynomials
are monic, elements of the homological spectrum 
are algebraic units.

In terms of the associated eigenvectors on homology/cohomology, 
there are a number of well-known 
interpretations of an \textit{oriented}
 measured foliation on a surface  as a cycle or
cocycle.  For example, the measured foliation generates 
a geometric current as in Ruelle and Sullivan \cite{ruellesullivan}.
Almost equivalently, one may flow along the foliation 
and the transverse measure induces a flow-invariant ergodic measure.
This measure can be assigned   a Schwartzman  asymptotic cycle
(\cite{Schwartzman}) which then represents the oriented measured foliation.
In addition, as noted in the introduction, there is
a closed one-form whose kernel is tangent
to the foliation and integration of the form 
along transverse arcs yields the transverse measure. Finally, 
a oriented measured foliation can be represented as a 
weighted oriented train track which can be interpreted as a
real homology chain.

Now given a \pA\ map, $\phi$,  with orientable foliations,
we may identify, say $\cF^u$ with a cycle.
Under the action of $\phi$, this cycle scales by a 
factor of  $\lambda$ if $\phi$ preserves the orientation
of the foliations and $-\lambda$ if it reverses it.
 Thus the  eigenvector corresponding
to $\pm\lambda$ is concretely represented by the cycle
of $\cF^u$. On implication is that when $\phi$ has
orientable foliations, the spectral radius of the action
on first homology is $\lambda$. The converse is also true,
but less well known, see \cite{bb} for a proof.  
Proposition~\ref{tcfact} below gives concrete realizations
 as eigen-cocycles for the eigenvectors of 
 the rest of the hyperbolic homological spectrum.

To discuss the Markov spectrum we must start by pointing
out that it has not yet been properly defined since
 each \pA\ map has infinitely
many different Markov partitions with each having its own
 spectrum. We therefore define the Markov spectrum as
the intersection of all the spectra of all these
Markov partitions. Since all Markov partitions 
give rise to subshifts of finite type with essentially
the same dynamics, results 
from symbolic dynamics (see \cite{kitchens})
 may be used to show that
all the Markov partitions yield   the same hyperbolic
spectrum (in addition to any   Galois conjugates of hyperbolic
elements). 

Birman et al (\cite{birman}) show that the characteristic polynomial
of any Markov matrix for a \pA\ map must be
palindromic or anti-palindromic with perhaps an
additional factor of $x^n$, and so as with the homological
spectrum, all elements of the Markov spectrum 
are algebraic units.   It is worth noting that
individual factors of the characteristic 
polynomial over $\Z$ do not in themselves
have to be palindromic. In particular,  $\lambda$ 
does not have to be a reciprocal algebraic unit,
\ie\ $\lambda$ and $\lambda\I$ do not have to share the same
minimal polynomial.  In simple examples
the case of reciprocal $\lambda$ is most
common;   see, for example, \cite{arnouxfathi} and  \cite{fried} 
 for  non-reciprocal examples.

Since the \pA\ map's dynamics are coded
by the corresponding subshift of finite
type, the  results from symbolic
dynamics concerning the spectra of   transition
matrices apply after a few provisos to \pA\ maps.
Let us fix a Markov partition with matrix $A$
which has spectrum $ \lambda > |\mu_2| \geq \dots
\geq |\mu_{d-2}| > \lambda\I$.

A simple result in 
symbolic dynamics says that $\trace(A^n)$ counts the number
of fixed points of the $n^{th}$ iterate of
the shift. Since $\trace(A^n) = \lambda^n + |\mu_2|^n 
+ \dots + \lambda^{-n}$, one has that the primary 
exponential growth rate of periodic points in $\Lambda^+_A$
is $\lambda$ with the rest of the spectrum providing 
 correction terms. The construction  of the
Markov partitions for a \pA\ map guarantees that
there are only a finite number of periodic orbits
of $\phi$ which are multiply-coded by the symbolic
model, and so  there is a constant $C$ (which depends on
the Markov partition)  with
$|\trace(A^n) - \# \Fix(\phi^n) | < C$ for all $n$.

We also recall  that $\phi$ acts
on the weights carried by its forward invariant train track
$\tau$ via $A^T$.  Thus the spectrum of $A$  describes 
how  laminations carried $\tau$ converge to the
unstable lamination of $\phi$ and to
other subspaces in the generalized eigen-flag of $A$.
After some work to handle   laminations near the
unstable lamination of $\phi$ but not 
carried by $\tau$, this implies that for the dynamics
induced by  $\phi$ on the boundary of Teichm\"uller
space, orbits converge (in the Thurston metric)
to the fixed point corresponding 
to $\cF^u$ at a slowest rate of  $|\mu_2/\lambda|^n$.

A final note on the \pA\ spectrum: it is certainly possible
that a given \pA\ map has only one unstable and one stable
eigenvalue in its spectrum. In this case either $\lambda$
is quadratic or else is a Salem number, and $\phi$'s foliations
lack all the additional transverse structures described here. It would 
be very interesting to understand what this implies about the
geometry and dynamics of $\phi$.

\section{Transverse  cocycles for \pA\ maps}\label{tcsec}
Given a \pA\ map and its invariant foliations,
we define a special kind of  path cocycle, called a \de{transverse
cocycle}, which is adapted to
to the stable foliation $\cF^s$ in the sense that the cocycle only
sees the part of paths which are
transverse to $\cF^s$. These cocycles are static objects connected
with the foliations, but the main result of this section connects 
them
to the dynamics by showing that the collection of
all transverse cocycles is a vector space which is
spanned by the expanding eigen-cocycles.

\begin{definition}[Transverse cocycle]
An path cocycle $F$ is said to be \de{transverse to $\cF^s$},
if $F(\gamma) = 0$ for all paths $\gamma$ whose image is contained in
leaves of $\cF^s$.  A \cmap\ $\sigma$ is said to be 
\de{transverse to $\cF^s$} if $\tx \in \tL(\ty)$ implies
that $\sigma(\tx) = \sigma(\ty)$.
\end{definition}
It will be usually be the case that there is
a fixed \pA\ map and stable foliation under consideration
in which case $F$ and $\sigma$ are just called a 
transverse path cocycles and \cmaps.
By the additivity
of  path cocycles, if two homotopic paths $\gamma_1,\gamma_2$ 
have their initial points on the same leaf of $\cF^s$ and their
final points on the same leaf of $\cF^s$, then $F(\gamma_1) = F(\gamma_2)$.
Thus transverse path cocycles are said to be  holonomy
invariant. Also
note that the natural correspondence of path cocycles and
\cmaps\ given in \S\ref{cocyclecmap} respects the property of being
transverse.

Let $\cR$ be the collection of transverse cocycles to $\cF^s$ with
values in $\F$; it
is immediate that $\cR$ is a vector space over $\F$.
The next fact says  that this space is isomorphic to the
unstable subspace of $\phi^*$ acting on $H^1(M;\F)$.
 It is
worth emphasizing that the elements of $\cR$ are individual
cocycles, not their cohomology classes. Thus, in particular,
the following fact says there is a unique transverse cocycle in
each unstable cohomology class. 
\begin{proposition}\label{tcfact} With the various structures
as defined above we have the commutative diagram:
\begin{equation*}
\begin{CD}
\cR @>\cong>> Un(\phi^*, H^1(M;\F)) \\
 @V\phi^*VV @VV\phi^*V \\
\cR @>\cong >> Un(\phi^*, H^1(M;\F))%
\end{CD}
\end{equation*}
This implies that the eigen-objects and their factors
correspond.
\end{proposition}
\begin{proof}
The result will follow after we prove two
observations.
The first observation that a generalized eigen-cocycle
is  transverse to $\cFs$ if and only if its factor satisfies $|\mu|>1$.
We shall prove the observation for eigen-cocycles and 
leave the small adjustments necessary for generalized
eigen-cocycles to the reader.

Recall from \S\ref{cocyclecmap} that
if $F$ is an eigen-cocycle with factor $\mu$, 
  after fixing a lift $\tphi$ and
a base point,  $F$ corresponds to an
almost eigen-\cmap\ $\sigma$ with
factor $\mu$. Now if $|\mu|> 1$, Lemma~\ref{simplelem}(b)
shows that $\sigma$ and thus $F$ are transverse to 
$\cF^s$.  On the other hand, 
if $|\mu|\geq 1$, Lemma~\ref{simplelem}(c)
says $\sigma$ and thus $F$ are transverse to $\cF^u$. Thus
if we assume that $F$ is also transverse to $\cFu$, 
$\sigma$ is constant on both stable and
unstable leaves in $\tM$. Using the local
structure of the foliations (either a product or
near a singularity) it follows easily that 
$\sigma$ is locally constant and so since
$\tM$ is connected, $\sigma$ is constant, and so
its corresponding cocycle is $F=0$. Thus we see
that there are no nontrivial 
transverse eigen-cocycles with factor $|\mu|\leq 1$,
completing the proof of the first observation.
 
The  second observation is that 
when a cohomology class contains
a transverse cocycle, the cocycle is unique. To prove this note
that by definition when two \cmaps\ 
$\sigma_1$ and $\sigma_2$ are cohomologous, then
$\sigma_1 = \sigma_2 + \chi\circ \pi$ for some
$\chi\in C^0(M,\F)$. If $\sigma_1$ and $\sigma_2$
are both transverse \cmaps, then so is $\chi\circ \pi$. This
implies that $\chi$ is constant on leaves of $\cF^s$, and
so $\chi$ is constant since every leaf of $\cF^s$ is dense
in $M$. Thus $\sigma_1$ and $\sigma_2$ differ by
a constant and so the correspond to the same cocycle.

To construct the isomorphisms in
the theorem statement, first it is immediate that $(\phi^*)^{Un}$ is 
a self-isomorphism of  $Un(\phi^*, H^1(M;\F))$.
Since $\phi$ and $\phi\I$
preserve leaves of $\cF^s$, the induced map
$\phi^*$ is also an isomorphism of $\cR$. 
By the second observation, $\cR$ is a finite
dimensional vector space. Thus the generalized
eigenvectors of $\phi^*$ acting on cocycles in
$\cR$ and those of $\phi^*$ acting on $Un(\phi^*, H^1(M;\F))$ give bases 
for those spaces. Using  Theorem~\ref{mainthm}
and the first observation
above we obtain a bijection between these collections
of generalized eigenvectors and thus 
an isomorphism between the spaces  
$\cR$ and  $Un(\phi^*, H^1(M;\F))$. The
commutativity of the diagram is immediate from the
construction of the isomorphism.\QED\
\end{proof}

\begin{remark} Note that the proof shows that
\pA\ have no eigen-cocycles with factors $|\mu| = 1$.
Also note that by using $\phi\I$ we have that the 
vector space of 
cocycles transverse to the unstable foliation $\cFu$
is isomorphic to $Stab(\phi^*, H^1(M;\F))$.
\end{remark}

For what follows we also need a local version of a
 transverse cocycle. 
\begin{definition}[Local transverse  cocycle]\label{loctrans}
Given a \pA\ map $\phi$, a \de{local transverse  cocycle}
for $\phi$ is a cover of $M$ by rectangles $\{R_1, \dots, R_d \}$
defined as in Definition~\ref{recdef}
in addition to a family of continuous functions 
$f_{i, j}:E^u_{i,j}\raw \F$  
for $i = 1, \dots, d$ and $j = 1,2$ with the properties that
\begin{compactenum}
\item For each $i$, $f_{i, 1}$ and $f_{i,2}$ are holonomy
invariant in $R_i$, \ie\ if  $g_i$ is the \homeo\ that takes
 $E^u_{i, 1}$ to $E^u_{i, 2}$ by sliding along leaves,
then $f_{i,1} = f_{i,2} \circ g_i$. 
\item When $E^u_{i,j} \cap E^u_{i',j'}\not = \emptyset$,
it is connected and 
$f_{i,j} -  f_{i', j'}$ is constant on the intersection.
\end{compactenum}
\end{definition}
The crucial features here are that  the rectangles
do not have to be oriented and the functions $f_{1,j}$ are
just defined on unstable edges, and they agree up
to a constant on   overlapping of  unstable edges.

\begin{fact}\label{localtcfact} A local transverse
cocycle gives rise to a unique transverse cocycle, and 
a transverse cocycle and a cover of $M$ by rectangles
yields a unique local transverse cocycle.
\end{fact}
\begin{proof} We first formally
define the operation of ``collapsing down
stable leaves to an unstable edge''. 
Specifically, given a rectangle $R_i$ define $h_i:R_i\raw E^u_{i, 1}$
so that $h_i(z)$ is the point where the stable
leaf through $z$ hits $E^u_{i, 1}$, or formally,
$h_i(z) = (L^u(z) \cap R_i) \cap  E^u_{i, 1}$.

Now given  local transverse cocycle 
$f_{i, j}:E^u_{i,j}\raw \F$ based on the cover
by rectangles $\{R_1, \dots, R_d \}$, for 
$i = 1, \dots, d$,  define $f_i:R_i\raw \F$ as 
$f_i(z) = f_{i,1} \circ h_i$. Now we fix a rectangle
$R_i$ and let $R_{i_1}, \dots, R_{i_m}$ be the
rectangles whose frontiers intersect that of $R_i$.
For each of these rectangles we may find a constant
$c_{i_n}$ so that $f_{i_n} + c_{i_n}$ agrees with
$f_i$ on the intersection of the respective frontiers.
Now enlarge $R_i$ slightly to an open $R_i^\prime$ 
and define $f_i^\prime:R_i^\prime\raw\F$ so that
$f_i^\prime = f_i$ on $R_i$ and $f_i^\prime =
f_{i_n} + c_{i_n}$ on each $R_{i_n}\cap R_i^\prime$.
After doing this construction for $i = 1, \dots, d$,
it is then easy to check that the family 
$(R_i^\prime, f_i^\prime)$ is a topological
one-form as defined in Definition~\ref{toponedef}.
By Fact~\ref{onepath} it gives a path cocycle and
by construction it is a transverse cocycle.

Now given a transverse cocycle $F$, for each unstable
edge $E^u_{i,2}$ fix a endpoint $p_i$ and define
$f_{i,2}:E^u_{i,2}\raw\F$ by $f_{i,2}(x) = F([p_i, x])$
where $[p_i, x]$ is the path in  $E^u_{i,2}$ from
$p_i$ to $x$. Now define $f_{i,1} = f_{i,2} \circ g_i$,
with $g_i$ as in Definition~\ref{loctrans}.
By the holonomy invariance of $F$ it follows
that the family of functions $f_{i, j}$ defines
a local transverse cocycle on the given cover
by rectangles. \QED\
\end{proof}

\section{Transverse arc functions for \pA\ maps}\label{tadsec}
In this section, \S\ref{sed}, and \S\ref{stdtad}
we consider transverse structures for
\pA\ maps which depend on the foliations and symbolic dynamics
and are not associated with cohomology. In Theorem~\ref{allthm}
we show that when the foliations are oriented, all the
 various structures agree.

A \de{transverse arc function (taf)} is a
geometric version of a transverse cocycle which
may also be viewed as a generalization of the 
transverse measure to the \pA\ invariant foliation.
The transverse measure is unique and it is usually constructed
using the \PF\ eigenvalue/vector of the transition matrix.
In \S\ref{stdtad} below we show how tafs arise from any other expanding
eigenvalues/vectors, and in Theorem~\ref{tafdist} we show that these other
tafs yield not transverse measures, but rather
distributions in the sense of elements of
the continuous dual to a space of test functions
which in this case just need to be H\"older.

\subsection{Definitions and basic properties}
By an \arc\ we mean the \textit{image} of an embedding
$\gamma:[a, b]\raw M$, for $a < b$.
Given an arc $\Gamma$, any embedding $\gamma:[a, b]\raw M$ with
$\Gamma = \gamma([a,b])$ is called a \de{parameterization}
of $\Gamma$, and we sometimes also write
$\Gamma = \im(\tau)$, with $\im$ meaning ``image''.
Thus, to be specific, an \textit{arc} is a closed subset of
$M$ and a \textit{path} is a parameterization of an arc. 
An arc carries no intrinsic orientation, while a path
is naturally oriented.

Given a \pA\ map $\phi$ with stable foliation $\cF^s$, let
$\cI$ be the collection of smooth \arcs\ which are
transverse to $\cF^s$ in $M-P$, where recall that
$P$ is the collection of singularities. We also allow \arcs\ in 
$\cI$  to have their endpoints at
a singularity of $\cF^s$. The collection $\cI$ is always
given the Hausdorff topology (but is not closed under it).

Informally, two arcs $\Gamma_0, \Gamma_1\in\cI$ are \de{holonomic} on
$\cF^s$ if one can slide one to the other along leaves of $\cF^s$. More
formally, they are holonomic if there is a family of 
parameterizations $\gamma:[0,1]\times [0,1]\raw M$ written
as $\gamma_s(t)$ with
\begin{compactenum}
\item $\im(\gamma_0) = \Gamma_0$ and  $\im(\gamma_1) = \Gamma_1$,
\item $\im(\gamma_s)\in\cI$ for all $s\in [0,1]$,
\item  $\gamma_s(t) \in L^s(\gamma_0(t))$ for all $s,t\in [0,1]$.
\end{compactenum}

\begin{definition}[Transverse arc function]\label{taddef}
Given a \pA\ map, 
a transverse arc function (taf) to the stable
foliation $\cF^s$ is a continuous map $G:\cI\raw\F$
which is 
\begin{compactenum}
\item Holonomy invariant,
\item Internally additive: If $\Gamma\in\cI$ and
$\gamma:[a,b]\raw M$ is a 
parameterization of $\Gamma$,  then for all $a <p < b$, 
$G(\gamma([a,b])) = G(\gamma([a,p]) + G(\gamma([p, b])$.
\end{compactenum}
\end{definition}
As with transverse cocycles, it will
be useful to have local version  
of a taf; the definition requires
a cover by \textit{oriented} rectangles.

\begin{definition}[Local transverse arc function]
A \de{local taf} is a cover of $M$ by oriented rectangles
 $\{R_1, \dots, R_n \}$
coupled with a family of continuous functions 
$f_{i, j}:E^u_{i,j}\raw \F$  
for $i = 1, \dots, n$ and $j = 1,2$ with the properties that
\begin{compactenum}
\item For each $i$, $f_{i, 1}$ and $f_{i,2}$ are holonomy
invariant, \ie\ if  $g_i$ is the \homeo\ that takes
 $E^u_{i, 1}$ to $E^u_{i, 2}$ by sliding along leaves,
then $f_{i,1} = f_{i,2} \circ g_i$. 
\item When $E^u_{i,j} \cap E^u_{i',j'}\not = \emptyset$,
it is connected and 
$f_{i,j} - \epsilon({i,j, i', j'}) f_{i', j'}$ 
is constant on the intersection,
where $\epsilon({i,j, i', j'}) = 1$ if the orientations of
$E^u_{i,j}$ and  $E^u_{i',j'}$ agree, and $-1$ otherwise.
\end{compactenum}
\end{definition}

We then have the analog of Fact~\ref{localtcfact} which simply says that
we have given a proper local version of a taf. We omit the
straightforward proof.
\begin{fact}\label{localtadfact} A local taf
 gives rise to a unique taf. 
A taf and a cover of $M$ by oriented  rectangles
yields a unique local taf.
\end{fact}

\begin{remark}\label{noatoms}
 Let $\cI'$ denote $\cI$ union the collection
of points in $M$, and so formally $\cI' = \cI \cup M$.
We also give $\cI'$ the Hausdorff topology.
 Fact~\ref{localtadfact}  implies that we may extend
$G$ to a continuous map $G'$ of $\cI'$ with 
$G'$ having the value of zero on any point. Informally this
says that a taf has no atoms.
\end{remark}

We define the pull back of a taf $G$ under the \pA\ map $\phi$
by $(\phi^* G)(\Gamma) = G (\phi(\Gamma))$. Since 
$\phi$ is smooth away from its 
singularities and $\phi$ preserves the foliations,
the image under $\phi$ of an element of $\cI$ is always in 
$\cI$, and thus the pull-back under $\phi$ of taf  to $\cF^s$
 is also a  taf to $\cF^s$.
We say that $G$ is an \textit{eigen-taf} for 
$(\phi, \cF^s)$ with factor  $\mu\in\F$
 if $\phi^* G = \mu G$.

Recall that given a \pA\ map on $M$, if the foliations are
not oriented, there is a unique two-fold branched cover
$p:\bM\raw M$ called the \de{orientation cover}. It is
the smallest branched cover in which the lift of  $\cF^s$ is oriented
(see, for example, \cite{bb} for more details). Since
the branch points of  orientation cover are always singularities,
transverse cocycles and tafs to  $\cF^s$ on $M$ pull back
to transverse cocycles and tafs to  $\bcF^s$ on $\bM$
with the obvious definitions.

\subsection{Transverse arc functions and transverse cocycles}
The next proposition gives the equivalence of transverse
arc functions  with
transverse cocycles when the foliation $\cF^s$ is oriented.
Thus in this case using Fact~\ref{tcfact}  
there is an  eigen-taf for every eigenvalue $\mu$ of $\phi^*$ acting on
$H^1(M;\Z)$ with $|\mu| > 1$.

\begin{proposition}\label{tadcocprop}
Let $\phi$ be a \pA\ map
on an orientable surface $M$ such that 
the stable foliation $\cF^s$ is orientable.
There are natural isomorphisms which make the
following commute.
\begin{equation*}
\begin{CD}
\cR @>\cong>> Un(\phi^*, H^1(M;\F))  @>\cong>> \cG \\
 @V\phi^*VV @VV\phi^*V  @VV\phi^*V  \\
\cR @>\cong>> Un(\phi^*, H^1(M;\F))  @>\cong>> \cG%
\end{CD}
\end{equation*}
Under the isomorphism
 eigen- and generalized eigen-cocycles  correspond
to  eigen- and generalized eigen-tafs with the same
factors. In particular, all eigen-tafs have factors $|\mu|>1$.
\end{proposition}

\begin{proof}
The first isomorphism is given by Proposition~\ref{tcfact}.
For the second the construction  
is local. When $\cF^s$ is orientable, one
may assign a coherent family of orientations to any
cover by rectangles, \ie\ the family has the property that
when $E^u_{i,j} \cap E^u_{i',j'}\not = \emptyset$,
the orientations of $E^u_{i,j}$ and  $E^u_{i',j'}$ always agree.
Thus when $\cF^s$ is orientable, a local taf and a local transverse
 cocycle are the
same object, and so the result follows from
Fact~\ref{localtcfact} and Fact~\ref{localtadfact}.
 The last two statements of the proposition
follow from the commutativity of the diagram and the fact that
the horizontal maps are all isomorphisms.
\QED\
\end{proof}

We have the following corollary which also holds for
 \pA\ with nonorientable foliations.
\begin{corollary}\label{factorfact}
 If $\phi$ is a \pA\ map, 
then all eigen-tafs to $\cF^s$ have factors $|\mu|>1$ and
all
eigen-cocycles have factors $|\mu| \not = 1$.
\end{corollary}

\begin{proof}
For a general \pA, if there was
 eigen-taf to $\cF^s$ with factor $|\mu|\leq 1$, then
we can pull it back to the orientation double cover and
get a contradiction to Proposition~\ref{tadcocprop}.
The second statement  follows
from Remark~\ref{factorrk}, since a \pA\ map has a local product structure
at all but finitely many points. \QED\
\end{proof}

\section{Symbolic transverse arc functions}\label{sed}
As is often the case in dynamics, symbolic
methods  simplify and clarify certain technical
issues. In this section we
 define the symbolic analog of a transverse arc function 
which is called
a \de{symbolic transverse arc function} and use it in
Theorem~\ref{tadstd}  to characterize the collection of
 taf's and eigen-taf's.

The construction makes use of a standard technique 
in hyperbolic dynamics which uses a one-sided subshift
of finite type to model the  stable foliations
of an  Axiom A \diffeo\ (see, for example,
\cite{bowenmarkus}). There are two basic but closely related approaches
to doing this. The first is to form a quotient space by identifying
sequences with the same future; the quotient is then  essentially
the leaf space of the foliation. The second, which we adopt here,
is to identify each length one cylinder set with the leaf
space in the corresponding Markov rectangle.

The construction of an staf uses the one-sided subshift
of finite type $\Lambda_A^+$ built from the
$\{0,1\}$-transition matrix $A$. We shall assume
that the symbol set is $S := \{1, 2, \dots, d\}$. An allowable
transition is a pair $(i,j)$ with $A_{ij} = 1$. The notation
$i\raw j$ means that $(i,j)$ is allowable. Thus the shift space
is defined as 
\begin{equation*}
\Lambda_A^+ = \{ \us\in S^{\N} : s_k \raw s_{k+1} \ \text{for all} \
k\in \N\}.
\end{equation*} 
An \de{allowable block} for the
subshift is a finite list of allowable transitions. 
The collection of allowable blocks for a given subshift 
is denoted $\cB(A)$.
 The number of symbols in a block is its \de{length} and is denoted
$\ell(b)$.  
For an allowable block $b$, $[b]$ is the corresponding cylinder set
 starting at the zero$^{th}$ place,
\begin{equation}
[b] = \{ \us\in\Lambda_+ : s_j = b_j \ \text{for} \ j = 0, \dots, \ell(b)-1\}.
\end{equation}
Note that in this paper all cylinder sets start at the zero$^{th}$ place
unless otherwise noted.
In what follows that matrix $A$ is usually fixed, and so
we will often suppress the dependence on $A$.

\begin{definition}[Symbolic transverse arc function]\label{Kdef}
Assume that $(\Lambda^+, \sigma)$ is a one-sided subshift
of finite type on $d$-symbols with $\{0,1\}$-transition matrix $A$.
Let $K$ be an $\F$-valued function on the collection
of allowable blocks $\cB(A)$, so $K:\cB(A)\raw \F$. The set function
$K$ is called a symbolic transverse arc function (staf) if
it is 
\begin{compactenum}
\item Additive: For all allowable blocks $s_0 s_1 \dots s_{n-1}\, j$,
\begin{equation}\label{Kadditive}
K(s_0 s_1 \dots s_{n-1}\, j) = \allowsum K( s_0 s_1 \dots s_{n-1}\, j\, k).
\end{equation} 
\item Coherent: For all allowable blocks $s_0 \dots s_n\; j$ and
$s_0^\prime \dots s_n^\prime\; j$,
\begin{equation}\label{Kcoherent}
K( s_0 \dots s_n\, j) = K( s_0^\prime \dots s_n^\prime \,j).
\end{equation}
\end{compactenum}
If there exist constants $C>0$ and $r>0$
  so that for all $b\in\cB$, 
\begin{equation}\label{expdef}
|K(b)| < C r^{-\ell(b)},
\end{equation}
then $K$ is said to have a $(C,r)$-exponential bound 
or just an $r$-exponential bound.
\end{definition}
 In most of the symbolic dynamics literature
in a bound such as \eqref{expdef} it is required that
$r>1$. This will often be the case here, but for the
definition of staf's of exponential bound we
are just requiring $r>0$.

 Let $\cA(A)$ be the
algebra generated by the cylinder sets of $\Lambda_A^+$. It is 
easy to check that $\cA$ is all finite disjoint unions of 
cylinder sets of $\Lambda_A^+$. 
Condition Definition~\ref{Kdef}(a) 
says that $K$ yields a finitely-additive function
on $\cA$. Conversely, any finitely additive map $\cA\raw F$ yields
an additive $K:\cB\raw \F$. Thus a staf is
a coherent, finitely additive set function on $\cA$. 

The smallest $\sigma$-algebra containing $\cA$ is
the Borel sets. In Theorem~\ref{radonstd}(a) below, we see that for a mixing
$(\Lambda_+, \sigma)$ the only staf which yields a Borel measure
 will be eigen-staf corresponding to   
the Perron-Frobenius eigenvector of $A$. On the other hand,
Theorem~\ref{radonstd}(b) shows that any staf
yields a distribution in sense of an element of the
dual space of a class of H\"older functions on $\Lambda_A^+$. 

We will also see in Fact~\ref{expfact} below that every staf has
an exponential bound. The ``transverse''
in the nomenclature `` symbolic transverse arc function''
comes from the condition in Definition~\ref{Kdef}(b) which ensures
that if $A$ is the transition matrix of a \pA\ map $\phi$,
then $K$ yields a 
 transverse structure to the stable foliation.
In  Fact~\ref{tadstd} we show that only staf with an
exponential bound with $r>1$ correspond to 
taf to $\cF^s$. 

The collection of all symbolic transverse arc functions to
the one-sided subshift determined by $A$  is 
denoted $\cK(A)$.
Before getting to  more measure theoretic type results
we give a simple alternative description of $\cK(A)$.
Given the $d\times d$-matrix $A$, a \de{thread} of $A$ is 
an infinite list of vectors $(\vv_0, \vv_1, \dots )$ 
with each $\vv_j\in\F^d$, 
 so that
\begin{equation}\label{threaddef}
\vv_n = A \vv_{n+1},
\end{equation}
for all $n\in\N$.  The collection of 
threads of $A$ is the inverse limit of $\F^d$ for which 
$A$ is  all the one-step transition maps and 
is denoted $\ilimit(\F^d, A)$.

Now $\cK(A)$ is also clearly a vector space over $\F$. 
For $K\in \cK(A)$, 
by coherence, for any allowable block $b = s_0 \dots s_{n-1} j$, 
the value $K(b)$ depends only
on the last symbol in the block $j$. Thus 
for each $n\in\N$ we can define   $\vKn$ as the vector constructed 
from the values of $K$ on length-$n$ 
cylinder sets with the $j^{th}$ component of
$\vKn$ being $\vKnj = K(s_0 \dots s_{n-1} j)$, for any allowable
block $s_0 \dots s_{n-1} j$. Thus $K\in \cK(A)$ yields
a list of vectors $\vK := (\vK_0, \vK_1, \dots )$. 
Since the transition matrix $A$ is a $\{0,1\}$-matrix,
 the additivity condition in Definition~\ref{Kdef}(a) translates as
$\vKn = A \vK^{(n+1)}$ for all $n\in\N$, and so $\vK$ is 
a thread.  

\begin{fact}\label{stdthread}
 The assignment $K \mapsto \vK$ just described
 is a vector space isomorphism from $\cK(A)$
to $\ilimit(\F^d, A)$. Further, $\ilimit(\F^d, A)$ is isomorphic to
$NonN(A, \F^d)$, the non-nilpotent subspace of $A$ acting on $\F^d$.
\end{fact}

\begin{proof} It is obvious that $K \mapsto \vK$ is vector space
monomorphism. To see that it is surjective, note that
it follows from \eqref{threaddef} that $\vK^{(k)} = A^n \vK_{(k+n)}$ for all
$n\in\N$, and thus each $\vK^{(k)}$ is contained in the eventual image of $A$,
namely $\cap_{n\in\N} A^n(\F^d)$.  Using the Jordan form of $A$
it is easily seen that the eventual image is
exactly $NonN(A, \F^d)$.
Since $A$ restricted to $NonN(A, \F^d)$ is invertible, 
each $\vK_0\in NonN(A, \F^d)$ yields a 
unique thread $\vK$, completing the proof. \QED\
\end{proof} 

\subsection{Eigen-staf}
There is a natural action of
$A$ on  threads, namely,
\begin{equation}
A^*(\vK) = (A\vK_0, A \vK_1, A\vK_2, \dots) =
(A\vK_0,  \vK_0, \vK_1, \dots).
\end{equation} 
using \eqref{threaddef}. It is clear that $A^*$ is a
vector space self-isomorphism of $\ilimit(\F^d, A)$.

We next define the induced co-action of the left shift 
$\sigma$ on
the space of staf, $\cK(A)$. It
will correspond to the action of $A$ on threads under the
isomorphism of Fact~\ref{stdthread}. 
\begin{definition}[Co-action of $\sigma$ on threads]\label{stafdef}
 For $K\in \cK(A)$ define $\sigma^*K$ on blocks as
\begin{compactenum}
\item For each block $b$ of length greater than one,
$(\sigma^*K)(b)  = K(\sigma(b))$,
\item For each symbol $j$ (a block of length one),
\begin{equation*}
\sigma^*K(j)= \allowsum K(k).
\end{equation*}
\end{compactenum} 
\end{definition}
It is also clear that 
$\sigma^*$ is a vector space self-isomorphism of $\cK(A)$.

If we let $K$ denote the action on cylinder sets instead
of blocks, $\sigma^*K$ is the pull back of
$K$ under $\sigma$, $(\sigma^* K)([b]) = K(\sigma([b]))$. 
Note that in contrast to what is usual for measures,
we are pulling back not pushing forward. This makes
sense since the image of a cylinder set is always the finite
union of cylinder sets.
 However, it is also important to note
that a cylinder set here are always based at the ``decimal
point''. Thus while $\sigma^*K$ is a staf and so
trivially extends to the algebra $\cA$ of finite unions
of cylinder sets, this extension does \textit{not} satisfy
 $\sigma^* K  = K\circ\sigma$. As a simple example,
by definition of the extension to $\cA$,
 $(\sigma^* K)([i k] \uplus [j k]) = 
 \sigma^* K([i k]) +   \sigma^* K( [j k]) 
= K([k]) + K([k])$, but
$K(\sigma( [i k] \uplus [j k])) =  K([k])$.

From Definition~\ref{stafdef} it is easy to confirm the following.
\begin{fact}\label{stdaction}
 If $A$ is a $\{0,1\}$-matrix defining a subshift
of finite type and $\cK(A)$ the space of symbolic 
transverse arc functions is defined as above, then
there are natural isomorphisms which make the following commute.
\begin{equation*}
\begin{CD}
NonN(A, \F^d) @>\cong>> \ilimit(\F^d, A)  @>\cong>> \cK(A) \\
 @V A VV @VV A^*V  @VV\sigma_* V  \\
NonN(A, \F^d) @>\cong>> \ilimit(\F^d, A)  @>\cong>> \cK(A) %
\end{CD}
\end{equation*}
Since every map in the diagram is an isomorphism the 
 eigen- and generalized
eigen-objects of the vertical maps correspond.
In particular, an eigen-staf $K_\mu$ with factor $\mu$ corresponds
to a thread $v^{(0)}, v^{(1)}, \dots$ with $v^{(0)}$ a 
right eigenvector
of $A$ with eigenvalue $\mu$, and $v^{(n)} = \mu^{-n} v^{(0)}$.
Thus $K_\mu(s_0 \dots s_{n-1}\, j) = \mu^{-n} (v^{(0)})_j$.
Similarly, generalized eigen-staf correspond to generalized
eigenvectors of $A$.
\end{fact}

\subsection{Staf with exponential bound}
In the sequel we will require the  additional information about 
the role of exponential bounds
contained in the next fact. 
\begin{fact}\label{expfact}
Let $A$ be a $\{0,1\}$-matrix defining a subshift
of finite type which is irreducible and aperiodic, or equivalently,
there is an $N$ so that  $A^n > 0$ for all $n>N$.

\begin{compactenum}
\item If $K$ is an eigen- or generalized eigen-staf with eigenvalue $\mu$,
 then it has exponential bound with $r = |\mu|$. 
\item Given a generalized eigenbasis $\{K_i\}$
 for $\sigma^*$ acting on $\cK$,
if $K\in\cK$ is written
\begin{equation}\label{expression}
K = \sum c_i K_{\mu_i},
\end{equation}
then $K$ has an exponential bound with $r$ equal to the minimum 
of the $|\mu_i|$ such that $c_i\not= 0$ in
expression \eqref{expression}, and $\mu_i$ is the eigenvalue
corresponding to $K_i$.
\item $K\in\cK$ has an exponential bound with  $r > 1$ if
and only if for any sequence of allowable blocks $b_n$ with
$b_n\raw \us$ for some sequence $\us$, we have $K(b_n) \raw 0$.
\end{compactenum} 
\end{fact}

\begin{proof}
The proofs of (a) and (b) are straightforward linear algebra
using the observation that a thread $\vK$ corresponds to a staf $K$
with an $r$-exponential bound if and only if for all $n>0$,  
$\| \vK^{(n)} \| \leq C'/r^n$ for some constant
$C'$ and vector norm $\|\cdot \|$.

For (c) assume that $K$ is not unstable. From 
Fact~\ref{stdthread} this means that $\vK^{(0)}\not\in Un(A, \F^d)$. 
  Letting $B$ be the inverse of $A$ restricted to
the non-nilpotent subspace of $A$, we then have that 
$\vK^{(0)}\not\in Stab(B, NonN(A, \F^d))$ and since 
 $B\vKn = \vK^{(n+1)}$ we have that $\vKn\not\raw 0$ as $n\raw\infty$.
Now since  $A$ is irreducible by hypothesis, 
 let $N$ be such that $n>N$ implies that $A^n > 0$.
Since $\vKn\not\raw 0$, we may find a
symbol (or component) $a$ and a subsequence $n_i\raw\infty$
with $n_{i+1} - n_i > N$ for all $i$  with 
$(\vK^{n_i})_a\not\raw 0$ as $i\raw\infty$. For each
$j>0$ since $n_{i+1} - n_i > N$, we may find a block 
$b_j = s_0 \dots s_{n_j}$ with $s_{n_i} = a$ for $i = 1, \dots, j$. 
By construction, $K(b_j) = (\vK^{n_j})_a\not\raw 0$, but
$[b_j] \raw \us = s_0 s_1 \dots$. The other implication 
in (c) is trivial. \QED\

\end{proof}

\subsection{Functions with exponential bound}\label{funcexp}
As is common in symbolic dynamics we will use an
exponential decay condition on the variation as an alternative
description of the H\"older condition.
A function $f:\Lambda^+\raw\F$ is said have an
$r$-exponential  bound if there is a constant $C>0$ so
that for all allowable
blocks $b$,
\begin{equation*}
\max\{ |f(\us) - f(\us')| : \us, \us'\in [b] \} \leq
C r^{-\ell(b)}
\end{equation*}
The collection of all $f$ with an 
$r$-exponential decay bound is denoted $\cE^r(\Lambda_+,\F)$. If 
$|f|_\infty$ is the usual sup-norm and 
\begin{equation*}
|f|_r = \sup_{b\in\cB} \sup_{\us,\us'\in [b]} 
{|f(\us) - f(\us')|}{r^{\ell(b)}},
\end{equation*}
then $\| f \|_r = |f|_\infty + |f|_r$ is a norm  making
$\cE^r(\Lambda^+,\F)$ into a Banach space when $r>1$.

There is a certain fluidity in the notion of H\"older functions
on a shift because there is a family of  natural metrics which
all give the same topology, namely, for $\theta > 1$ let
\begin{equation*}
d_\theta (\ut, \us) = \sum_{i=0}^\infty
 \frac{1 - \delta(t_i, s_i)}{\theta^i},
\end{equation*}
where $\delta(t_i, t_s)$ is the Kronecker delta.
As one varies the metrics it will vary H\"older exponents,
whereas the exponential decay description is metric
independent.

The main observation needed to go back and forth from
H\"older bounds to exponential bounds is
that there exist constants $k_1, k_2 > 0$ 
\begin{equation}\label{holdexp}
\frac{k_1}{r^{\ell(b)}} \leq \diam_\theta ([b]) 
\leq \frac{k_2}{r^{\ell(b)}}
\end{equation}
for all blocks $b\in\cB$, where $\diam_\theta([b])$ is the
diameter of the set $[b]$ in the metric $d_\theta$.
This implies when
$r > \theta > 1$  that $f\in\cE^r(\Lambda^+,\F)$
if and only if  $f\in C^\nu(\Lambda^+)$ for
$\nu = \log(r)/\log(\theta)$, and further
the  identity map $\cE^r(\Lambda^+,\F)
\raw  C^\nu(\Lambda^+)$ is continuous in the
given norms.

\subsection{Connections to other standard measures and
distributions.}
In  this subsection we make a few remarks 
on the relation of staf to other standard elementary
constructions in the ergodic theory of 
subshifts of finite type.  
To easily discuss the connection of staf's to some standard 
measures on shifts we need to change notation a bit.
Let $\vr_\mu$ and $\vell_\mu$ be right (column) and 
left (row) eigenvectors
of $A$ corresponding to an eigenvalue $\mu$
normalized so that $\vr_\mu \vell_\mu  = 1$.
Recall from above that $\cA$ is the algebra generated by
the cylinder sets.  In this notation the function
$K_\mu:\cA \raw \F$ induced by  
$K_\mu([s_0 \dots s_{n}]) = \mu^{-n} (\vr_\mu)_{s_n}$ is
the eigen-staf with eigenvalue $\mu$ given in Fact~\ref{stdaction}.

 On the other hand,   
$J_\mu:\cA \raw \F$ induced by
$J_\mu([s_0 \dots s_{n}])  = \mu^{-n} (\vell_\mu)_{s_0} (\vr_\mu)_{s_n}$
will not be a staf, but is
$\sigma$-invariant in the usual sense 
that $J_\mu(\sigma\I(Y)) =J_\mu(Y)$ for all $Y\in \cA$. 
These $J_\mu$ will generate H\"older distributions in the sense of 
Theorem~\ref{radonstd}(b) below, and as in that Theorem,
 only $J_\lambda$ where
$\lambda$ is  the \PF\ eigenvalue of $A$ will extend to an
invariant measure on the Borel $\sigma$-algebra of $\Lambda^+_A$.
This invariant measure is usually called the \de{Parry measure}
and is the measure of maximal entropy for $(\Lambda^+_A, \sigma)$.

These $J_\mu$ can also be extended to invariant distributions
on the full-shift $\Lambda_A$. Letting $[s_0 \dots s_{n}]_j$
denote the cylinder set beginning at place $j$, so 
$[s_0 \dots s_{n}]_j = \{ \ut\in\Sigma_A : t_{j + i} = s_i \ 
\text{for}\ i = 0, \dots, n\}$, we get an $\sigma$-invariant 
distribution on the full shift generated by 
$\hJ_\mu([s_0 \dots s_{n}]_j)  = \mu^{-n} (\vell_\mu)_{s_0} (\vr_\mu)_{s_n}$.
Once again  $\hJ_\lambda$ extends to an invariant measure
also called the Parry measure. 

We also remark that $\hJ_\mu$ is the product of the staf 
 $K_\mu$ on the positive one-sided shift
 $\Lambda^+_A$ and an analogous
$K_\mu^-$ on the negative one-sided shift $\Lambda^-_A$ in the following
sense. For a cylinder set $ b^- = [s_{-n} \dots s_{-1} s_0]$ 
in the negative one sided shift, let $K_\mu^-(b^-) = 
\mu^{-n} (\vell_\mu)_{s_{-n}}$, and so $K_\mu^-$ is
eigen under the right shift on  $\Lambda^-_A$. Given
a ``rectangle block'' in the full shift, 
$b_r = [ s_{-n} \dots s_{-1} s_0 \dots s_n]$, we have
that $\hJ_\mu(b_r) =  K_\mu(\pi_+ b_r) K_\mu^-(\pi_- b_r)$,
where $\pi_+:\Lambda_A \raw \Lambda^+_A$ and
$\pi_-:\Lambda_A \raw \Lambda^-_A$ are the projections.

\subsection{Ruelle's transfer operator}\label{transfersec}
In this subsection we comment on the connection
of the action $\sigma^*$ on staf and Ruelle's
transfer operator. See \cite{baladi} for more
information on the transfer operator and related
structures.

The (unweighted) transfer operator $\cL$ acts on
bounded functions $ f:\Lambda_A^+\raw \F$ by
\begin{equation}\label{transferdef}
\cL f(\us) = \sum_{\ut:\sigma(\ut) = \us}  f(\ut)= 
\sum_{j\raw s_0} f(j \us).
\end{equation}
It is standard and simple to verify that for each $r>1$, 
if $ f\in \cE^{r}(\Lambda^+, \F) $, then 
$\cL f\in\cE^{r}(\Lambda^+, \F) $, and
$\cL:\cE^{r}(\Lambda^+, \F)\raw \cE^{r}(\Lambda^+, \F)$
is a continuous, linear operator. The dual of $\cL$
acting on a functional $L\in\cE^{r}(\Lambda^+, \F)^*$ is
$\cL^* L(f) = L(\cL f)$.
A simple computation yields that the action of $\cL$ on
indicator functions of cylinder sets is 
$\cL \bbone_{[b]} = \bbone_{\sigma([b])}$ when
 $\ell(b) > 1$, and for length
one cylinder sets
$\cL \bbone_{[j]} = \bbone_X$, where $X = \uplus_{j\raw k} [k]$.

Now given a staf $K$, denote its corresponding functional
on indicator functions of cylinder sets by $\hK$, and
so $\hK(\bbone_{[b]}) = K(b)$. By definition the dual of $\cL$
acts on this functional by $\cL^*\hK(\bbone_{[b]}) 
:= \hK(\cL \bbone_{[b]})$ and  another simple computation
shows that 
\begin{equation}\label{switchone}
\cL^*\hK(\bbone_{[b]}) = (\sigma^* K)(b),
\end{equation}
or put another way, $\cL^*\hK = \widehat{\sigma^* K}$. 
Thus the co-action of $\sigma$ on stafs 
is basically the same as the dual action of the unweighted
transfer operator.

This is to be expected. The co-action $\sigma^*$ was
defined to mimic the action of a \pA\ map
on tafs. A standard technique in hyperbolic
dynamics is to collapse down stable manifolds 
and use the transfer operator of the resulting
expanding dynamical systems to find interesting
measures. Symbolically, this is often accomplished
using the one-sided shift to represent the 
leaf space of the stable foliation as is done here. 

In the case
of \pA\ maps, after collapsing down stable manifolds
one gets a branched one-manifold (the train track) on
which the \pA\ map induces an expanding map. It would
be interesting the apply the transfer operator theory
directly to this system.

\subsection{Staf as measures and distributions} 
In this subsection we show that all staf give rise
to functionals dual to a  class of functions
with exponential bound, but only the \PF\
staf gives a functional dual to the space of all  continuous
functions, \ie\ gives a signed measure.

The construction of a distribution from a
staf is straightforward. We do it in an elementary manner: 
for a staf $K$ we define ``$\int f\; dK$'' by imitating
 the usual construction of the Riemann integral but we 
use the partition of $\Lambda^+_A$ into blocks of length
$n$ at the $n^{th}$ stage. Since the staf $K$
has an exponential bound, convergence is
obtained by restricting to
functions $f$ with an appropriate exponential
bound.  

To describe why only the \PF\ staf gives a measure, 
we first recall a definition. Let $\cA$ be an algebra of
subsets of a set $X$ and   $F:\cA \raw \F$  a finitely additive
set function. Given a finite partition $\cP = 
\{ X_1, \dots, X_n\}$ of $X$ into disjoint subsets (often written 
 $X = \uplus X_n$) with each $X_n\in\cA$,
 the \de{variation} of $F$ on $\cP$ 
is $\var(F, \cP) = \sum |F(X_n) |$ and the \de{total variation} of $F$ 
on $X$ is $\var(F, X) = \sup \{ \var(F,\cP)\}$ with the sup
over all finite partitions of $X$. It is a standard fact
(see, for example, Theorem 6.4 in \cite{rudin}) that if $F$ extends to
a measure on the $\sigma$-algebra generated by $\cA$,
then $F$ has bounded variation. More specifically,
if $F$ does not have bounded variation, then $F$ is
not countably additive in the sense that there
are a sequence of disjoint sets $Y_n\in \cA$ with
$\sum F(Y_n)$ diverging.

Let us say that $A^N > 0$ for  $n>N$ and 
for simplicity that $A$ is invertible. 
There are two main ideas behind the fact that
only the \PF\ eigenvectors give actual
measures.
 Using the \PF\ Theorem, 
the only vectors $\bv$ with $A^{-n}\bv > 0$ for all
$n>0$ are  \PF\ eigenvectors.   Thus
by the correspondence of Fact~\ref{stdaction},
any staf other than the \PF\  one assigns
both positive and negative values to various cylinder
sets. The second idea is that number of cylinder sets 
which end in a given symbol
grows like $\lambda^n$ whereas the $K$-value 
of each shrinks by at most $1/|\mu|^n$, where
$\mu$ is the eigenvalue of second largest
modulus. Thus the $K$-value of the collection
of  length $n$ cylinder sets  
which end in a given symbol grows like
$\pm \lambda^n/|\mu|^n\raw\infty$. 

As noted in the introduction, 
Ruelle in \cite{ruelle} defines Gibbs distributions
dual to H\"older functions on a subshift
of finite type and Haydn \cite{haydn} shows that
Gibbs distributions are always eigen with respect to 
the action of the transfer operator.  In the previous
subsection we saw that  by treating a
staf as a functional on the indicator functions 
of blocks  that the action of
$\sigma^*$ on stafs is the same as that of the
dual of the transfer operator $\cL^*$. After
showing that ``$\int \cL f \; dK = \int f\; d(\sigma^* K)$'',
we have that eigen-stafs induce eigen-distributions.

\begin{theorem}\label{radonstd}
 Assume that $A$ is irreducible and aperiodic
with \PF\ eigenvalue $\lambda$, and let 
$\cA$ be the algebra generated by the 
cylinder sets of $\Lambda^+$. Assume
$K\in\cK$ and $\vK$ is the corresponding
thread with $\vK^{(0)} \in NonN(A)$ given by Fact~\ref{stdaction}
\begin{compactenum}

\item $K$ has bounded variation as a set function
if and only if $\vK^{(0)}$ is a \PF\ eigenvector.
In particular, 
the only  $K\in\cK$ 
which extend to a  Borel measure  on $\Lambda^+$ 
are those constructed from the \PF\ eigenvector as 
in Fact~\ref{stdaction}

\item If 
$K\in \cK(A)$ has an  $r_2$-exponential bound
 with $0 < r_2 \leq\lambda$, it  defines a continuous, linear functional
$L_K$ on $\cE^{r_1}(\Lambda^+, \F) $ for all $r_1 > \lambda/r_2$.

\item If $K$ is an eigen-staf under the action of
$\sigma^*$, then the functional
$L_K$ is an eigen-distribution under the action of
the dual transfer operator $\cL^*$.
\end{compactenum}
\end{theorem}

\begin{proof}
The argument that the \PF\ staf $K_\lambda$ extends
to a measure is standard in symbolic dynamics:
since  $K_\lambda$ viewed
as a finitely additive set function on $\cA$  has  positive
values on all cylinder sets, and there no countable
unions of cylinder sets in $\cA$, the Caratheodory
extension theorem gives that $K_\lambda$ extends
to a Borel measure.

Now assume that $K$ is such that $\vK^{(0)}$ is not 
a \PF\ eigenvector. Thus by Fact~\ref{PFfact}, there is a
$C>0$,   $N$ and $0 < r < \lambda$ so that  $n>N$ 
implies that 
\begin{equation}\label{valuebound}
\|\vK^{(n)} \|_1 > C/ r^n.
\end{equation}
For each $n>0$ we define the partition $\cP^{(n)}$
of $\Lambda_+$ as having elements
\begin{equation*}
P^{(n)}_j := \{ \us\in\Lambda_+ : s_n = j\}
\end{equation*}
for $ j = 1, \dots, d$. If we let $N^{(n)}_j$ denote the
number of disjoint blocks of length $(n+1)$  making up
$P^{(n)}_j$, then $N^{(n)}_j$ is exactly the number of
ways to make $n$ allowable transitions and end with
$j$, and so $N^{(n)}_j$ is counted by the sum of
the entries in the $j^{th}$ column of $A^n$, or 
\begin{equation*}
N^{(n)}_j = \sum_{i=1}^d (A^n)_{ij}.
\end{equation*}
The \PF\ Theorem as in \eqref{PFconv} shows that
there is a constant $C_1 > 0$ so that 
\begin{equation}\label{numberbound}
N^{(n)}_j > C_1 \lambda^n
\end{equation}
for all $j$ and  $n > N$.

Since each block
of length $n+1$ that ends with $j$ has $K$-value
equal to $\vK^{(n)}_j$, we have
$K(P^{(n)}_j) = N^{(n)}_j \vK^{(n)}_j$. Thus by
\eqref{valuebound} and \eqref{numberbound},
\begin{equation}\label{diverges}
\var(K, \cP^{(n)}) = \sum_{j=1}^d  |K(P^{(n)}_j)|
\geq \frac{C C_1 \lambda^n}{r^n} \raw \infty,
\end{equation}
as $n\raw\infty$, showing that $K$ is not
of bounded variation on $\cA$. As noted above the
theorem, this implies that $K$ cannot be extended
to a signed or complex Borel measure.

Now to prove (b), assume that $K$ has a $(c_2, r_2)$-exponential
bound. Let $b_{n,1}, b_{n, 2}, \dots, b_{n, N(n)}$
be an enumeration of the allowable length $n$ blocks. Note that
it follows from \eqref{PFconv}, there exists a $c_3$ with 
$N(n) < c_3 \lambda^n$. Given
$f$ has an $(c_1, r_1)$ exponential bound, we may assume that the
$c_1$ is optimal in the sense that
$c_1 = |f|_{r_1}$. Now for each $n\in\N$ pick elements 
$\ux_{n,i}\in b_{n,i}$ and define
\begin{equation}\label{Sndef}
S_n = \sum_{i=1}^{N(n)} f(\ux_{n,i}) K( b_{n,i}).
\end{equation}
We shall show that the sequence $S_n$ is Cauchy.

Fix a length $n$-block $b_{i,n}$ and label the length
$(n+1)$ blocks it contains as $b_{i_1,n+1}, \dots, b_{i_k,n+1}$, so
 $b_{i,n} = \uplus_{j=1}^k b_{i_j,n+1}$ as a disjoint union.
Using the finite additivity of $K$,
\begin{equation}\label{oneblock}
\begin{split}
|f(x_{i,n}) K(b_{i,n}) -  \sum_{j=1}^k f(x_{i_j,n+1}) K(b_{i_j,n+1})|
&= | \sum_{j=1}^k (f(x_{i,n})- f(x_{i_j,n+1})| K(b_{i_j,n+1})|\\
&\leq k c_1 r_1^{-n} c_2 r_2^{-(n+1)}.
\end{split}
\end{equation}
Now certainly $k \leq d$ the total number of symbols, and
the total number of blocks of length $n$ is  $N(n) < c_3 \lambda^n$.
Thus adding \eqref{oneblock} over the length $n$ blocks 
\begin{equation}\label{cauchyest}
|S_n - S_{n+1}| \leq d c_1 c_2 c_3 r_2^{-1} (\frac{\lambda}{ r_1 r_2})^n.
\end{equation}
Since by assumption, $\frac{\lambda}{ r_1 r_2}< 1$, the sequence $S_n$
is Cauchy and we define
\begin{equation*}
L_K(f) = \lim_{n\raw\infty} S_n.
\end{equation*}

To see that this definition is independent of
the choice of evaluation points, note that 
 if $\ux_{n,i}^\prime\in b_{n,i}$ is another choice of points,
\begin{equation*}
\begin{split}
|S_n - S_n^\prime| &\leq \sum_{i=1}^{N(n)} c_1 r_1^{-n} c_2 r_2^{-n}\\
    &\leq c_1 c_2 c_3 (\frac{\lambda}{ r_1 r_2})^n\\
     &\raw 0,
\end{split}
\end{equation*}
as $n\raw\infty$ when  $\frac{\lambda}{ r_1 r_2}< 1$.

Now $L_K$ is obviously linear. To see that it is continuous, note
that $|L_K(f)| \leq |S_0| + |L_K(f) - S_0|$. Recalling that
$c_1 = |f|_r$, the Cauchy estimate \eqref{cauchyest} and summing
the tail of the geometric series shows that 
$ |L_K(f) - S_0| \leq \hat{C} |f|_{r_1}$ with 
\begin{equation*}
\hat{C} = \frac{d c_2 c_3}{ r_2} \left(1- \frac{\lambda}{r_1r_2})\right)^{-1},
\end{equation*}
 independent of $f$.
Since $|S_0| \leq \|f\|_\infty K(\Lambda_+)$, we have
that for some $C'$, $|L_K(f)| \leq C' \|f\|_{r_1}$, and
so $L_K$ is bounded and thus continuous.

To prove (c) we show that if $K$ has an $r_2$-exponential
bond and $f\in\cE^{r_1}(\Lambda^+, \F)$, then
\begin{equation}\label{switcheq}
L_K(\cL f) = L_{\sigma^* K} (f).
\end{equation}
Thus if $\sigma^* K = \mu K$, since $L_{\mu K} = \mu L_K$,
we have proved (c).

First note that for the indicator function of a
cylinder set $\bbone_{[b]}$,  the construction
of $L_K$ yields $L_K(\bbone_{[b]}) = \hat{K}(\bbone_{[b]})
=K(b)$, with $\hat{K}$ as in \S\ref{transfersec}. Thus using
\eqref{switchone} we have  that
\begin{equation}\label{switchtwo}
L_K(\cL\bbone_{[b]}) = (\sigma^*K)(b)
\end{equation}
Continuing the notation as in \eqref{Sndef},
for each $n\in\N$ let 
\begin{equation}\label{simplefunc}
f_n = \sum_{i=1}^{N(n)} f(\ux_{n,i}) \bbone_{[ b_{n,i}]}.
\end{equation}
and so  $L_K(f_n) = S_n$. 
Now above we defined $L_K(f)$ as $\lim S_n$, but
here we need to note that each $f_n\in\cE^{r_1}(\Lambda^+, \F)$
and $f_n\raw f$ in the norm of that space, and so
now we have $L_k(f_n) \raw L_K(f)$ by the continuity
of $L_K$.

Using \eqref{switchtwo} and \eqref{simplefunc},
\begin{equation*}
L_K(\cL f_n) = \sum_{i=1}^{N(n)} f(\ux_{n,i}) \sigma^* K( b_{n,i})
\end{equation*}
and so $L_K(\cL f_n)\raw L_{\sigma^* K} (f)$.
But since $\cL$ is continuous,  $L_K(\cL f_n)\raw L_K(\cL f)$,
and we have \eqref{switcheq}, finishing the proof of (c).
\QED\ 
\end{proof}

\begin{remark} It is worth noting that the eigen-staf
given in Fact~\ref{stdaction},
 do \textit{not} give rise to eigen-distributions
under the induced action of $\sigma$ on   $\cE^{r_1}(\Lambda^+, \F) $.
As a trivial example, for a constant function $c$, one
has $L_K(c) = L_K(c \circ \sigma)$, for all staf $K$. The
underlying reason is contained 
in the comment given after Definition~\ref{stafdef}:  
an eigen-staf $K$ is eigen for the induced action of $\sigma$
on the space of staf's, or as we have just seen,
for the dual action of the transfer operator, and
not for the action of $\sigma$ on
$\Lambda_A^+$. 
\end{remark}

\section{Transverse arc functions and symbolic transverse arc
 functions}\label{stdtad}
In this section 
 we connect the staf constructed from the transition
matrix  to tafs to the stable foliation $\cF^s$ of a \pA\ map.
In contrast to the connection of 
cocycles and tafs in Proposition~\ref{tadcocprop},
the correspondence of stafs and tafs does not
require that the foliations be oriented.

We start with a few remarks on the task at hand
in order to motivate the subsequent definitions and proof.
A staf $K$ assigns numbers to allowable blocks of a one-sided subshift
of finite type. When a shift is coding a \pA\ map, an allowable
block $b$ of the one-sided shift corresponds naturally to certain
arcs inside unstable leaves.
 More specifically, if $b = s_0 s_1 \dots s_{n-1} s_n$
is an allowable block, let $S_b := \cap_{i=0}^n \phi^{-i} R_{s_i}$.
Then $S_b$ is also a rectangle in the sense of \S\ref{rectdef} and any
arc $\Gamma_b$ connecting its stable edges is said to represent $b$.
Alternatively, $\Gamma\in\cI$ represents $b$ if and
only if  $\phi^{i} (\Gamma) \subset R_{s_i}$ for all
$i = 0, \dots, n$.

If a staf $K$ is to correspond to a taf $G$, then clearly for any
arc $\Gamma$ representing $b$,  $G(\Gamma)$ should
be equal to $K(b)$. 
Now note that given a Markov rectangle $R_j$ and $n>0$, 
$\phi^{-n}(R_j)$ is a long, thin band of stable leaves
bounded by $\phi^{-n}(E_{j,1})$ and $\phi^{-n}(E_{j,2})$. 
Arcs from $\cI$ connecting these two
boundaries are all holonomic, and so all should be assigned
the same value by a taf. 
Since $\phi^{-n}(R_j) = \cup S_b$ with the union over all allowable
blocks $b$ of length $n$ that terminate in $j$, we see 
given a taf, a corresponding staf must assign the same
value to all allowable blocks of the same length that
end in the same symbol. This is exactly the coherence condition
on a staf given in Definition~\ref{Kdef}.

As just noted, given a staf $K$, the first step in using it to construct
a taf is clear, any arc representing an allowable block
$b$ should be assigned the value $K(b)$.  But what about 
other transverse arcs? Any arc $\Gamma\in\cI$ can be
written as the union of arcs representing allowable blocks 
$b_1, \dots $ and
so $G(\Gamma) = \sum K(b_n)$ is the proper definition. However,
it is difficult to
see  from the symbolic dynamics for a given arc $\Gamma\subset M$
what the constituent allowable blocks $b_1, \dots $  
should be.
 What we do here is allow the \pA\ dynamics to organize
the sum. Specifically, start with a transverse arc
and take a high iterate which yields a very long image
arc. We can then list the Markov rectangles
it crosses. These yield subarcs which represent blocks.
We then take higher and higher iterates and
pass to the limit.

In constructing a taf, Fact~\ref{localtadfact} above says that 
it suffices to construct a local taf, and for
this, it is only necessary to assign values
to arcs $\Gamma\in \cI$ which are contained
in the unstable edges of Markov rectangles.
Thus if we let $\cI^u$ be
the collection of arcs $\Gamma$ contained in
unstable leaves, then a function
$G:\cI^u \raw \F$ which satisfies Definition~\ref{taddef}(a) and (b)
 will yield a unique taf. 

We need a few more definitions before the main result.
An \arc\ $\Gamma\in\cI^u$ is called
a \de{loose bit} if $\Gamma$ is contained wholly inside a single
Markov rectangle, but doesn't go all the way across. In 
other words, $\Gamma$ perhaps hits one of the unstable edges,
but not both. For an  \arc\  $\Gamma\in\cI^u$,
its \de{loose bits} are the portions at each end of
the arc that are loose bits. Formally, the loose bits of $\Gamma$
are the connected components of some $\Gamma\cap R_j$ which
are loose bits. The \de{main bit} of $\Gamma$ is $\Gamma$ minus
its loose bits, or equivalently, the part of $\Gamma$ that goes all the way
across Markov rectangles.  See Figure.

\begin{floatingfigure}{5cm}\label{fig1}\centering
\includegraphics[scale=0.55]{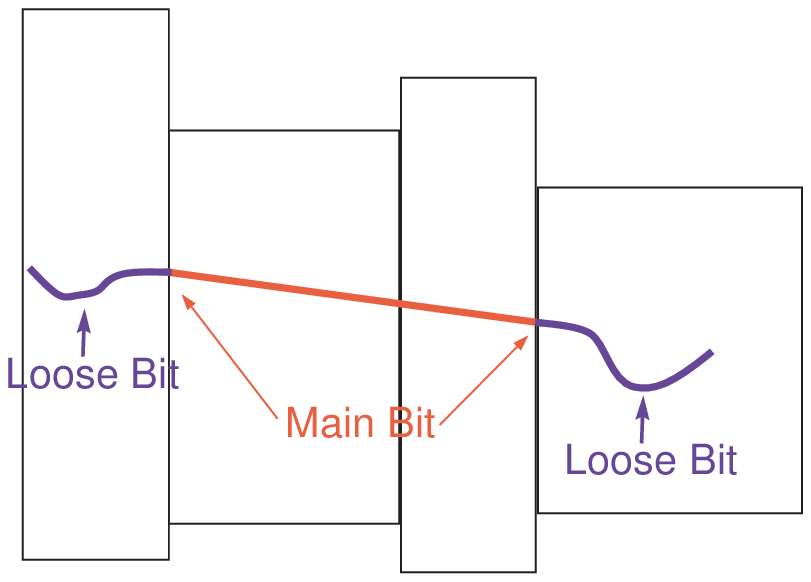}
\end{floatingfigure}

 By Remark~\ref{noatoms}, we know  a taf has no atoms. By 
Fact~\ref{expfact}(c), a staf  has no 
atoms only when it is unstable, \ie\ has an exponential bound with
$r > 1$. Equivalently, the corresponding thread's 
initial vector $\vK^{(0)}$ should be in
$Un(A, \F^d)$.  We call such stafs  
\de{unstable} staf and denote the
collection of them as $\cK^{u}$.

 \begin{theorem}\label{tadstd}
Let $\phi$ be a \pA\ map with a Markov partition giving
rise to the $\{0,1\}$-transition matrix $A$.
The collection of unstable symbolic transverse
arc functions to the one-sided shift $\Lambda^+_A$
is naturally identified with the collection of all 
transverse arc functions to the foliation $\cF^s$.
The identification is a vector space isomorphism
which makes the following diagram commute:
\begin{equation}\label{GKdiagram}
\begin{CD}
\cG  @>\cong>> \cK^{u}\\
@V \phi^ *VV @V \sigma^* VV \\
\cG  @>\cong>> \cK^{u}.\\ %
\end{CD}
\end{equation}
As a consequence, under the identification  eigen-staf
and generalized eigen-staf with
eigenvalues $|\mu|>1$ correspond  to
the  eigen-taf and generalized eigen-taf.
\end{theorem}

\begin{proof}
Throughout the proof it will be convenient to identify
a staf $K$ with its thread $\vK$ as justified by Fact~\ref{stdthread}.
 We will frequently do this without further comment.

We first consider the map $\cK^u\raw \cG$. As noted above,
 it suffices to define a taf $G$ on arcs from $\cI^u$.
Let $E^u$ denote the union of the unstable edges of
Markov rectangles. For a $\Gamma\in\cI^u$ with $\Gamma\cap E^u = \emptyset$,
define $I(\Gamma)\in\N^d$ so that its
$j^{th}$ component is the
number of complete crossings of $\Gamma$ with the rectangle
$R_j$ (so we don't count loose bits), and let $w(\Gamma)\in\N^d$
be
\begin{equation}\label{basic}
w(\Gamma) = I(\phi(\Gamma)) - A^T I(\Gamma).
\end{equation}
Thus $w$ counts the crossings of $\phi(\Gamma)$ which
are missed by the symbolic count, $A^T I(\Gamma)$. This is
the same as the full crossings contributed by the image
of the loose bits of $\Gamma$, or more precisely,
if $\ell_1$ and $\ell_2$ are the loose bits of
$\Gamma$, then $w(\Gamma) = I(\phi(\ell_1)) + I(\phi(\ell_2))$. Thus
there is a constant $c_1 > 0$ with 
\begin{equation}\label{unifbd}
\|w(\Gamma)\| \leq 2 \sup\{ \|I(\phi(\ell))\| : \ell \ \text{
is a loose bit}\} < c_1 < \infty.
\end{equation}

Now by the definition of a 
Markov partition (see, for example, \cite{FLP}, page 201),
$\phi\I(E^u)\subset E^u$. This implies that any $\Gamma$ with
$\Gamma\cap E^u = \emptyset$, satisfies
$\phi^n(\Gamma) \cap E^u = \emptyset$, for all $n\geq 0$.
Call these $\Gamma$ \de{well-coded}. For 
a well-coded $\Gamma$ and for $n\in\N$, let 
\begin{equation}\label{Gndef}
G_n(\Gamma) = \langle \vKn, I(\phi^n(\Gamma))\rangle,
\end{equation}
where $\langle \cdot, \cdot \rangle$ is the standard inner product
on $\F^d$. 
By induction from \eqref{basic},
\begin{equation}\label{basic2}
I(\phi^n(\Gamma)) = (A^T)^n I(\Gamma) + 
\sum_{i=1}^n (A^T)^{n-i} w(\phi^{i-1}(\Gamma)), 
\end{equation}
and so 
\begin{equation}\label{Gsum}
\begin{split}
G_n(\Gamma) &= \langle \vKn, (A^T)^n I(\Gamma) \rangle
+ \sum_{i=1}^n \langle \vKn, (A^T)^{n-i}  w(\phi^{i-1}(\Gamma))\rangle\\
&=   \langle A^n \vKn, I(\Gamma) \rangle
+ \sum_{i=1}^n \langle A^{n-i} \vKn,   w(\phi^{i-1}(\Gamma))\rangle\\
&=   \langle \vK^{(0)}, I(\Gamma) \rangle
+ \sum_{i=1}^n \langle \vK^{(i)},   w(\phi^{i-1}(\Gamma))\rangle.
\end{split}
\end{equation}
Thus if $K$ has an exponential bound with constants
$c_2$ and $r >1$, using the Cauchy-Schwartz 
inequality, $G_n(\Gamma)$ is bounded by 
$\langle \vK^{(0)}, I(\Gamma) \rangle + \sum_{i=1}^n c_1 c_2 r^{-n}$,
and so $G_n(\Gamma)$ converges as $n\raw \infty$ 
by the Weierstrass $M$-test. Thus we may define 
\begin{equation}
G(\Gamma) = \lim_{n\raw\infty} G_n(\Gamma).
\end{equation}

It follows easily from the definition that $G$ is internally
additive on well-coded arcs. We now show that it is holonomy 
invariant. Since $K$ has an
 exponential bound with $r>1$, this will follow directly from
\eqref{Gndef} after we prove the following claim: If
$\Gamma_0$ and $\Gamma_1$ are  well-coded and holonomic, then
there is a constant $C = C(\Gamma_0, \Gamma_1)$ with 
\begin{equation}\label{bound}
\| I(\phi^n(\Gamma_0)) - I(\phi^n(\Gamma_1))\| \leq C,
\end{equation}
for all $n > 0$. 

To prove the claim, first assume that $\Gamma_1$ and
$\Gamma_2$ are ``close together'' in the sense that there is
a holonomy $H:[0,1]\times [0,1] \raw M$ that is
bijective. Then $\im(H)$ is a rectangle in  the sense of \S\ref{rectdef}.
In particular, for all $n\in\N$, $\phi^n(\im(H))$ at most contains
all the edges from $E^u$. This implies that \eqref{bound} holds
for $\Gamma_1$ and $\Gamma_2$ with $C$ equal to twice the  
number of Markov rectangles, or $C = 2d$. More generally,
if $\Gamma_1$ and $\Gamma_2$ are such that their holonomy
$H$ is at most $k$-to-one, then by decomposing $H$ into
a collection of holonomies of arcs close together, we
get that \eqref{bound} holds with $C = 2 k d$, completing
the proof of the claim.

We next show that $G$ is continuous on well-coded arcs.
For this, first note that from \eqref{Gsum} and the fact that $K$ has
an exponential bound with constants $c_2$ and $r$ we have that
for any $\Gamma$ and $n>0$, $|G_n(\Gamma) - G(\Gamma) | \leq
c_1 c_2 r^{-n-1}/(1-r\I)$, 
and thus using \eqref{Gndef}, for any $\Gamma_0$ and
$\Gamma_1$ and $n>0$,
\begin{equation}\label{cont}
|G(\Gamma_0) - G(\Gamma_1)| < 2 c_1 c_2 r^{-n-1}/(1-r\I) +
c_2 r^{-n} |I(\phi^n(\Gamma_0)) - I(\phi^n(\Gamma_1))|.
\end{equation}
Since $G$ is holonomy invariant, it suffices to show 
continuity of $G$ for $\Gamma_0$ and $\Gamma_1$ in the
same leaf of $\cF^u$. By the continuity of $\phi$, for
any $n>0$ there is a $\delta>0$ so that $d(\Gamma_0, \Gamma_1) <
\delta$ in the Hausdorff topology implies that 
$|I(\phi^n(\Gamma_0)) - I(\phi^n(\Gamma_1))| < 2$, which coupled
with \eqref{cont} gives the continuity of $G$. 

Finally, to complete the definition of $G$ on arcs
contained in unstable leaves, for 
 $\Gamma'\in E^u$, pick a well-coded arc $\Gamma$ 
with $\Gamma$ holonomic to $\Gamma'$, and let $G(\Gamma') 
= G(\Gamma)$.  Since we have shown that the
value of $G$ is holonomy invariant among well-coded \arcs,
$G$ is well-defined. The fact that this extension
to all of $\cI$ satisfies Definition~\ref{taddef} follows
from its properties on well-coded arcs, completing the
definition of the map $\cK^u \raw \cG$.

Next we consider the map $\cG \raw \cK^u$. For
a given taf $G$ and an allowable block $b$,
pick an arc $\Gamma$ that represents $b$ and define
$K:\cB\raw\F$ by  $K(b) = G(\Gamma)$.
Because $G$ is  additive and holonomy invariant, 
$K$ is additive and coherent, and thus is a staf. Finally,
if $K$ was not unstable,  then by Fact~\ref{expfact}(c) there 
would be a sequence of allowable blocks $b_n$ 
with the sets $[b_n]$ converging to a point in the
Hausdorff topology  but $K(b_n)\not\raw 0$. This
would imply the existence of a nested collection
of arcs $\Gamma_n\in\cI^u$ which represent the
$b_n$ and the sets $\Gamma_n$ converging to a point in the
Hausdorff topology 
and $G(\Gamma_n)\not\raw 0$. This violates the
property of tafs given in Remark~\ref{noatoms}, and so 
$K\in\cK^u$, as required. 

It is clear that both of the maps $\cK^u \raw \cG$ 
and $\cG\raw \cK^u$ are vector space homomorphisms. 
We next show that they are isomorphisms 
by showing their compositions are the identity.

For $K\in\cK^u$ denote the image
under $\cK^u \raw \cG$ as $G_K$. To show that 
$\cK^u \raw \cG \raw \cK^u$ is the identity we
must show that if $\Gamma$ represents the block
$b\in\cB$, then $G_K(\Gamma) = K(b)$. Assume that
$b = [s_0 \dots a]$ and $\ell(b) = \ell$. This
implies that $\phi^\ell(\Gamma)$ is contained in
the Markov rectangle $R_a$ and in fact connects its
stable edges. This, in turn, implies that 
$I(\phi^\ell(\Gamma))_a = 1$ for $i=1$, and is zero otherwise.
Further, 
for  all $j\geq \ell$, $I(\phi^{j+1}(\Gamma)) = A^T I(\phi^j(\Gamma))$.
Thus for all $n>\ell$, 
$I(\phi^{n}(\Gamma)) = (A^T)^{n-\ell} I(\phi_\ell(\Gamma))$, and so
from \eqref{Gndef},
\begin{equation*}
\begin{split}
G_n(\Gamma) & = \langle \vKn, (A^T)^{n-\ell} I(\phi^\ell(\Gamma))\rangle\\
 &=  \langle A^{n-\ell}  \vKn,  I(\phi^\ell(\Gamma))\rangle\\
& = \langle \vK^{(\ell)},  I(\phi^\ell(\Gamma))\rangle\\
& =  (\vK^{(\ell)})_a.
\end{split}
\end{equation*}
Finally, since $b$ terminates in $a$ and has length $\ell$, 
 $K(b) = (\vK^{(\ell)})_a = G_n(\Gamma) \raw G_K(\Gamma)$.

For $G\in\cG$ denote the corresponding element
under $\cG \raw \cK^u$ as $K_G$. To show that the composition 
$ \cG \raw \cK^u \raw \cG $
is the identity, using the result of the last  paragraph 
we have that
$G_{K_G}(\Gamma) = G(\Gamma)$ for any arc $\Gamma$ that
represents a block $b$. It suffices then to prove the claim that a taf is
uniquely determined by its value on arcs which represent
allowable blocks. Since a taf is continuous by definition,
this will follow once we show that finite unions of  arcs which represent
allowable blocks are dense in $\cI^u$. Given an arc $\Gamma\in\cI^u$,
for any $n > 0$, if $J^{(n)}_i$ is a component of $\phi^n(\Gamma) - E^s$,
then define 
$\Gamma^{(n)}_i := Cl(\phi^{-n} (J^{(n)}_i)) \subset \Gamma$, 
and we have that $\Gamma^{(n)}_i$
represents a block in $\Lambda^+_A$.
Since $\Gamma$ is contained in an unstable leaf,
 $\phi$ is uniformly expanding on $\Gamma$, and so 
$\cup_i \Gamma^{(n)}_i \raw \Gamma$ in the Hausdorff
topology  as $n\raw\infty$, proving the claim.

Finally, we show that the diagram \eqref{GKdiagram} commutes.
Given a $K\in\cK^u$ with thread $\vK$, since by
definition $\phi^* G_K(\Gamma) = G_K(\phi(\Gamma))$ we have
that 
\begin{equation*}
\begin{split}
\phi^* G_K(\Gamma)
 &= \lim_{n\raw\infty} \langle \vKn, I(\phi^{n+1} (\Gamma))\rangle\\
    &= \lim_{n\raw\infty} \langle A \vK^{(n+1)}, I(\phi^{n+1} (\Gamma))\rangle.
\end{split}
\end{equation*}
On the other hand, since $A^*\vK = \{ A\vK_0, A\vK_1, \dots\}$,
using Fact~\ref{stdaction}  we have
\begin{equation*}
G_{\sigma^* K}(\Gamma) =  \lim_{n\raw\infty} 
            \langle A \vK^{(n)}, I(\phi^{n} (\Gamma))\rangle,
\end{equation*}
and so $\phi^* G_K = G_{\sigma^* K}$. \QED\
\end{proof}

\begin{remark}\label{PFmeas}
One of the standard construction of the transverse
measure $m^u$ to the stable foliation $\cF^s$ is essentially
as the \PF\ eigen-taf $G_\lambda$ which
as just shown  corresponds to the \PF\ eigen-staf $K_\lambda$.
In Theorem~\ref{tafdist} we indicate the standard
proof that $G_\lambda$ extends to a measure. We shall
need this fact before the proof and note that
its proof is in fact independent of 
its use (and in any event is standard).   

In what follows it will be convenient to adopt the 
convention that the transverse measure $m^u$ is adapted to the
transition matrix $A$ in the sense that it is constructed from a
\PF\ eigenvector  $\vv_\lambda$ which have been is normalized 
so that $\|\vv_\lambda\|_1 = 1$. This implies, in particular, that
the Markov rectangle $R_j$ has $m^u$-width equal to
the $j^{th}$ component of $\vv_\lambda$.
\end{remark}

\section{Putting it all together}
For the reader's convenience we summarize the
results of the last few sections concerning
structures associated with the various spectra
of \pA\ maps.

\begin{theorem}\label{allthm}
Let $\phi$ be an orientation-preserving
\pA\ map on an orientable surface with a Markov partition giving
rise to the $\{0,1\}$-transition matrix $A$.

\begin{compactenum}
\item The collection of symbolic transverse arc functions
on $\Lambda^+_A$ is naturally
identified with the non-nilpotent subspace of $A$, 
and under this identification
eigen-staf and generalized eigen-staf
 correspond to eigenvectors and generalized eigenvectors of $A$. 
\item The collection of unstable staf
is naturally identified with the collection of 
transverse arc functions to the foliation $\cF^s$,
and under this identification eigen-staf
and generalized eigen-staf with
eigenvalues $|\mu|>1$ correspond  to
the eigen- and generalized eigen-transverse arc functions.
\item When the foliation $\cF^s$ is orientable, each of
the eigen- and generalized eigen-objects in (b) is also 
identified with an eigen- and generalized eigen-transverse
cocycle. 
\end{compactenum}
Diagrammatically when $\cF^s$ is orientable we have
\begin{equation*}
\begin{CD}
Un(\phi^*, H^1(M;\F)) @>\cong>>\cR @>\cong>> \cG  @>\cong>> \cK^{u}@>\cong>> Un(A, \F^d)\\
@V \phi^* VV @V \phi^* VV@V \phi^* VV  @VV \sigma^* V  @VV A V \\
Un(\phi^*, H^1(M;\F)) @>\cong>>\cR @>\cong>> \cG  @>\cong>> \cK^{u}@>\cong>> Un(A, \F^d)%
\end{CD}
\end{equation*}
When $\cF^s$ is not orientable the isomorphism
$\cR\cong\cG$ does not exist, but the left portion and
the right portion  of the diagram still commutes. 

Using the inverse $\phi\I$ one has analogous results
connecting staf of $A^T$ to structures transverse
to  the unstable foliation $\cF^u$.
\end{theorem}

\begin{remark}\label{rykremark}
Theorem~\ref{tadcocprop} implies the main portion of
a theorem of Rykken \cite{rykken}, namely, 
 in the case of orientable foliations
the hyperbolic portions of the homological and Markov
spectra of $\phi_*$ agree.
\end{remark}

\section{Transverse arc distributions}
It is well-known that \pA\ foliations are uniquely
ergodic, \ie\ up to scalar multiple they have a 
unique holonomy invariant transverse measure,
and this measure is usually built from the 
\PF\ eigenvector of the transition matrix. Thus
the corresponding tafs are the only ones that extend
to measures. In light of  the correspondence of
tafs and stafs given in Theorem~\ref{tadstd}
and the result in Theorem~\ref{radonstd}(b), it is not surprising that 
all  tafs yield distributions in the sense of an
elements of the dual to a space of H\"older functions.
These results could be develped directly, but given
the symbolic results above it is technically easier to
use the semi-conjugacy to ``push them down'' to
the manifold.
We first review the aspects of the semiconjugacy that
are of relevance here.

Fix a Markov rectangle $R_j$ and a well-coded 
(in the sense given in the proof of Theorem~\ref{tadstd}) arc
$\Gamma\in\cI^u$ contained in $R_j$ which connects
the two stable edges $E^s_{j,1}$ and  $E^s_{j,2}$.
Since $\phi$ uniformly expands $\Gamma$ in forward time, 
for each $\us\in [j]\subset \Gamma^+_A$ there is a unique point
$x \in \Gamma$ with $\phi^n(x) \in R_{s_n}$ for all
$n>0$. Thus we maybe define $\omega:[j] \raw  \Gamma$
by the assignment of $\us$ to its corresponding
$x$. The definition of a Markov partition and
its corresponding subshift ensure that $\omega$ is
onto. Now if for some $n>0$, the forward
iterate $\phi^n(x)$ intersects a stable boundary
of a Markov rectangle, then there will be exactly
two sequences $\us, \us'$ with $\omega(\us) = \omega(\us')
=x$. Those $x$ whose forward orbits miss the unstable
boundaries have a unique $\omega$-preimage.

It follows from the construction of $\omega$
that  for any cylinder $[b] \subset [j]$, its image
$\gamma_b :=\omega([b])\subset \Gamma$ is a
compact arc representing $[b]$ 
in the sense given above Theorem~\ref{tadstd}. We call the collection
of all such arcs $\cC$, explicitly, 
\begin{equation}\label{cCdef}
\cC = \{ \gamma\subset \Gamma : \gamma = \omega([b])
\ \text{for some} \ [b] \subset [j] \}.
\end{equation} 
The collection of endpoints of arcs in
$\cC$ is exactly the set of $x\in\Gamma$ with non-unique 
$\omega$-preimage.

Given a staf $K$, we may  define its push-forward
onto $\cC$ as  $\omega_* K: \cC \raw \F$ via 
$\omega_* K(\gamma_b) = K(b)$. 
The proof of Theorem~\ref{tadstd} says that 
$\omega_* K $ always extends to a taf on $\cF^s$,
and all taf are so obtained. For this reason if 
a staf $K$ corresponds to a taf $G$ in the isomorphism
of Theorem~\ref{tadstd}, we will write $G = \omega_* K$. This
is a slight abuse of notation because, as just described,
endpoints of arcs $\gamma_b\in\cC$ are double coded and so
$\omega\I(\gamma_b)$ is equal to $[b]$ union
a pair of sequences not in $[b]$ whose $\omega$-images
are the two endpoints of $\gamma_b$. In general,
each closed subarc $\Gamma'\subset\Gamma$ has a preimage
$\omega\I(\Gamma')$ which can be  written 
as the countable  union of blocks $b_n$ and perhaps one or 
two additional sequences, with the latter only
present when one or both endpoints of $\Gamma'$ 
is also an endpoint of an arc $\gamma_b\in\cC$. 
When $\omega\I(\Gamma') = \uplus [b_n]$ (with
the possible addition of two extra sequences), it follows
from the proof of Theorem~\ref{tadstd} that $\omega_* K (\Gamma')
= \sum K(b_n)$. 

The connection of 
the \PF\ taf $G_\lambda$ to the transverse
measure $m^u$ (\cf\  Remark~\ref{PFmeas}) gives
that for any subarc $\Gamma'\subset\Gamma$,
$G_\lambda(\Gamma')$ is the $d_\phi$-diameter of
$\Gamma'$, \ie\ $G_\lambda(\Gamma') = m^u(\Gamma')$.
 Since  $K_\lambda$ has an exponential bound 
of $\lambda$  and $G_\lambda = \omega_* (K_\lambda)$, 
we have that  $\omega$  
has an exponential bound 
of $\lambda$ , and so by \S\ref{funcexp}, 
$\omega$ is Lipschitz  considered as a map $(\Lambda^+, d_\lambda)
\raw (\Gamma, d_\phi)$.

\begin{remark}\label{badstaf}
The connection of stafs and tafs via the map $\omega$
make it clear why 
the stable and center staf on $\Lambda^+$ don't
push forward under $\omega$ to taf for $\phi$.
Specifically,  let  $x\in\Gamma$ be such that it is not
an endpoint of 
an arc in $\cC$.  If $K$ is not unstable  then any collection
 $\gamma_{b_n}\in\cB$ 
with $x = \cap \gamma_{b_n}$ will have $|K(b_n)| \not\raw 0$.
Further, let $\Gamma'\subset\Gamma$ be a closed arc
such that at least one of its endpoints is not also
the endpoint of an arc from $\cC$. Again, if $K$ is not unstable   
then although we may
write $\Gamma' = \cup \gamma_{b_n}$ with the blocks
$b_n$ disjoint, the 
$\sum K(b_n)$ won't converge for the simple
reason that $|K(b_n)| \not\raw 0$. On the other
hand, given a unstable staf $K$, the exponential
bound $|K(b_n)| \leq r^{-\ell(b_n)}$ for $r>1$ will imply
that the sum $\sum K(b_n)$ will converge
if, for example, there is a uniform bound on
the number of blocks of each length. This is
the case when $w(\cup [b_n])$ is an arc, but
not the case in the construction of 
Theorem~\ref{tafdist}(a) below based on that
of Theorem~\ref{radonstd}(a).
\end{remark}

In constructing measures and distributions
from taf we focus attention 
on the taf restricted to a single arc $\Gamma$ in 
the unstable foliation. By holonomy invariance and
density of leaves for   \pA\ foliations this
information will transfer to any other transverse arc.
For concreteness and simplicity we fix $\Gamma$ as
in the beginning of this section: $\Gamma$ is an
arc in an unstable leaf connecting the two 
stable boundaries of a Markov box $R_j$. We
also assume that $\Gamma$ is \textit{not} one
of the unstable boundary components.
\begin{definition}[H\"older transverse arc function]
A taf $G$ is \de{$(C, \nu)$-H\"older} if 
\begin{equation}
|G(J)| \leq C (m^u(J))^\nu,
\end{equation}
for all subarcs $J\subset\Gamma$. 
\end{definition}
 The taf $G$ extends to a signed Borel measure
if there is such measure which agrees with $G$ on
all compact subarcs of $\Gamma$.

\begin{theorem}\label{tafdist} 
Let $\phi$ be a \pA\ map with a Markov partition $\{R_i\}$ with
associated subshift of finite type $\Lambda_A^+$.
For a   length one cylinder set $[j]$, let
$\omega:[j] \raw \Gamma$ be the coding map
defined above where $\Gamma$ is an arc in an unstable
leaf connecting the stable boundaries of the Markov
rectangle $R_j$.
\begin{compactenum}
\item If $G$ is a taf with $G = \omega_* K$ for
a staf $K$ with an $r$-exponential
bound with $1<r\leq\lambda$, 
then $G$ is $ (\log r)/(\log \lambda)$-H\"older
\item The only taf which has bounded variation  and
thus extends to a signed Borel  measure on
transversals to $\cF^s$ are the \PF\ taf.
\item If the taf $G$ is $\nu_2$-H\"older, it defines
a continuous, linear functional on $C^{\nu_1}(\Gamma, \F)$
for all $\nu_1 > 1 - \nu_2$.
\end{compactenum}
\end{theorem}

\begin{proof}
We use the notation in the proof of Theorem~\ref{tadstd} . To
prove (a), for the given subarc $J\subset\Gamma$ let
$N$ be the smallest integer so that a portion
of $\phi^N(J)$ goes all the way across some Markov
box. The smallest $m^u$-width of
any Markov box is the smallest component (which is always
nonzero) of the normalized \PF\ eigenvector
(\cf\ Remark~\ref{PFmeas}), and we call this smallest
value $c_4$. We then have that
$\lambda^N m^u(J) = m^u(\phi^N(J))) \geq c_4$.
Now also note that by the choice of $N$, 
$w(\phi^n(J) = 0$ for all $n<N$. Thus since $I(J) = 0$,
by \eqref{Gsum} we have using  Cauchy-Schwarz with
$c_1$ and $c_2$ as defined in   the proof of Theorem~\ref{tadstd}
and $\nu := (\log r)/(\log \lambda)$,
\begin{equation*}
\begin{split}
G(J) &= \sum_{i=N}^\infty \langle \vK^{(i)},   w(\phi^{i-1}(\Gamma))\rangle
\leq \sum_{i=N}^\infty c_1 c_2 r^{-i} \\
&= 
r^{-N} c_1 c_2 (1-r\I)\I 
=
(\lambda^{-N})^\nu c_1 c_2 (1-r\I)\I\\ 
&\leq (m^u(J))^\nu \frac{c_1 c_2}{c_4^\nu (1-r\I)},
\end{split}
\end{equation*}
as required.

For part (b), using the sets $P^{(n)}_j$ 
in the proof of Theorem~\ref{radonstd}, 
for each $n>0$ and $ j = 1, \dots, d$,
let $\hat{P}^{(n)}_j = \omega(P^{(n)}_j)$.
Thus for each $n$, 
$\Gamma = \cup_j \hat{P}^{(n)}_j$. By Theorem~\ref{tadstd},
if $G$ is not a \PF\ taf, then $G = \omega_* K$ for
$K$ not a \PF\ staf. Thus 
 using \eqref{diverges},
 $\sum_j |G( \hat{P}^{(n)}_j) | = \sum_j |K( {P}^{(n)}_j) |
\raw \infty$
as $n\raw\infty$. Now   $G$ has the property
that when a family of intervals $J_k$ converge
down to a point (cf. Remark~\ref{noatoms}), then $G(J_k)\raw 0$, 
and thus any
extension measure can have no atoms.  For each $n$, when 
$j\not= j'$, $\hat{P}^{(n)}_j\cap \hat{P}^{(n)}_{j'}$
is a finite set of points, and so we see that any countably 
additive extension
of $G$ to the Borels would not be of bounded variation
and thus not a measure.

For part (c), given $f:(\Gamma, d_\theta)\raw \F$ that
is $\nu_1$-H\"older, 
since as noted above, 
$\omega:(\Lambda^+, d_\lambda)
\raw (\Gamma, d_\phi)$ is Lipschitz,
 $f \circ \omega:(\Lambda_+, d_\lambda)\raw\F$
is also $\nu_1$-H\"older. It then follows from the
statement at the end of \S\ref{funcexp} that
$f$ has an $r_1 := \lambda^{\nu_1}$ exponential decay
bound. Now let $K$ be the staf with $G = \omega_* K$,
then by part (a), $K$ has an $r_2 := \lambda^{\nu_2}$
exponential
decay bound. Thus $\nu_1 + \nu_2 >1$ implies
$\lambda/(r_1 r_2) < 1$, and so we may define
$L_G:C^{\nu_1}(\Gamma, \F)\raw\F$ by
$L_G(f) = L_K(f\circ\omega)$ with $L_K$ as in
Theorem~\ref{radonstd}. It follows immediately that $L_G$ in
linear and its continuity follows from that
of $L_K$ in conjunction with the fact that
$\omega$ is Lipschitz and the equivalence of the
H\"older and exponential bound norms given at the
end of \S\ref{funcexp}.
\QED\
\end{proof}

\section{Regularity}
In this section we investigate the regularity
of the various transverse structures for \pA\ maps.
Notions of regularity are easiest to formulate
with \cmaps. As noted in \S\ref{tcsec}, an eigen-\cmap\
$\talpha$ with factor $|\mu|>1$ is constant on leaves of the lifted
stable foliation $\tcF^s$ in $\tM$. Since the lifted unstable
foliation $\tcF^u$ is transverse to $\tcF^s$, the regularity
of $\talpha$ restricted to a leaf $\tL$ of $\tcF^u$ indicates
the regularity of $\talpha$ as well as that of its associated 
path cocycle. Thus by parameterizing unstable
leaves the analysis is simplified to studying
the regularity of a map $f:\R\raw\R$. 

We shall also study the regularity of tafs 
and transverse cocycles using their restriction
to unstable leaves. To be definite, fix an arc $\Gamma$
as in Theorem~\ref{tafdist} which is a connected arc of an
unstable leaf  intersected with a Markov box. 
We parameterize $\Gamma$ by its $m^u$-arclength
and identify points with their parameterization in 
$[0,a]$, where $a = m^u(\Gamma)$. Given a taf $G$,
define the ``cumulative distribution function''
$H:[0,a]\raw \F$ via $H(x) = G([0,x])$. Using
the additivity of a taf, we can reconstruct $G$ 
from $H$ via the formula   $G([c,d]) 
= H(d) - H(c)$ for $0 \leq c \leq d \leq a$. A similar 
construction may be done for a transverse cocycle. Since after 
lifting to the orientation double cover taf's and path
cocycles correspond by Proposition~\ref{tadcocprop},
 and this lifting does not
change transverse structures for small intervals,
we see that the regularity of taf and cocycle are the
same. In addition, using the correspondence of
\cmaps\ and cocycles from \S\ref{cocyclecmap}, if we lift $\Gamma$ to
$\tGamma$ the universal Abelian cover $\tM$ and for
a path cocycle $F$ we  define
$\tH$ there analogously to $H$, we have that for
$\tx\in\tGamma$ (again identifying a point
in $\Gamma$ with its parameterization), 
$\tH(x) = \talpha(x) - \talpha(0)$.
Thus  taf,   transverse cocycles and
\cmaps\ all have the same local regularity properties.

There are certain facts about regularity 
which follow from what we have done thus far. Firstly,
we know from Theorem~\ref{tafdist}(b), that  the only taf with
bounded variation are  the
\PF\ ones.  Thus we would
expect that the corresponding non-\PF\ \cmaps\ would also
not be BV, which we show below in Theorem~\ref{cmapreg}.

We could also proceed more directly using
the standard results
from elementary measure theory which connect
cumulative distribution functions like $H$ on intervals
 to  measures. The result of relevance is that
  a set function $G$ defined on subarcs  extends to
a signed measure if and only if the corresponding 
cumulative distribution function $H$ has bounded variation.
Thus from the unique ergodicity of the stable foliation
we know immediately that \cmaps\ and tafs are not BV.

We shall reprove this result below as part of the
larger investigation into the properties of 
eigen-\cmaps. It is also worth noting that 
once we know  a taf on a small subarc in
any unstable leaf, we essentially know it everywhere.
This is a consequence of holonomy invariance coupled
with the fact that every leaf of $\cF^s$ is dense in 
the surface.

The point of view which motivates the next few
sections is that
the irregularity of eigen-\cmaps\ associated with
non-\PF\ eigenvalues  is the consequence of the 
scaling properties of the functions $f$ 
 which are  the \cmaps\ restricted to unstable leaves.
Specifically, in $\tM$ under the action of
$\tphi$ the parameterization of a
leaf $\tL$ by arc length transforms by multiplication by 
$\lambda$, the dilation of the \pA\ map. Thus an eigen-\cmap\
for eigenvalues $1 < |\mu|<\lambda$ yield functions $f$ with
the scaling $f(\lambda t) = \mu f(t)$ everywhere. 
This implies using Lemma~\ref{scaling}(c) that  the associated
eigen-\cmap\ is H\"older but nowhere differentiable
and not locally of bounded variation. Such scalings
are an example of a more general principle which roughly
states that
semi-conjugacies from higher entropy to lower entropy
systems must have fractal-like structure and thus the
resulting   low regularity. This is a basic
 principle which has many
applications so it is worth pursuing 
in the fairly well understood situation considered here. 

Many of the main ideas in this section were adapted from
Fathi \cite{fathi}.

\subsection{Steep functions}
As just described, the functions $\R\raw\R$ which are the restrictions of
an eigen-\cmap\ to an unstable leaf will have 
 the property that  the action of $\tphi$ rescales the parameterization
by the dilation while rescaling the image by the eigenvalue 
$\mu$, or $f(\lambda t) = \mu f(t)$. When $|\mu|<\lambda$
this will imply that $f$ has a property 
called steepness at $0$ and using the density of
leaves, $f$ will be steep everywhere. Steepness is
a  kind of ``anti-H\"older'' condition.

\begin{definition}[Steep function]
A continuous function $f:(a,b)\raw \C$ is said to be
$(C,\nu)$-steep from the right in 
the neighborhood $(c,d)$ at the point  $p\in (c,d)\subset (a,b)$  if
for all $p'\in (p,d)$,
\begin{equation}\label{steepdef}
\max\left\{|f(t) - f(p)| : 0 < t-p \leq p'-p\right\}  > C (p'-p)^\nu.
\end{equation}
The notion of $(C,\nu)$-steep from the left is defined
using the obvious alterations. The function is called
$(C,\nu)$-steep from both sides 
if it is steep from the left and right, and simply 
$(C,\nu)$-steep if is steep from the left or right for some
neighborhood.
\end{definition}

The proof of next lemma is routine and we omit it.
\begin{lemma}\label{convprop}
Assume that $f_n:(a,b)\raw\C$ and 
$f_n\raw f_0$ uniformly. 
\begin{compactenum}
\item If every $f_n$  is \cns\ at $p\in (a,b)$ for
the neighborhood $(c,d)$, then $f_0$ is also.
\item If $f_0$ is \cns\ at $p$ for the neighborhood $(c,d)$, 
then for each $\epsilon$ with
 $0 < \epsilon < C$ there is an $N$ so that $n>N$ implies that each $f_n$ is 
$(C-\epsilon, \nu)$-steep at $p$ for the neighborhood 
$(c-\epsilon,d-\epsilon)$.
\end{compactenum}
\end{lemma}

The next lemma shows that the scaling property
of $f$ implies steepness as well as giving some
of the consequences of being everywhere steep.
\begin{lemma}\label{scaling}
Assume that $f:\R\raw\C$ is continuous, and $\mu\in\C$, $\lambda\in\R$ 
satisfy $1 < |\mu| <\lambda$. Let 
$\nu = \log(|\mu|)/\log(\lambda)$.
\begin{compactenum}
\item If $f$ satisfies 
\begin{equation}\label{scaleeq}
f(\lambda t) = \mu f(t) \ \text{or} \ f(-\lambda t) = \mu f(t)
\end{equation} for all
$t\in\R$ and $f$ is not identically zero,
 then there exists a $C>0$ so that $f$ is \cns\ 
at zero in both directions for all neighborhoods of zero.
\item If $f$ is \cns\ at zero for the neighborhood $(c,d)$,
 then for all  $k\in\Z$  and $r\in\R$, 
\begin{equation}
\frac{f(\lambda^k t)}{\mu^k} + r \ \text{and} \ \ 
\frac{f(-\lambda^k t)}{\mu^k} + r,
\end{equation}
 are  \cns\ at zero for the neighborhoods
 $(\lambda^{-k} c,\lambda^{-k} d)$ and
 $(-\lambda^{-k} d, -\lambda^{-k} c)$, respectively.
\item If $f$ is \cns\ at every point $p\in (a,b)$,
then on $(a,b)$  $f$ is nowhere differentiable,  nowhere locally of
bounded variation, nowhere locally injective and 
 not H\"older for any exponent larger than $\nu$. Further, 
if $f$ is real-valued, 
there exists a dense, $G_\delta$-set $Z\subset f(a,b)$ so that
$z\in Z$ implies that $f^{-1}(z)$ is a Cantor set, and
if $f$ is complex-valued
there exists a dense, $G_\delta$-set $R$ contained in
the interval between the infimum and supremum
of $|f(t)|$ for $ t\in (a,b)$  so that $r\in R$ implies that
the image $f(a,b)$ intersects the circle $|z| = r$ in
a Cantor set.
\end{compactenum}
\end{lemma}
\begin{proof}
To prove (a), assume first that
$f(\lambda t) = \mu f(t)$ which 
implies that $f(0) = 0$. 
On $(0,\infty)$ we may  
define the  continuous function  $g(t) := |f(t)|/t^\nu$ which
then satisfies $g(\lambda t ) = g(t)$.
Since $f$ is not identically
zero, 
 $m := \max\{|g(t)| : t\in (1, \lambda]\} > 0$, and we let 
$t_0 \in (1, \lambda]$ be such that $|g(t_0)| = m$. Now
given $p'>0$, let $k$ be such that $\lambda^k t_0 
\leq p' < \lambda^{k+1} t_0$. Letting $C = m/|\mu|$, we
have $|f(\lambda^k t_0)| = |\mu|^k t_0^\nu m = C (\lambda^{k+1} t_0)^\nu
> C (p')^\nu$. This shows steepness to the right at zero;
the proof of steepness to  the left is similar.
If $f$ is complex-valued, apply the same argument to 
$|f(t)|$.

Now if we assume that $f(-\lambda t) = \mu f(t)$, then
$f(\lambda^2 t) = \mu^2 f(t)$, and since
$\nu = \log(\mu^2)/\log(\lambda^2)$, the result follows
by what we have just proved, completing (a).

The proof of (b) is  easy. To prove (c), since 
$0 < \nu < 1$, the definition \eqref{steepdef} implies that
at any point $p\in\R$, 
\begin{equation}
\limsup_{t\raw p} \frac{|f(t) - f(p)|}{|t-p|} >
\limsup_{t\raw p} C |t-p|^{\nu-1}= \infty,
\end{equation}
and so $f$ is nowhere differentiable. Further, 
\eqref{steepdef} also implies that
on any interval $[a,b]$ on which $f$ is \cns\ at every point,
 the variation of $f$ on $[a,b]$ satisfies
\begin{equation*}
Var(f;[a,b]) > C |b-a|^\nu.
\end{equation*}
Given any interval $[c,d]$ and an $n\in\N$,
 we define the subintervals
\begin{equation*}
I_{i,n} = [c + i(d-c)/n, c + (i+1)(d-c)/n],
\end{equation*}
for $i = 0, 1, \dots n-1$. The variation of $f$ on
$[c,d]$ satisfies
\begin{equation*}
Var(f;[c,d]) \geq \sum_{i=1}^n Var(f;I_{i,n}) 
\geq C n (\frac{d-c}{n})^\nu\raw \infty,
\end{equation*}
as $n\raw\infty$, showing that $f$ is nowhere 
locally BV.

Assume that $f$ is real-valued. 
Since $f$ is nowhere locally BV, it is certainly
nowhere locally constant and since we are in dimension
one this implies $f$ is light (i.e. point inverses
are totally disconnected). It is also
clearly nowhere locally injective, and so 
the assertion about the typical point inverse follows
from  a theorem of Blokh, 
et al \cite{oversteegen}.
The assertion about complex-valued $f$ follow
by applying the argument just given to $|f|$.

Finally, to prove the result about the
H\"older exponent, first note that
if $\nu' > \nu$  and for a given $[a,b]$ if the supremum in \eqref{steepdef}
is achieved at $t_0$, then 
\begin{equation*}
|f(t_0) - f(a)| > C |b-a|^\nu \geq C |t_0-a|^\nu >C |t_0-a|^{\nu'}.
\end{equation*}
Thus if $f$ were H\"older with exponent $\nu' > \nu$, the corresponding
constant must be $C' > C$. But then chose an interval $[c,d]$
with
\begin{equation*}
|d-c| < \left(\frac{C}{C'}\right)^{\frac{1}{\nu'-\nu}},
\end{equation*}  
and then if  the supremum on $[c,d]$ in \eqref{steepdef}
is achieved at $t_1$,
\begin{equation*}
|f(t_1) - f(c)| > C |d-c|^\nu > C' |t_1-c|^{\nu'},
\end{equation*}
a contradiction to the assumption  that $f$ is $(C',\nu')$-H\"older.
\QED\
\end{proof}

\subsection{Regularity of eigen-\cmaps\ for \pA\ \homeos}
We now apply the results of the previous section to
study the regularity of \cmaps\ restricted to leaves
of $\tcF^u$. 

Recall from \S\ref{leafinfo} that $\tcF^s$ and
$\tcF^u$ are the  lifts of $\phi$'s foliations 
to the universal Abelian $\tM$. The leaves containing
a point $\tx\in\tM$ are denoted $\tL^s(\tx)$ and 
$\tL^u(\tx)$. Leaves ``terminating'' in a    singularities
are called half-infinite. Further, the metric 
$\td_\phi$ is derived from the transverse measures.
In particular, arc length along an unstable leaf is
induced by the measure $m^u$.

Now fix once and for all an orientation on all
the leaves of $\tF^s$. There is no assumption that
these chosen orientations fit together in any
kind of coherent fashion. 
For $\tx$ that is not a singularity of the lifted
foliations,
if  $\tL^u(\tx)$ is a half-infinite leaf,
let $a(\tx)$ be the $m^u$-distance in $\tL^u(\tx)$ from
$\tx$ to the singularity, and if
 $\tL^u(\tx)$ is a regular leaf, 
 let $a(\tx) = \infty$.
Now define  $s_{\tx} : (-a(\tx), \infty)\raw \tL(\tx)$ 
as the parameterization  by   $m^u$-arclength  
  which agrees
with the chosen orientation on $\tL^u(\tx)$,
and further, $s_{\tx}(0) = \tx$.
Since arc length on unstable leaves is the  measure $m^u$, we have
\begin{equation}\label{actscale}
\tphi \circ \stx(t) = s_{\tphi(\tx)} (\epsilon \lambda t),
\end{equation}
for all $t\in\R$ and $\tx\in\tM$, with $\epsilon = 1$, if
$\tphi$ preserves the chosen orientations from
$\tL^u(\tx)$ to $\tL^u(\tphi(\tx))$, and 
$\epsilon = -1$, if $\tphi$ reverses these orientations.
Also note that 
\begin{equation}\label{equi2}
s_{\delta_{\vn} \tx} = \delta_{\vn} s_{\tx}, 
\end{equation}
and if $\tx$ and $\tx'$ are on the same leaf and
$\tx' = s_{\tx}(t_0)$, then  
\begin{equation}\label{equshift}
s_{\tx}(t) = s_{\tx'}(t - t_0 ).
\end{equation}

\begin{lemma}\label{regprelim}
 Assume $\phi:M\raw M$ is a \pA\ map and $\mu$ is an unstable eigenvalue
of $\phi^*$ on $H^1(M;\Z)$. Fix a lift $\tphi$ of $\phi$ to  $\tM$  and 
let $\talpha:\tM\raw\F$ be the eigen-\cmap\ given by Theorem~\ref{mainthm}.  
For each $\tx\in \tM$ which is not a singularity, 
let $f_{\tx}:(-a(\tx), \infty)\raw \F$ be given by 
$f_{\tx} = \talpha\circ s_{\tx}$ with $s_{\tx}$ defined
as above. Then
\begin{compactenum}
\item $f_{\delta_{\vn} \tx} =  f_{\tx} + \Phi(\vn)$ where
$\Phi:\Z^d\raw \F$ represents the eigen-cohomology class of $\mu$.
\item $f_{\tphi(\tx)} (t) = \mu f_{\tx} (\epsilon t/\lambda)$ with
$\epsilon = 1$, if
$\tphi$ preserves the chosen orientations from
$\tL^u(\tx)$ to $\tL^u(\tphi(\tx))$, and 
$\epsilon = -1$, if $\tphi$ reverses these orientations.
\item If $\tx_j$ are not singular points and $\tx_j \raw \tx$, 
then there are $\epsilon_j = \pm \id$ and a neighborhood of the
origin  $(c,d)$ so that
$f_{\tx_j}\circ \epsilon_j \raw f_{\tx}$ uniformly on $(c,d)$.
\end{compactenum}
\end{lemma}
\begin{proof}
Using \eqref{equi2} and fact that $\talpha$ is a \cmap,
\begin{equation*}
f_{\delta_{\vn} \tx}(t) = 
\talpha\circ s_{\delta_{\vn} \tx}
= \talpha\circ\delta_{\vn} s_{\tx} 
= \talpha\circ s_{\tx} + \Phi(n),
\end{equation*}
proving (a).
Since $\talpha$ is a semiconjugacy,
\eqref{actscale} yields
\begin{equation*}
f_{\tphi(\tx)}(t) 
= \talpha\circ s_{\tphi(\tx_)}(t)
= \talpha\circ\tphi\circ s_{\tx}(\epsilon t/\lambda)
= \mu \talpha\circ s_{\tx}(\epsilon t/\lambda),
\end{equation*}
proving (b). Part (c) follows from the fact
that as $\tx_j \raw \tx$ their corresponding leaves
converge smoothly.
\QED\
\end{proof}

If the eigenvalue 
$\mu$ is the \PF\ eigenvalue $\lambda$, then 
$\phi$ has oriented foliations and so the eigen-\cmap\
is essentially the \PF\ eigen-taf. Thus 
as a consequence of the correspondence of 
the \PF\ taf  and the transverse measure 
$m^u$ (\cf\ Remark~\ref{PFmeas}), when $\mu = \lambda$ 
each function $f$ is a translate of the identity 
or minus the identity. Thus we henceforth assume
that $\mu \not=\lambda$.

\begin{lemma}\label{aresteep}
Assume $\phi:M\raw M$ is a \pA\ map and $\mu$ is an unstable eigenvalue
of $\phi^*$ on $H^1(M;\Z)$ with $\mu$ not equal to the dilation $\lambda$. 
Fix a lift $\tphi$ of $\phi$ to  $\tM$  and 
let $\talpha:\tM\raw\F$ be the eigen-\cmap\ given by Theorem~\ref{mainthm}. 
Let $\nu = \log(|\mu|)/\log(\lambda)$ and for a nonsingular
point $\tx\in\tM$, let the function $f_{\tx}$ 
be defined as in Lemma~\ref{regprelim}. There exists
a $C>0$ so that 
the maps $f_{\tx}$ are \cns\ at every point in their domain.
\end{lemma}
 
\begin{proof} We first prove the result under the assumption
that $\phi$ has an interior fixed point which is not a singularity 
and is not on the boundary and satisfies one more property
given shortly.  The general case will then follow easily.

Let $p\in M$ be the assumed fixed point of $\phi$. Pick
a lift $\tp$  and
let $\tphi$ be the lift of $\phi$ to $\tM$ 
with  $\tphi(\tp) = \tp$. The additional assumption
is  that $\tphi$ 
preserves the chosen orientation on $\tL^u(\tx)$.
From Lemma~\ref{regprelim}(b) we get 
$f_{\tp}(\lambda t) = \mu f_{\tp}(t)$, 
and so by Lemma~\ref{scaling}(a), $f_{\tp}$ is \cns\ at 
zero.

Let $y\in M$ be such that its forward $\phi$-orbit is
dense in $M$ and its unstable leaf $L^u(y)$ 
 is not associated with any singularity.
Fix a lift $\ty$ of $y$.
 There exist $k_j\raw\infty$ and $\vn_j\in\Z^d$ with
$\ty_j :=\delta_{n_j}\circ \tphi^{k_j} (\ty) \raw \tp$, and so by
Lemma~\ref{regprelim}(c), there are $\epsilon_j$ so that
$f_{\ty_j}\circ \epsilon_j \raw f_{\tp}$ uniformly in some neighborhood of
zero. Thus by 
Lemma~\ref{regprelim}(a)(b), 
\begin{equation*}
f_j(t) := \mu^{k_j} f_{\ty}
 \left(\frac{\epsilon_j(t)}{\lambda^{k_j}}\right) + \Phi(n_j) 
\raw f_{\tp},
\end{equation*}
uniformly in some neighborhood of
zero, and so by Lemma~\ref{convprop}(b), 
for all $m\in\N$ there
exists a $J_m$, so that $j>J_m$ implies that $f_j$ is $(C-(1/m), \nu)$-steep
at $0$. Thus by Lemma~\ref{convprop}(a), $f_{\ty}$ is $(C-(1/m), \nu)$-steep
at zero for all $m$ and thus is \cns\ at zero.

Now for any $\tx$, we also have $\ell_j\raw\infty$ and 
$\vm_j\in\Z^d$ with
$\delta_{m_j}\circ \tphi^{\ell_j}(\ty)\raw \tx$, and so by the same argument
as the previous paragraph, $f_{\tx}$ is \cns\ at zero. 
To get that $f_{\tx}$ is also \cns\ at other
points in its domain, simply observe that by \eqref{equshift},
$f_{\tx}(t) = f_{\tx'}(0)$ for $\tx'= s_{\tx}(t)$, finishing
the proof under the assumption that $\phi$ has
a nonsingular, interior fixed point.

The proof for the case when $\phi$ does not have
an interior, nonsingular fixed point follows easily
from the following observations. Since the periodic
points of $\phi$ are dense in $M$, we can certainly
find an interior point $p$ and an $M>1$ with
$\phi^M(p) = p$. We then pick a lift $\tp$, let $N = 2M$,
and note there is a  $\vn\in\Z^d$ so
that $\delta_{\vn}\; \tphi^N(\tp) = \tp$. Further, since
$N$ is even, 
$\delta_{\vn}\; \tphi^N$ preserves the orientation
on $\tL^u(\tp)$.
If $\talpha$ is the \cmap\ for $\tphi$ with factor
$\mu$, then by definition, $\talpha \tphi = \mu \talpha$.
Letting $\talpha' := \talpha + \Phi_\mu(\vn)/(\mu^N -1)$ and 
$\tpsi = \delta_{\vn}\; \tphi^N$,
we have that  $\talpha' \tpsi = \mu^N \talpha'$,
and so by the uniqueness of the eigen-\cmaps\ given
in Theorem~\ref{mainthm}, we know that $\talpha'$ is the eigen-\cmap\
for $\tpsi$ with factor $\mu^N$. We now apply the above
arguments to functions $f^\prime_{\tx}$ defined
using $\talpha'$. First noticing that
$\nu =\log(|\mu|)/\log(\lambda) =
\log(|\mu^N|)/\log(\lambda^N)$, then noting that  $\talpha'$ 
differs from $\talpha$ by at most  a constant,  
we have the desired results for
the original functions $f_{\tx}$ defined
using $\talpha$.
\QED\
\end{proof}

Using Lemma~\ref{aresteep} and Lemma~\ref{scaling} we get
\begin{theorem}\label{cmapreg}
Assume $\phi:M\raw M$ is a \pA\ map and $\mu$ is an  eigenvalue
of $\phi^*$ on $H^1(M;\Z)$ with $1 < |\mu| < \lambda$, where
$\lambda$ is the dilation of $\phi$.
 Fix a lift $\tphi$ of $\phi$ to  $\tM$  and 
let $\talpha:\tM\raw\F$ be the eigen-\cmap\ given by Theorem~\ref{mainthm}, 
 and let $\nu = \log(|\mu|)/\log(\lambda)$. 

 The \cmap\ $\talpha$ restricted to each leaf of the
lifted unstable foliation $\tcF^u$ is
$\nu$-H\"older,  
nowhere differentiable,  nowhere locally of
bounded variation, and nowhere locally injective. Further,
$\talpha$ is not H\"older for any exponent larger than $\nu$.
If $\mu$ is real, the generic point inverse of
$\talpha$ restricted to an unstable leaf is a Cantor set.
If $\mu$ is complex, then for a dense, $G_\delta$-set
of $r$ values in $[0, \infty)$, the image in $\C$
 of $\talpha$ restricted to an unstable leaf
intersects  the circle $|z| = r$ in a Cantor set.
\end{theorem}

\subsection{Example}\label{eviltwin}
We give  an example due to Gavin Band which illustrates 
the main results. Let
$M$ be a closed, genus-two surface and $\psi:M\raw M$ 
 a \pA\ map such that the characteristic polynomial
of $\psi_*$ acting on $H_1(M,\Z)$ splits over the integers
into a pair of irreducible quadratic factors. Assume that
the roots of the first factor are $\lambda_1$ and $\eta_1$ and
those of the second are $\lambda_2$ and $\eta_2$. Note that
of necessity $\eta_i = \lambda^{-1}_i$.  Further, we assume that
all roots are real, $\lambda_1 > \lambda_2 > 1$, and 
$\lambda_1$ is the dilation of $\psi$. These
conditions  imply
that $\psi$ has orientable foliations. It is easy to
build   examples of this type using Rauzy induction.

Using Theorem~\ref{mainthm} with $\psi$ and $\psi\I$ we get
four semi-conjugacies which we denote
$\talpha_{\lambda_i}$ and $\talpha_{\eta_i}$ for
$i = 1,2$. 
As a consequence  of a theorem of Fathi \cite{fathi}, 
for both $i=1$ and $i = 2$
the paired semi-conjugacies 
$(\talpha_{\lambda_i}, \talpha_{\eta_i}):
\tM\raw \R^2$ descend to semiconjugacies
$(M,\psi)\raw (\T^2, \Phi_i)$, where 
$\Phi_i$ is a linear toral automomorphims with eigenvalues
$\lambda_i$ and ${\eta_i}$.  We call these semiconjugacies
$\beta_1$ and $\beta_2$. Note that both semi-conjugacies
of necessity take leaves of $\cF^u$ and $\cF^s$ to leaves
of the unstable and stable foliations of the toral automorphisms. 

The existence of $\beta_1$ also follows from 
Franks and Rykken \cite{franksrykken}, who show that
$\beta_1$ is always a 
branched cover with branch points and their
images   singularities. Thus $\beta_1$ is
smooth at all but finitely many points
and the preimages ${\beta_1}\I(x)$ are finite sets
with a uniformly bounded cardinality.

On the other hand, as a consequence of
Theorem~\ref{cmapreg}, $\beta_2$ is nowhere
differentiable, is nowhere locally BV, is
$\nu$-H\"older for $\nu = \log(|\lambda_2|)/\log(\lambda_1)$,
not  H\"older for exponents greater than $\nu$, and for
a dense, $G_\delta$-set $Z\subset \T^2$, $z\in Z$ implies
that ${\beta_2}\I(z)$ is a Cantor set. 

This example is reminiscent of a theorem from symbolic
dynamics which gives a dichotomy for 
 the semiconjugacy $h$ between two
transitive subshifts of finite type.
The theorem says that the shifts have the same entropy if and only if $h$ 
is bounded to one and have different entropy if and
only if the generic point inverse of $h$ is a Cantor set.
This theorem is an easy extension of results in
Chapter 4 of  \cite{kitchens} and could be applied with
some work to the example using the symbolic models
of $\psi$, $\Phi_1$ and $\Phi_2$.   

\bibliography{eigen}

\end{document}